\documentclass{article}

\usepackage{arxiv}

\usepackage{xcolor}
\usepackage[utf8]{inputenc} 
\usepackage[T1]{fontenc}    
\usepackage[hidelinks]{hyperref}       
\usepackage{bm}             
\usepackage{url}            
\usepackage{booktabs}       
\usepackage{microtype}      
\usepackage{natbib}
\usepackage{amsmath, amsfonts, amssymb, amsthm}
\usepackage{shortcuts}

\usepackage{float}

\title{Gaussian Prepivoting for Finite Population Causal Inference}

\author{
    Peter L. Cohen \\
  Operations Research Center\\
  Massachusetts Institute of Technology\\
  1 Amherst Street\\
  Cambridge, Massachusetts 02142, U.S.A\\
  \texttt{plcohen@mit.edu} \\
  \And
    Colin B. Fogarty \\
  Operations Research and Statistics Group\\
  Massachusetts Institute of Technology\\
  100 Main Street\\
  Cambridge, Massachusetts 02142,  U.S.A\\
  \texttt{cfogarty@mit.edu}
}

\begin{document}

\maketitle

\begin{abstract}
	In finite population causal inference exact randomization tests can be constructed for sharp null hypotheses, i.e. hypotheses which fully impute the missing potential outcomes. Oftentimes inference is instead desired for the weak null that the sample average of the treatment effects takes on a particular value while leaving the subject-specific treatment effects unspecified. Without proper care, tests valid for sharp null hypotheses may be anti-conservative should only the weak null hold, creating the risk of misinterpretation when randomization tests are deployed in practice. We develop a general framework for unifying modes of inference for sharp and weak nulls, wherein a single procedure simultaneously delivers exact inference for sharp nulls and asymptotically valid inference for weak nulls. To do this, we employ randomization tests based upon prepivoted test statistics, wherein a test statistic is first transformed by a suitably constructed cumulative distribution function and its randomization distribution assuming the sharp null is then enumerated.  For a large class of commonly employed test statistics, we show that prepivoting may be accomplished by employing the push-forward of a sample-based Gaussian measure based upon a suitably constructed covariance estimator. In essence, the approach enumerates the randomization distribution (assuming the sharp null) of a P-value for a large-sample test known to be valid under the weak null, and uses the resulting randomization distribution to perform inference. The versatility of the method is demonstrated through a host of examples, including rerandomized designs and regression-adjusted estimators in completely randomized designs.

\end{abstract}

\section{Introduction}




In finite population causal inference two distinct hypotheses of ``no treatment effect'' are commonly tested: Fisher's sharp null and Neyman's weak null.  Fisher's sharp null of no effect refers to the null that the responses under treatment and under control are the same for \textit{all} individuals in the study \citep{ObsStudies}.  The sharp null imputes the missing values of the potential outcomes for all individuals, in so doing facilitating the use of randomization tests to provide exact inference with randomization alone acting as the basis for inference \citep{fis35}.  Neyman's weak null instead specifies only that the average of the treatment effects for the individuals in the experiment equals zero, while allowing for heterogeneity in the unit-specific effects. The missing potential outcomes are no longer imputed under the weak null, such that the randomization distribution under the weak null remains unknown. Consequently, inference for the weak null has historically proceeded using asymptotically conservative analytical approximations to the limiting distribution of the treated-minus-control difference in means.

While the exactness attained under the sharp null is appealing, randomization tests have been criticized for the seemingly restricted nature of the conclusions to which the researcher is entitled should the null be rejected \citep{cau17}. While the researcher may suggest that the treatment effect is not zero for all individuals, generally one is not entitled to a statement of whether the treatment effect is positive or negative on average for the individuals in the study. To address this, a recent literature has emerged on how randomization tests may be modified to maintain asymptotic validity under the weak null hypothesis.  The resulting methods provide a \textit{single} testing procedure that is asymptotically valid for the weak null hypothesis, while maintaining exactness should the sharp null also be true \citep{din17, loh17, din17ls, RandTestsWeakNulls, fog19student}.

The existing literature attains this unified mode of inference largely on a case-by-case basis: for a given experimental design, a specially catered test statistic is constructed so that the corresponding randomization test under the sharp null maintains asymptotic validity under the weak null. In this work, we provide both general conditions under which the unification may be achieved and a general methodology for attaining it. The central idea is to leverage prepivoting, an idea introduced in \citet{PrePivIntroduction, ber88}. For most commonly employed experimental designs and test statistics, the reference distribution generated by the prepivoted statistic under the assumption of the sharp null asymptotically stochastically dominates the true, but unknowable, randomization distribution under the weak null, yielding asymptotically conservative inference for the weak null while maintaining exactness under the sharp null hypothesis. As we demonstrate, prepivoting succeeds in many scenarios where other common resolutions such as studentization prove inadequate.

At a high level, prepivoting takes a test statistic $T_0$ and composes it with a cumulative distribution function $\hat{F}$ constructed from the observed data, forming the new test statistic $T_1 = \hat{F}(T_0)$. If $\hat{F}$ were a consistent estimate of $T_0$'s limit distribution, $\hat{F}(T_0)$ would, through an asymptotic application of the probability integral transform, tend to a standard uniform. Under the weak null hypothesis, the true distribution function for common test statistics $T_0$ cannot generally be consistently estimated. Fortunately, as developed in \S \ref{sec: parametric prepiv} a distribution function for a random variable that asymptotically stochastically dominates $T_0$ may be constructed. For most common test statistics for the weak null hypothesis, under conditions outlined in \S \ref{sec: parametric prepiv} this dominating distribution function amounts to a suitable pushforward of a multivariate Gaussian measure constructed using a conservative covariance estimator. Using this estimated distribution function, $T_1 = \hat{F}(T_0)$ is instead stochastically dominated by a standard uniform in the limit. Observe that through this construction, the prepivoted test $T_1 = \hat{F}(T_0)$ is precisely one minus the large sample $p$-value for the test statistic $T_0$ leveraging the central limit theorem. Rather than using this $p$-value to reach a conclusion by comparing its value to the desired $\alpha$, we instead use the reference distribution of this large-sample $p$-value enumerated over all possible randomizations assuming the sharp null holds. This reference distribution generally converges pointwise to the standard uniform distribution function for commonly used covariance estimators. As a result, inference is guaranteed to be asymptotically conservative under the weak null while maintaining exactness under the sharp null. The general takeaway is that rather than looking at the randomization distribution of a test statistic itself under the sharp null, one should instead enumerate the randomization distribution of one minus an asymptotically valid $p$-value to restore validity of Fisher randomization tests when only the weak null holds.

In \S \ref{sec: notation} we introduce notation for finite population causal inference and detail some standard assumptions.  Section \ref{sec: rand and perm inference} defines the reference distribution assuming the truth of Fisher's sharp null and juxtaposes it with its true though unknowable randomization distribution under Neyman's weak null of no effect on average.  After an overview of useful asymptotic results on completely randomized designs in \S\ref{sec: useful}, \S\ref{sec: parametric prepiv} introduces Gaussian prepivoting in the context of suitably constructed functions of treated-minus control difference in means. Section \ref{sec: illustration} provides examples of and insight into prepivoting using Gaussian measure. Section \ref{sec: linear} extends these results to other asymptotically linear estimators including regression-adjusted estimators, while \S \ref{sec: simulations} provides simulation studies highlighting the benefits of Gaussian prepivoting.

\section{Notation and Review}\label{sec: notation}
\subsection{Notation for finite population causal inference}
While the developments in this work apply quite generally across common experimental designs and with two or more levels of the treatment, in this work we focus on completely randomized experiments and rerandomized experiments with two treatments; see the appendix for extensions to paired designs and to completely randomized designs with multi-valued treatments. Consider a collection of $N$ individuals, where $n_{1}$ receive treatment and $n_{0} = N - n_{1}$ receive the control.  For the $i$th individual, the random variable $Z_{i}$ is the treatment indicator, taking the value $1$ if the $i$th individual receives treatment and $0$ otherwise.  We assume that the stable unit treatment value assumption holds, such that there is no interference and that there are no hidden levels of the treatment \citep{rub80}. The $i$th individual has two deterministic potential outcomes: $\mathbf{y}_{i}(1)$, the $d$-dimensional outcome under treatment, and $\mathbf{y}_{i}(0)$ the $d$-dimensional outcome under control.  Furthermore, the $i$th unit has deterministic covariates $\mathbf{x}_{i} \in \R^{k}$.  The $j$th coordinate of $\bfy_{i}(z)$ is $y_{ij}(z)$, and the analogous statement holds for $x_{ij}$.  The random vector $\bfZ$ represents $(Z_{1}, \ldots, Z_{N})^{\T}$; likewise $\Ytreat = (\mathbf{y}_{1}(1), \ldots, \mathbf{y}_{N}(1))^{\T}$ and $\Ycontrol = (\mathbf{y}_{1}(0), \ldots, \mathbf{y}_{N}(0))^{\T}$.  Under the finite population model the potential outcomes are viewed as fixed across randomizations, and the only randomness enters through $\bfZ$, the treatment allocation.  For a discussion of the finite population inference framework, we suggest \citet{FiniteSuperpop}.  The observed outcome-vector for individual $i$ is $\mathbf{y}_{i}(Z_{i})$ and the collection of these is denoted $\yobs$. 
\textcolor{black}{Causal inference with multiple outcomes is becoming increasingly common in modern applications ranging from drug repurposing studies to A/B tests assessing the impact of competing web page designs on various user engagement metrics. See \citet{psychExample} and \citet{cardiologyExample} for concrete examples of causal inference with multiple endpoints in biomedical sciences, and see \citet{decompTreatmentEffectVar} for a reference on the underlying mathematics of multivariate potential outcome models.} 

The vector of treatment effects for the $i$th individual is $\bm{\tau}_{i} = \mathbf{y}_{i}(1) - \mathbf{y}_{i}(0)$.  The average treatment effect for the individuals in the experiment is $\bar{\bm{\tau}} = N^{-1}\sum_{i = 1}^{N}\bm{\tau}_{i}$.  As the two potential outcomes are not jointly observable, $\bm{\tau}_i$ is unknown for all individuals. Neyman's weak null of no treatment effect on average is $H_{N} : \bar{\bm{\tau}} = \bm{0}$,  while Fisher's sharp null further stipulates $H_{F} : \bm{\tau}_{i} = \bm{0}\;\;(i = 1, \ldots, N)$  \textcolor{black}{such that the treatment made no difference among any of the $d$ outcomes measured.  We implicitly define the alternative hypothesis as that which complements the null, so for $H_{N}$ the alternative is $H_{A}:\overline{\bm{\tau}} \neq \mathbf{0}$ and for $H_{F}$ the alternative is $H_{A}: \, \exists\, i \text{ s.t. } \bm{\tau}_{i} \neq \mathbf{0}$.  Consequently, our tests are non-directional; this differs from the one-sided bounded alternatives tested by \cite{cau17}.  Furthermore, the one-sided bounded alternatives of \cite{cau17} bound each individual's treatment effect, whereas we are interested in unifying inference for both individual effects and aggregate effects.}

For any matrix $\mathbf{r} \in \mathbb{R}^{N\times d}$ and any binary vector $\mathbf{W}$ with $\sum_{i=1}^NW_i = n_1$, we define the function
\begin{align*}
    \hat{\tau}(\mathbf{r}, \mathbf{W}) &= \frac{1}{n_1}\sum_{i=1}^NW_i\mathbf{r}_i - \frac{1}{n_0}\sum_{i=1}^N(1-W_i)\mathbf{r}_i.
\end{align*}
Using this notation, the observed treated-minus-control difference in means for the outcome variables is $\hat{\tau}(\yobs, \mathbf{Z})$ and is often denoted by $\hat{\bm{\tau}}$ as shorthand. In general, ``hats'' are used to denote functions of observed quantities whose limiting properties will eventually be studied herein. Define $\bar{\mathbf{y}}(0) = N^{-1}\sum_{i = 1}^{N}\mathbf{y}_{i}(0)$ and $\bar{\mathbf{y}}(1) = N^{-1}\sum_{i = 1}^{N}\mathbf{y}_{i}(1)$ to be the average potential outcomes for the $N$ individuals in the study population.  Likewise, we define the covariance matrices
\begin{align*}
	\Sigma_{y(z)} & = (N - 1)^{-1}\sum_{i = 1}^{N}(\mathbf{y}_{i}(z) - \bar{\mathbf{y}}(z))(\mathbf{y}_{i}(z) - \bar{\mathbf{y}}(z))^{\T},\; z \in \{0, 1\};\\
	\Sigma_{\tau} & = (N - 1)^{-1}\sum_{i = 1}^{N}(\bm{\tau}_{i} - \bar{\bm{\tau}})(\bm{\tau}_{i} - \bar{\bm{\tau}})^{\T}.
\end{align*}

To emphasize the distinction between functions of observed outcomes and functions of covariates, we define the function $\hat{\mathbf{\delta}}(\mathbf{x}, \mathbf{W})$ with binary $\mathbf{W}$ such that $\sum_{i=1}^NW_i = n_1$ as
\begin{align*}
    \hat{\delta}(\mathbf{x}, \mathbf{W}) &= \frac{1}{n_1}\sum_{i=1}^NW_i\mathbf{x}_i - \frac{1}{n_0}\sum_{i=1}^N(1-W_i)\mathbf{x}_i.
\end{align*}
\textcolor{black}{The function $\hat{\mathbf{\delta}}(\mathbf{x}, \mathbf{W})$ is a special case of $\hat{\tau}(\mathbf{r}, \mathbf{W})$.  } The observed difference in means for covariates is $\hat{{\delta}}(\mathbf{x}, \mathbf{Z})$, abbreviated as $\hat{\bm{\delta}}$. The finite population mean of the covariates is $\bar{\mathbf{x}} = N^{-1}\sum_{i = 1}^{N}\mathbf{x}_{i}$. The finite population covariance matrix for the covariates is $\Sigma_{x}$, defined by simply replacing $\mathbf{y}_i(z)$ with $\mathbf{x}_i$ and $\bar{\mathbf{y}}(z)$ with $\bar{\mathbf{x}}$ in the definition of $\Sigma_{y(z)}$. The finite population covariance between potential outcomes and covariates is $\Sigma_{y(z)x}$ for $z = 0, 1$, and the covariance between treatment effects and covariates is $\Sigma_{\tau x}$.
Asymptotic arguments that follow will imagine a single sequence of finite populations of increasing size, with $N\rightarrow\infty$. As a result, quantities such as $\Sigma_{\tau}$ themselves vary as $N\rightarrow\infty$ and should be denoted by $\Sigma_{\tau, N}$ to reflect this. Generally the dependence is suppressed to reduce notational clutter; however, we do employ the notation $\Sigma_{\tau, \infty}$ to denote the limiting value of $\Sigma_{\tau,N}$ as $N\rightarrow \infty$, and likewise for other finite population quantities. For more on the finite population model for causal inference, see \citet{imbensRubin15} and \citet{FiniteSuperpop} among many.

\subsection{Rerandomized designs and balance criterion}
The set of all possible treatment assignments $\bfZ$ is denoted by $\Omega$, and is determined by the experimental design. In completely randomized experiments, covariates are not used to inform the chosen treatment assignment and $\Omega_{CRE} = \{\bfz: \sum_{i=1}^Nz_i = n_1\}$. To mitigate the risk of significant covariate imbalance, \citet{reRandMorganRubin} suggest instead building covariate balance into the treatment allocation process through rerandomization.  The study is conducted by collecting covariate data for the study participants, determining a measure of imbalance and a threshold for deciding what imbalances are acceptable, and selecting a treatment allocation uniformly over the set of allocations satisfying the balance criterion \citep{asymptoticsOfRerand}. Stringent balance criterion reduce the cardinality of $\Omega$ by eliminating undesirable assignments, with the hopes of improving precision as a consequence.  Naturally, randomization inference must take into account that the allowable realizations of $\bfZ$ depend upon the condition that covariate balance is met.

A \textit{balance criterion} is a Boolean-valued function $\phi(\cdot)$, where  $\phi(\sqrt{N}\hat{\bm{\delta}}) = 1$ is taken to mean that the treatment allocation $\bfZ$ which results in the particular realization of $\hat{\bm{\delta}}$ under consideration satisfies appropriate covariate balance.  We impose the following restriction on $\phi$:

\begin{condition} \label{cond: phi}$\phi: \mathbb{R}^k \mapsto \{0,1\}$ is an indicator function such that the set $M = \{\mathbf{b}: \phi(\mathbf{b}) = 1\}$ is closed, convex, mirror-symmetric about the origin (i.e. $\mathbf{b} \in M \Leftrightarrow -\mathbf{b} \in M$) with non-empty interior.\end{condition}


\subsection{Regularity conditions}
We make the following assumptions about the structure of the finite populations and experimental designs as $N$ goes to infinity. These assumptions are for the most part standard in the literature; see, for instance, \citet{RandTestsWeakNulls}.
\begin{assumption}\label{asm: non-degen sampling limit}
	The proportion $n_{1} / N$ limits to $p \in (0, 1)$ as $N \rightarrow \infty$.  
\end{assumption}

\begin{assumption}\label{asm: means and covs stabilize}
	All finite population means and covariances having limiting values for both the potential outcomes and the covariates. For instance, $\lim_{N \rightarrow \infty}\bar{\mathbf{y}}(z) = \bar{\mathbf{y}}_{\infty}(z)$ for $z \in \{0, 1\}$ and $\lim_{N \rightarrow \infty}\Sigma_{y(1)} = \Sigma_{y(1),\infty}$.
\end{assumption}


\begin{assumption}\label{asm: bounded fourth moment}
 There exists some $C < \infty$ for which, for all $z\in\{0,1\}$, all $j=1,..,d$ and all $N$,
	\begin{equation*}
		\frac{\sum_{i = 1}^{N}\left(y_{ij}(z) - \overline{y}_j(z)\right)^{4}}{N} < C,
	\end{equation*}
	where $\overline{y}_j(z)$ denotes the $j$th coordinate of $\overline{\bfy}(z)$. Further, the above holds for the covariates with ${x}_{ij}$ replacing ${y}_{ij}(z)$ above for $j=1,..,k$.
\end{assumption}

Assumption~\ref{asm: bounded fourth moment} is used to obtain finite population-inference strong laws of large numbers for mean and variance estimators. Such an assumption is made at times for mathematical convenience to simplify the analysis of certain random distributions and may hold under weaker assumptions.  Assumption~\ref{asm: bounded fourth moment} is commonplace in the literature on finite population causal inference; see, for instance, \citet{RandTestsWeakNulls, agnosticRegAdj, freedmanRegAdj2, freedmanRegAdj1}.

\subsection{A technical note on the convergence of random measures}

A random sequence of probability measures $\hat{\mu}_{N}$ on $S$ converges weakly in probability to a deterministic probability measure $\mu$ if $\int_{S} f\,d\hat{\mu}_{n} \convP[] \int_{S} f\,d\mu$ for all continuous bounded functions $f: S \rightarrow \R$ \citep[Section 2]{lowDim_HighDim}.  Aspects of the Portmanteau Theorem \citep[Theorem 1.3.4]{vanDerVaar_Wellner} extend to weak convergence in probability of random measures \citep{lowDim_HighDim, randomMeasures}.  Most importantly for our purposes is that if $\{\hat{F}_{N}\}$ are random cumulative distribution functions and $F$ is a fixed cumulative distribution function, then their associated measures converge weakly in probability if and only if $\hat{F}_{N}(t) \convP[] F(t)$ for all $t$ which are continuity points of $F$; we take this as the definition of weak convergence in probability for random cumulative distribution functions.

\color{black}

\section{Randomization distributions and tests}\label{sec: rand and perm inference}
\subsection{Randomization distributions}
Consider a scalar test statistic $T(\mathbf{y}(\mathbf{Z}), \mathbf{Z})$, a function of the observed responses and the treatment assignment received. The \textit{randomization distribution} for the test statistic $T$ is
\begin{equation}\label{eqn: general rand dist}
	\Rand_{T}(t) = \frac{1}{|\Omega|}\sum_{\mathbf{w} \in \Omega}\indicatorFunction{T(\mathbf{y}(\mathbf{w}), \mathbf{w}) \leq t}.
\end{equation}
$\Rand_{T}$ is the true cumulative distribution function of $T(\yobs, \mathbf{Z})$ with respect to the randomness in the treatment allocation $\bfZ$, distributed uniformly over $\Omega$. If we had access to $\Rand_T$ under the null hypothesis in question, we could make direct use of it to provide inference that is exact in finite samples, proceeding without dependence on asymptotics. Under Fisher's sharp null hypothesis, $\Rand_T$ \textit{is} specified by the observed outcomes as $\mathbf{y}(\mathbf{Z}) = \mathbf{y}(\mathbf{w})$ for any $\mathbf{w}\in \Omega$. Unfortunately, the distribution is generally unknown under the weak null, as the weak null merely constrains the missing potential outcomes without determining them.



\subsection{Randomizaton tests assuming the sharp null}\label{sec: perm dists}
In practice an experimenter draws a single realization of $\bfZ$, in so doing only revealing the values of the potential outcomes corresponding to the observed assignment. Suppose that regardless of whether or not Fisher's sharp null hypothesis actually holds, the researcher considers use of the randomization distribution to which she or he would be entitled if the sharp null were true. This \textcolor{black}{reference} distribution takes the form
\begin{equation}\label{eqn: general perm dist}
	\hatPerm_{T}(t) = \frac{1}{|\Omega|}\sum_{\mathbf{w} \in \Omega}\indicatorFunction{T\left(\yobs, \mathbf{w}\right) \leq t}.
\end{equation}
 While $\Rand_T = \hatPerm_T$ under the sharp null, under the weak null $\hatPerm_{T}$ is a random distribution function as it varies with $\mathbf{Z}$. Inference using $\hatPerm_T$ proceeds as though $\mathbf{y}(\mathbf{Z})$ would have been the observed response for any $\mathbf{w}\in \Omega$.  As the true response $\mathbf{y}(\mathbf{w})$ under assignment $\mathbf{w}$ need not align with $\yobs$, $\hatPerm_T$ does not actually reflect the true randomization distribution under the weak null. This gives rise to potentially anti-conservative inference should $\hatPerm_T$ be used to test the weak null hypothesis.


For $\alpha \in (0, 1)$ define the Fisher randomization test of nominal level $\alpha$ by
\begin{equation}\label{eqn: permutation test general}
	\varphi_{ T}(\alpha) = \indicatorFunction{T(\yobs, \mathbf{Z}) \geq \hatPerm_{T}^{-1}(1 - \alpha)}.
\end{equation}

Under the sharp null, $\mathbb{E}\{\varphi_T(\alpha)\} \leq \alpha$ for any sample size as $\hatPerm_{T} = \Rand_{T}$. Throughout this paper, we examine the extent to which certain choices of test statistics entitle us to asymptotic Type I error control at $\alpha$ when $\varphi_T(\alpha)$ is used to conduct inference but only the weak null holds. For a given test statistic $T$, we will often proceed by juxtaposing its true limiting behavior under the randomization distribution $\Rand_T$ with the limiting behavior of $\hatPerm_T$, the randomization distribution if we (incorrectly) assumed that the sharp null held.

\subsection{Towards a unified mode of inference}\label{subsec: unified mode of inference}
Suppose that for a test statistic $T(\yobs, \bfZ)$ based upon the observed outcomes $\yobs$ and the treatment allocation $\bfZ$,
\begin{itemize}
    \item[(a)] $\hatPerm_{T}$ converges weakly in probability to a fixed distribution $\Perm_{T,\infty}$ as $N\rightarrow \infty$; and
    \item[(b)] $\Rand_{T}$ converges pointwise to a fixed distribution $\Rand_{T,\infty}$ at all continuity points of $\Rand_{T,\infty}$. \textcolor{black}{Formally,
       $\Rand_{T}(t) \rightarrow \Rand_{T, \infty}(t) \quad \forall\, t \in \text{cont}(\Rand_{T, \infty})$,
    where $\text{cont}(\Rand_{T, \infty})$ is the set of continuity points of $\Rand_{T, \infty}$.}
\end{itemize}
The test statistic $T(\yobs, \bfZ)$  is called \textit{asymptotically sharp-dominant} if, under $H_{N}$,  $\Perm_{T,\infty}(t)\leq \Rand_{T,\infty}(t)$ for any scalar $t$. This implies that the $(1 - \alpha)$ quantile of $\Perm_{T,\infty}$ is at or above the $(1 - \alpha)$ quantile of $\Rand_{T,\infty}$.  If $T(\yobs, \bfZ)$ is asymptotically sharp-dominant, then inference based upon the reference distribution $\hatPerm_T$ will be asymptotically conservative even if only $H_N$ holds  \citep[Proposition 4]{RandTestsWeakNulls}, satisfying $\limsup \mathbb{E}\{\varphi_T(\alpha)\} \leq \alpha$ as $N\rightarrow \infty$ while maintaining exactness should the sharp null be true.

Many common test statistics are not asymptotically sharp-dominant over all elements of the weak null. For instance, with univariate potential outcomes and under a completely randomized design with imbalanced treated and control groups, the absolute difference in means $T(\yobs, \bfZ) = \sqrt{N}|\tauhat|$ is not generally asymptotically sharp-dominant as there exist sequences of potential outcomes satisfying the weak null such that $\liminf \mathbb{E}\{\varphi_T(\alpha)\} > \alpha$; see \citet{din17}, \citet[Cor. 3]{RandTestsWeakNulls}, or \citet{loh17} for details.  For this test statistic,  simply studentizing by the usual standard error estimator ensures sharp dominance.  However, studentization fails to generalize to other more complicated test statistics and complex experimental designs.  Significant efforts have recovered appropriate studentization techniques for some test statistics, but each test statistic requires its own separate analysis \citep{RandTestsWeakNulls}.  For some experimental designs, studentizing the difference in means is not sufficient to regain asymptotically valid inference even in the univariate case; we explore this topic in \S \ref{sec: examples} and \S\ref{subsec: studentization in rerandomization} in the context of rerandomization. In \S\ref{sec: parametric prepiv}, we present a general method called Gaussian prepivoting which both recovers studentization when it alone would be sufficient, but also yields asymptotic sharp-dominance in circumstances where studentization would be insufficient. Before describing the method, we recall a few important results on the difference in means in completely randomized designs which underpin the success of Gaussian prepivoting.


\section{Useful results for the difference-in-means in completely randomized designs}\label{sec: useful}
\subsection{Asymptotic normality and conservative covariance estimation for the randomization distribution}
Consider the distribution of $\sqrt{N}(\hat{\bm{\tau}} - \bar{\bm{\tau}}, \hat{\bm{\delta}})^\T$ in a completely randomized design. Under Assumptions \SLLNAssumptions, a finite population central limit theorem applies \citep{FiniteCLT}, and $\sqrt{N}(\hat{\bm{\tau}} - \bar{\bm{\tau}}, \hat{\bm{\delta}})^\T$ converges in distribution to a mean-zero multivariate Gaussian with covariance matrix $V$ of the form
	\begin{align*}
	V\;\; &= \begin{pmatrix} V_{\tau \tau} & V_{\tau \delta}\\ V_{\delta \tau} & V_{\delta \delta}\end{pmatrix};\\
			V_{\tau\tau} & = p^{-1}\Sigma_{y(1),\infty} + (1 - p)^{-1}\Sigma_{y(0),\infty} - \Sigma_{\tau,\infty};       \\
			V_{\delta\delta}       & = \{p(1 - p)\}^{-1}\Sigma_{x,\infty};\\
			V_{\tau \delta}   & = p^{-1}\Sigma_{y(1)x,\infty} + (1-p)^{-1}\Sigma_{y(0)x,\infty} = V_{\delta\tau}^\T.\end{align*}

While $V_{\delta\delta}$ and $V_{\tau\delta}$ can be consistently estimated, $V_{\tau\tau}$ cannot be in the presence of effect heterogeneity due to its dependence on $\Sigma_{\tau}$,  the covariance of the unobserved treatment effects. Consequently, one cannot consistently estimate the probability that $\sqrt{N}(\hat{\bm{\tau}}-\bar{\bm{\tau}})$ falls within a given region $\mathcal{B}$. While consistent variance estimates are not available, there are several covariance estimators $\hat{V}_{\tau\tau}(\yobs, \mathbf{Z})$ for $V_{\tau\tau}$ satisfying $\hat{V}_{\tau\tau} - V_{\tau\tau} \overset{p}{\rightarrow} \Delta$ for some $\Delta \succeq 0$ under Assumptions \ref{asm: non-degen sampling limit} - \ref{asm: bounded fourth moment} in completely randomized designs. These estimators typically have the property that $\Sigma_{\tau\tau} = 0$ implies consistency, rather than asymptotic conservativeness; see \citet{decompTreatmentEffectVar} for more details.  So while the matrix $V$ cannot generally be consistently estimated, one can construct an estimate converging in probability to a matrix $\bar{\bar{{V}}}$ of the form


\begin{align*}
\bar{\bar{{V}}}\;\; &= \begin{pmatrix} V_{\tau \tau} + \Delta & V_{\tau \delta}\\ V_{\delta \tau} & V_{\delta \delta}\end{pmatrix}\end{align*}
	with $\Delta \succeq 0$.

	As an illustration, consider the conventional covariance estimator for the difference in means in a two-sample problem, $\hat{V}_{\tau\tau} = N\left(\hat{\Sigma}_{y(1)}/n_1 + \hat{\Sigma}_{y(0)}/n_0\right)$ with \begin{align*}
			\hat{\Sigma}_{y(1)}  & = \frac{1}{n_{1} - 1}\sum_{i = 1}^{N}Z_i\left(y_i(1) - n_1^{-1}\sum_{i=1}^NZ_iy_i(1)\right)\left(y_i(1) - n_1^{-1}\sum_{i=1}^NZ_iy_i(1)\right)^\T
		\end{align*} and the analogous for $\hat{\Sigma}_{y(0)}$. Under both completely randomized experiments and rerandomized experiments with balance criterion satisfying Condition \ref{cond: phi}, this estimator satisfies $\hat{V}_{\tau\tau} - V_{\tau\tau}\overset{p}{\rightarrow} \Sigma_{\tau,\infty} \succeq 0$ under Assumptions \ref{asm: non-degen sampling limit} - \ref{asm: bounded fourth moment}.

	\subsection{Limiting behavior of the reference distribution}
Suppose we have a completely randomized design, and consider the random variable \begin{align*}&\{\sqrt{N}\hat{\tau}(\tilde{\mathbf{y}}(\mathbf{Z}), \mathbf{W}),\sqrt{N}\hat{\delta}(\mathbf{x}, \mathbf{W})\}^\T\\ &= \sqrt{N}\left\{\frac{1}{n_1}\sum_{i=1}^N{W_i\tilde{\mathbf{y}}_i(Z_i)} - \frac{1}{n_0}\sum_{i=1}^N{(1-W_i)\tilde{\mathbf{y}}_i(Z_i)},\; \frac{1}{n_1}\sum_{i=1}^N{W_i\mathbf{x}_i} - \frac{1}{n_0}\sum_{i=1}^N{(1-W_i)\mathbf{x}_i}\right\}^\T \end{align*} where $\mathbf{Z}$ and $\mathbf{W}$ are independent, identically distributed, and drawn uniformly from $\Omega$ and $\tilde{\mathbf{y}}_i(Z_i) = \mathbf{y}_i(Z_i) - Z_i\bar{\bm{\tau}}$, such that $\tilde{\mathbf{y}}(\mathbf{Z}) = \yobs - \mathbf{Z}\bar{\bm{\tau}}^\T$.
\begin{proposition}\label{prop:permute}
    Subject to Assumptions \SLLNAssumptions, under a completely randomized design the distribution of $\{\sqrt{N}\hat{\tau}(\tilde{\mathbf{y}}(\mathbf{Z}), \mathbf{W}),\sqrt{N}\hat{\delta}(\mathbf{x}, \mathbf{W})\}^\T \mid \mathbf{Z}$ converges weakly in probability to a multivariate Gaussian measure, with mean zero and covariance $\tilde{V}$ of the form
    \begin{align*}
    \tilde{V}\;\; &= \begin{pmatrix} \tilde{V}_{\tau \tau} & \tilde{V}_{\tau \delta}\\ \tilde{V}_{\delta \tau} & \tilde{V}_{\delta \delta}\end{pmatrix};\\
    			\tilde{V}_{\tau\tau} & = (1-p)^{-1}\Sigma_{y(1),\infty} + p^{-1}\Sigma_{y(0),\infty};       \\
    			\tilde{V}_{\delta\delta}       & = \{p(1 - p)\}^{-1}\Sigma_{x,\infty};\\
    			\tilde{V}_{\tau \delta}   & = (1-p)^{-1}\Sigma_{y(1)x,\infty} + p^{-1}\Sigma_{y(0)x,\infty} = \tilde{V}_{\delta\tau}^\T.\end{align*}
\end{proposition}
The proof of this statement is contained within the proof of Theorem 1 in \citet{RandTestsWeakNulls} and is omitted. Under the sharp null, $\tilde{V} = V$ as $\mathbf{y}_i(1) = \mathbf{y}_i(0)$ for all $i$. Under the weak null however, while $\tilde{V}_{\delta\delta} = V_{\delta \delta}$ generally $\tilde{V}_{\tau\tau} \neq V_{\tau\tau}$ and $\tilde{V}_{\delta \tau} \neq V_{\delta \tau}$. The divergence between $V$ and $\tilde{V}$ can render randomization tests for the weak null hypothesis anti-conservative; examples are given in \S \ref{sec: examples}. We now describe how prepivoting may be used to guarantee asymptotic correctness when inference for the weak null hypothesis is conducted using a reference distribution generated under the sharp null.

\section{Gaussian Prepivoting}\label{sec: parametric prepiv}

\subsection{Prepivoting with an estimated pushforward measure}\label{sec: gaussian}
Consider functions $f_{\eta}:\R^{d} \rightarrow \R$ subject to the following requirement:
\begin{condition}\label{cond: f}
For any $\eta \in \Xi$, $f_{\eta}(\cdot): \mathbb{R}^d\mapsto \mathbb{R}_+$ is continuous, quasi-convex, and nonnegative with $f_\eta(\mathbf{t}) = f_\eta(-\mathbf{t})$ for all $\mathbf{t}\in \mathbb{R}^d$. Furthermore, $f_{\eta}(\mathbf{t})$ is jointly continuous in $\eta$ and $\mathbf{t}$.
\end{condition}

We begin with statistics for $H_N$ of the form
\begin{align}\label{eq: Tform}T(\yobs, \mathbf{Z}) &= f_{\hat{\xi}}(\sqrt{N}\hat{\bm{\tau}}),\end{align} where $\hat{\xi} = \hat{\xi}(\tilde{\mathbf{y}}(\mathbf{Z}), \mathbf{Z})$ satisfies the following condition for some set $\Xi$:

\begin{condition}\label{cond: xi}
    With $\mathbf{W}, \mathbf{Z}$ independent and each uniformly distributed over $\Omega$,\begin{align*}
    \hat{\xi}(\tilde{\mathbf{y}}(\mathbf{Z}), \mathbf{Z})  \overset{p}{\rightarrow}\xi;\;\;
    \hat{\xi}(\tilde{\mathbf{y}}(\mathbf{Z}), \mathbf{W}) \overset{p}{\rightarrow}\tilde{\xi}, 
    \end{align*}for some $\xi, \tilde{\xi} \in \Xi$.
\end{condition}

As will be shown in \S \ref{sec: examples}, several commonly encountered statistics for Neyman's null are of this form.  A detailed discussion of Condition~\ref{cond: f} is included in the appendix. Suppose further that one employs a covariance estimator $\hat{V}(\tilde{\mathbf{y}}(\mathbf{Z}), \mathbf{Z})$ with the following property:
\begin{condition}\label{cond: var}
With $\mathbf{W}, \mathbf{Z}$ independent, both uniformly distributed over $\Omega$, and for some  $\Delta \succeq 0$, $\Delta \in \mathbb{R}^{d\times d}$,\begin{align*} \hat{V}(\tilde{\mathbf{y}}(\mathbf{Z}), \mathbf{Z}) - V \overset{p}{\rightarrow} \begin{pmatrix}\Delta & 0_{d,k}\\  0_{k,d} & 0_{k,k} \end{pmatrix};\;\;
\hat{V}(\tilde{\mathbf{y}}(\mathbf{Z}), \mathbf{W}) - \tilde{V} \overset{p}{\rightarrow} 0_{(d+k), (d+k)}.\end{align*}\end{condition}

    As a concrete example satisfying Conditions~\ref{cond: f}-\ref{cond: var}, suppose that $f_{\eta}(\mathbf{t}) = \mathbf{t}^{\T}\eta^{-1}\mathbf{t}^{\T}$ and $\hat{\xi}(\bfy(\bfz), \bfw) = \hat{V}_{Neyman}(\bfy(\bfz), \bfw)$ with $\hat{V}_{Neyman}$ denoting the usual Neyman variance estimator of \citet{NeymanVariance}; numerous other examples are included in Section~\ref{sec: examples}.  Observe that when assuming the weak null for the purpose of testing, $\tilde{\mathbf{y}}(\mathbf{Z}) = \yobs$ and $\hat{\bm{\tau}} - \bar{\bm{\tau}} = \hat{\bm{\tau}}$.
Gaussian prepivoting transforms the test statistic $T(\bfy(\bfz), \bfw) = {f}_{\hat{\xi}}(\sqrt{N}\hat{\bm{\tau}}(\bfy(\bfz), \bfw))$ into a new statistic of the form
\begin{align}
    G(\mathbf{y}(\bfz), \bfw) = \frac{\gamma^{(d+k)}_{\bm{0}, \hat{V}(\bfy(\bfz), \bfw)}\left\{(\mathbf{a}, \mathbf{b})^\T: {f}_{\hat{\xi}}(\mathbf{a}) \leq T(\bfy(\bfz), \bfw)\; \wedge\; \phi(\mathbf{b})=1\right \}}{\gamma^{(k)}_{\bm{0}, \hat{V}_{\delta \delta}}\left\{\mathbf{b}:  \phi(\mathbf{b})=1\right\}}
\label{eqn: gaussian}
\end{align}
\color{black}
where $\gamma^{(p)}_{\mu, \Sigma}(\mathcal{B})$ is the $p$-dimensional Gaussian measure of a set $\mathcal{B}$ with mean parameter $\bm{\mu}$ and covariance $\Sigma$, i.e.
\begin{align*}
    \gamma^{(p)}_{\mu, \Sigma}(\mathcal{B}) &= \frac{1}{\sqrt{(2\pi)^p|\Sigma|}}\int_{\mathbf{x} \in \mathcal{B}}\exp\left\{-\frac{1}{2}(\mathbf{x}-\bm{\mu})^\T\Sigma^{-1}(\bm{x}-\bm{\mu})\right\}\;dx.
\end{align*}  For $(\mathbf{A},\mathbf{B})^\T$ jointly multivariate normal with mean zero and covariance $\hat{V}$, $\mathbf{A}\in \mathbb{R}^d$, $\mathbf{B}\in \mathbb{R}^k$, $G(\mathbf{y}(\mathbf{Z}), \mathbf{Z})$ represents the $f_{\hat{\xi}}$-pushforward measure of $\mathbf{A} \mid \phi(\mathbf{B}) = 1$ evaluated on the set $(-\infty, T(\mathbf{y}(\mathbf{Z}), \mathbf{Z})]$. That is, $G(\mathbf{y}(\mathbf{Z}), \mathbf{Z})$ treats ${f}_{\hat{\xi}}$ and $\hat{V}$ as fixed and computes the conditional probability that ${f}_{\hat{\xi}}(\mathbf{A})$ falls at or below the observed value for $T(\yobs, \mathbf{Z})= {f}_{\hat{\xi}}(\sqrt{N}\hat{\bm{\tau}})$ given that $\varphi(\mathbf{B})=1$. From the perspective of hypothesis testing, $G(\mathbf{y}(\mathbf{Z}), \mathbf{Z})$ is 1 minus the large-sample $p$-value for $T(\yobs, \mathbf{Z})$ leveraging the finite population central limit theorem and the estimated covariance $\hat{V}$.


        We now describe how to use the prepivoted statistic $G(\yobs, \bfZ)$ to provide a single procedure that is both exact for $H_F$ and asymptotically conservative for $H_{N}$.  In order to provide a precise implementation of this, we give detailed pseudocode in Algorithm~\ref{alg: Gaussian prepivoting} and provide example code through the appendix.  First, we compute the prepivoted test statistic $G(\yobs, \bfZ)$ given the observed data; this proceeds according to Equation~\ref{eqn: gaussian} and is Step~1 of Algorithm~\ref{alg: Gaussian prepivoting}.  Next, we construct the reference distribution $\hatPerm_{G}(\cdot)$, the construction of which requires imputing counterfactual outcomes as if Fisher's sharp null held; this is Step~2 of Algorithm~\ref{alg: Gaussian prepivoting}.  Finally, the $p$-value for testing $H_{N}$ is computed and we reject the null when this lies below or at the nominal level $\alpha$.

		\begin{algorithm}[H] 
			\SetAlgoLined
			\KwIn{An observed treatment allocation $\bfz$, with observed responses $\mathbf{y}(\bfz)$, test statistic $T(\mathbf{y}(\bfz), \bfz) = f_{\hat{\xi}}(\sqrt{N}\hat{\bm{\tau}}_{obs})$ and covariance estimator $\hat{V}(\mathbf{y}(\bfz), \bfz)$}
			\KwResult{The $p$-value for the Gaussian prepivoted test statistic}
			\textbf{Step 1: The observed prepivoted statistic}\\
			Compute $f_{\hat{\xi}(\mathbf{y}(\bfz), \bfz)}(\cdot)$; $\hat{V}(\mathbf{y}(\bfz), \bfz)$.
			Compute
			\begin{align*}
			 g_{\mathbf{z}} &= \frac{\gamma^{(d+k)}_{\bm{0}, \hat{V}(\mathbf{y}(\bfz), \bfz)}\left\{(\mathbf{a}, \mathbf{b})^\T: {f}_{\hat{\xi}(\mathbf{y}(\mathbf{z}), \bfz)}(\mathbf{a}) \leq T(\mathbf{y}(\mathbf{z}), \mathbf{z})\; \wedge\; \phi(\mathbf{b})=1\right \}}{\gamma^{(k)}_{\bm{0}, \hat{V}_{\delta \delta}(\mathbf{y}(\bfz), \bfz)}\left\{\mathbf{b}: \phi(\mathbf{b})=1\right\}}
			 \end{align*}

            \textbf{Step 2: The reference distribution $\hatPerm_{G}$}\\
			\For{$\mathbf{w} \in \Omega$}{
				Compute $f_{\hat{\xi}(\mathbf{y}(\bfz), \mathbf{w})}(\cdot)$; $\hat{V}(\mathbf{y}(\bfz), \mathbf{w})$.

				Compute \begin{align*}
				 g_{\mathbf{w}} &= \frac{\gamma^{(d+k)}_{\bm{0}, \hat{V}(\mathbf{y}(\bfz), \mathbf{w})}\left\{(\mathbf{a}, \mathbf{b})^\T: {f}_{\hat{\xi}(\mathbf{y}(\mathbf{z}), \mathbf{w})}(\mathbf{a}) \leq T(\mathbf{y}(\mathbf{z}), \mathbf{w})\; \wedge\; \phi(\mathbf{b})=1\right \}}{\gamma^{(k)}_{\bm{0}, \hat{V}_{\delta \delta}(\mathbf{y}(\bfz), \mathbf{w})}\left\{\mathbf{b}: \phi(\mathbf{b})=1\right\}}
				 \end{align*}
			}
			\Return{ \begin{align*} p_{val} &= \frac{1}{|\Omega|}\sum_{\mathbf{w} \in \Omega}\indicatorFunction{g_\mathbf{w}\geq g_{\mathbf{z}}};\\
			\varphi_G(\alpha) &= \indicatorFunction{p_{val}\leq \alpha}.
			\end{align*}}
			\caption{Inference for the weak null through Gaussian prepivoting}\label{alg: Gaussian prepivoting}
		\end{algorithm}
\color{black}
Observe that \textcolor{black}{$1-g_{\mathbf{z}}$} defined within Algorithm \ref{alg: Gaussian prepivoting} is the usual large-sample $p$-value based upon a Gaussian approximation and using the covariance estimator $\hat{V}$. The large-sample test compares \textcolor{black}{$1-g_{\mathbf{z}}$} to $\alpha$, the desired Type I error rate, and rejects if \textcolor{black}{$1-g_{\mathbf{z}}\leq \alpha \Leftrightarrow g_{\mathbf{z}} \geq 1-\alpha$}. The Gaussian prepivoted randomization test instead rejects if \textcolor{black}{$g_{\mathbf{z}} \geq \hatPerm_G^{-1}(1-\alpha)$}. The following Theorem, in concert with Lemma 11.2.1 of \citet{TestingStatHyp}, show under our assumptions $\hatPerm_G^{-1}(1-\alpha)\overset{p}{\rightarrow} 1-\alpha$, such that the prepivoted randomization test is asymptotically equivalent to large sample test under the weak null. By using $\hatPerm_G^{-1}(1-\alpha)$ instead of $1-\alpha$, exactness under the sharp null is preserved.

\begin{theorem}\label{thm: gaussian}
    Suppose we have either a completely randomized design or a rerandomized design with balance criterion $\phi$ satisfying Condition~1. Suppose ${T}(\yobs, \mathbf{Z})$ is of the form (4) for some $f_\eta$ and $\hat{\xi}$ satisfying Conditions 2 and 3. Suppose further that we employ a covariance estimator $\hat{V}$ satisfying Condition 4 when forming the prepivoted test statistic $G(\yobs, \bfZ)$. Then, under $H_N: \bar{\bm{\tau}} = 0$ and under Assumptions 1-3, $G(\mathbf{y}(\mathbf{Z}), \mathbf{Z})$ converges in distribution to a random variable $\tilde{U}$ taking values in $[0,1]$ satisfying
    \begin{align*}
        \mathbb{P}(\tilde{U} \leq t) \geq t,
    \end{align*}
    for all $t \in [0,1]$. Furthermore, the distribution $\hatPerm_G(t)$ satisfies $\hatPerm_G(t)\overset{p}{\rightarrow} t$ for all $t\in [0,1]$.\looseness=-1
\end{theorem}
\begin{corollary}
    Under the conditions of Theorem~\ref{thm: gaussian}; the prepivoted test statistic $G(\yobs, \bfZ)$ is asymptotically sharp dominant regardless of whether the base statistic $T(\yobs, \bfZ)$ was.  Consequently, $p$-values derived under $\hatPerm_{G}$ via Algorithm~1 are guaranteed to be exact under $H_{F}$ and asymptotically conservative under just $H_{N}$.
\end{corollary}
\color{black}
Theorem \ref{thm: gaussian} states that under the weak null, $G(\yobs, \mathbf{Z})$ converges in distribution to a random variable which is stochastically dominated by the standard uniform. Meanwhile, the reference distribution for $G(\yobs, \mathbf{Z})$ constructed assuming (incorrectly) that the sharp null holds converges pointwise to the distribution function of a standard uniform. As a result, the randomization distribution for $G(\yobs, \mathbf{Z})$ is asymptotically sharp-dominant: the reference distribution generated in this manner yields asymptotically conservative inference for the weak null hypothesis, while maintaining exactness should the sharp null also hold.  \textcolor{black}{By exploiting the duality between hypothesis testing and confidence sets Theorem~\ref{thm: gaussian} provides the basis for generating exact and asymptotically conservative confidence sets for treatment effect; this is explored in the appendix.}\looseness=-1

\begin{remark}
Consider the function\begin{align*}
\hat{F}(t) = \frac{\gamma^{(d+k)}_{\bm{0}, \hat{V}}\left\{(\mathbf{a}, \mathbf{b})^\T: f_{\hat{\xi}}(\mathbf{a}) \leq t\; \wedge\; \phi(\mathbf{b})=1\right \}}{\gamma^{(k)}_{\bm{0}, \hat{V}_{\delta \delta}}\left\{\mathbf{b}:  \phi(\mathbf{b})=1\right\}}\end{align*}
the estimated distribution function for $f_{\hat{\xi}}(\sqrt{N}\hat{\bm{\tau}}) \mid \phi(\sqrt{N}\hat{\bm{\delta}}) = 1$ based upon a finite population central limit theorem. In special cases, the function $\hat{F}(t)$ may have a known closed form. This is true of the test statistics which are sharp-dominated by a $\chi^2_d$ distribution considered in \citet{RandTestsWeakNulls}, for example. Should this not be the case, one can approximate $\hat{F}(\cdot)$ by way of Monte-Carlo approximation, replacing the measures $\gamma_{\bm{0}, \hat{V}}$ and $\gamma_{\bm{0}, \hat{V}_{\delta \delta}}$ with estimates based upon a $B$ draws from a multivariate normal with mean $\bm{0}$ and covariance $\hat{V}$ when enumerating the reference distribution. Importantly, such Monte-Carlo approximation does \textit{not} corrupt finite-sample exactness under Fisher's sharp null.
\end{remark}

\subsection{Examples of Gaussian prepivoting}\label{sec: examples}
Through a series of examples, we now provide illustrations of the transformations achieved by (\ref{eqn: gaussian}). As will be demonstrated, the form recovers several randomization tests previously known to be valid for weak null hypotheses in the literature while providing a basis for new randomization tests for weak nulls using other test statistics.  \textcolor{black}{ These examples serve four objectives: (i) unify previous \textit{ad hoc} solutions under the framework of Gaussian prepivoting; (ii) provide an alternative approach to already valid procedures; (iii) highlight that prepivoting can succeed even where studentization fails; and (iv) extend randomization inference for $H_{F}$ and $H_{N}$ to new experimental designs.}

	\begin{example}[Absolute difference in means]\label{example: difference in means for cre}
		Let $\sqrt{N}\hat{{\tau}}$ be univariate, consider a completely randomized design with no rerandomization, and let $T_{DiM}(\yobs, \mathbf{Z}) = \sqrt{N}|\tauhat|$, with $f_{\eta}(t) = |t|$ and $\hat{\xi} = 1$. The randomization distribution for $T_{DiM}(\yobs, \mathbf{Z})$ is not asymptotically sharp-dominant, such that employing the reference distribution assuming that the sharp null holds may lead to anti-conservative inference. The conventional fix is to studentize $\sqrt{N}|\hat{{\tau}}|$ using a variance estimator estimator satisfying Condition \ref{cond: var}, forming instead $T_{Stu}(\yobs, \mathbf{Z}) = \sqrt{N}|\hat{{\tau}}|/\sqrt{\hat{V}_{\tau\tau}}$ \citep{loh17}.

		As $\phi(\cdot) = 1$ deterministically in a completely randomized design, Gaussian prepivoting via (\ref{eqn: gaussian}) yields the test statistic
		\begin{align*}
		    G_{DiM}(\mathbf{y}(\mathbf{Z}), \mathbf{Z}) = \gamma^{(1)}_{0, \hat{V}_{\tau\tau}}\{a: |a| \leq \sqrt{N}|\hat{{\tau}}|\} = 1 - 2\Phi\left(-\frac{\sqrt{N}|\hat{{\tau}}|}{\sqrt{\hat{V}_{\tau\tau}}}\right),
		\end{align*}
		where $\Phi(\cdot)$ is the standard normal distribution function. For any $\mathbf{Z}$, the pairs \\$\{G_{DiM}(\mathbf{y}(\mathbf{Z}), \mathbf{w}), T_{DiM}(\yobs, \mathbf{w})\}$ have rank correlation equal to 1 when computed for all $\mathbf{w}\in \Omega$. As a result, the reference distribution using the studentized difference in means assuming the sharp null will furnish identical $p$-values to those attained using Gaussian prepivoting. That is, in the univariate case Gaussian prepivoting is equivalent to studentization for completely randomized designs. This highlights \textcolor{black}{objectives (i) and (ii)}.
\end{example}
\begin{example}[Multivariate studentization] \label{example: mult}
Let $\sqrt{N}\hat{\bm{\tau}}$ now be multivariate and suppose we have a completely randomized design. \citet{RandTestsWeakNulls} suggest the test statistic
		\begin{align}\label{eqn: studentized test statistic}
			T_{\chi^{2}}(\yobs, \mathbf{Z}) &= \left(\sqrt{N} \hat{\bm{\tau}}\right)^{\T} \hat{V}_{\tau\tau}^{-1} \left(\sqrt{N} \hat{\bm{\tau}}\right), 
		\end{align}
		with $\hat{V}_{\tau\tau} = \frac{N}{n_{1}}\hat{\Sigma}_{y(1)} + \frac{N}{n_{0}}\hat{\Sigma}_{y(0)}$.  For this test statistic, $f_\eta(\mathbf{t}) = \mathbf{t}^\T\eta^{-1}\mathbf{t}$ and $\hat{\xi} = \hat{V}_{\tau\tau}$. \citet{RandTestsWeakNulls} show that under our assumptions, under the weak null this test statistic converges in distribution to $\sum_{i = 1}^{d}w_{i}\zeta_{i}^{2}$ where $w_{i} \in [0, 1]$ are weights and $\zeta_{1}, \ldots, \zeta_{d} \iid \Normal{0}{1}$ while the reference distribution of $T_{\chi^{2}}(\yobs, \mathbf{Z})$ attained assuming that the sharp null holds converges weakly in probability to the $\chi_{d}^{2}$-distribution. As a result, $T_{\chi^{2}}(\yobs, \mathbf{Z})$ is asymptotically sharp-dominant, and its reference distribution assuming the sharp null may be used for inference for the weak null hypothesis. Here, Gaussian prepivoting produces
		\begin{align*}
		    G_{\chi^2}(\mathbf{y}(\mathbf{Z}), \mathbf{Z}) = \gamma^{(d)}_{\bm{0},\hat{V}_{\tau\tau}}\{\mathbf{a}: \mathbf{a}^\T\hat{V}_{\tau\tau}^{-1}\mathbf{a} \leq T_{\chi^2}(\yobs, \mathbf{Z})\} = F_d\{T_{\chi^2}(\yobs, \mathbf{Z})\},
		\end{align*}
		where $F_d(\cdot)$ is the distribution function of a $\chi^2_d$ random variable. For any $\mathbf{Z}$, the pairs \{$G_{\chi^2}(\mathbf{y}(\mathbf{Z}), \mathbf{w}), T_{\chi^2}(\yobs, \mathbf{w})\}$ have rank correlation equal to 1 when computed for all $\mathbf{w}\in \Omega$, such that Gaussian prepivoting yields equivalent inference to that attained using the distribution of $T_{\chi^2}(\yobs, \mathbf{Z})$ under the sharp null.  \textcolor{black}{This demonstrates objective (ii).}

		Suppose instead that, erroneously, a practitioner proceeded with the more typical form of Hotelling's $T$-squared statistic employing a pooled covariance estimator,
		\begin{align*}
			T_{Pool}(\yobs, \mathbf{Z}) = \left(\sqrt{N} \hat{\bm{\tau}}\right)^{\T} \left(\hat{V}_{Pool}\right)^{-1} \left(\sqrt{N} \hat{\bm{\tau}}\right);\\
			\hat{V}_{Pool} = \left(\frac{N}{n_{0}} + \frac{N}{n_{1}} \right)\left(\frac{(n_1-1)\hat{\Sigma}_{y(1)} + (n_0-1)\hat{\Sigma}_{y(0)}}{n_1+n_0-2}\right). \nonumber
		\end{align*}
		For this test statistic, $f_\eta(\mathbf{t}) = \mathbf{t}^\T\eta^{-1}\mathbf{t}$ as before, but $\hat{\xi} = \hat{V}_{Pool}$. In this case, $T_{Pool}(\yobs, \mathbf{Z})$ is not asymptotically sharp-dominant, such that the reference distribution using this statistic and assuming the sharp null may yield invalid inference. Gaussian prepivoting returns the test statistic
		\begin{align*}
		    G_{Pool}(\mathbf{y}(\mathbf{Z}), \mathbf{Z}) = \gamma^{(d)}_{\bm{0},\hat{V}_{\tau\tau}}\{\mathbf{a}: \mathbf{a}^\T\hat{V}_{Pool}^{-1}\mathbf{a} \leq T_{Pool}(\yobs, \mathbf{Z})\}.
		\end{align*}
Importantly, $G_{Pool}(\mathbf{y}(\mathbf{Z}), \mathbf{Z})$ continues to use the Gaussian measure computed with the covariance matrix $\hat{V}_{\tau\tau}$ in forming the suitable transformation, despite the fact that the pooled covariance matrix is used in forming $T_{Pool}(\mathbf{y}(\mathbf{Z}), \mathbf{Z})$. For fixed $\mathbf{Z}$, $G_{Pool}(\mathbf{y}(\mathbf{Z}), \mathbf{w})$ generally will not have perfect rank correlation with $T_{Pool}(\yobs, \mathbf{w})$ when computed over $\mathbf{w}\in\Omega$, such that the two randomization tests assuming the sharp null no longer furnish identical $p$-values. This divergence is necessary: while  $T_{Pool}(\yobs, \mathbf{Z})$ is not asymptotically sharp-dominant, Theorem \ref{thm: gaussian} asserts that $G_{Pool}(\yobs, \mathbf{Z})$ is, such that the reference distribution for $G_{Pool}(\yobs, \mathbf{Z})$ assuming the sharp null yields asymptotically conservative inference for the weak null. Gaussian prepivoting can thus restore asymptotic validity to a test statistic employing improper studentization, illustrating objective (iii).
\end{example}

\begin{example}[Max absolute $t$-statistic]
Consider again multivariate $\sqrt{N}\hat{\bm{\tau}}$ and a completely randomized design, and consider the test statistic
\begin{align*}
    T_{|max|}(\mathbf{y}(\mathbf{Z}), \mathbf{Z}) = \max_{1\leq j \leq d}\frac{\sqrt{N}|\hat{{\tau}}_j|}{\sqrt{\hat{V}_{\tau\tau, jj}}},
\end{align*}
where $\hat{V}_{\tau\tau, jj}$ is the $jj$ element of $\hat{V}_{\tau\tau}$. For this statistic, $f_{\bm{\eta}}(\mathbf{t}) = \max_{1\leq j\leq d} |t_j|/\eta_j$, and $\hat{\bm{\xi}} = (\hat{V}^{1/2}_{\tau\tau, 11},...,\hat{V}^{1/2}_{\tau\tau, dd})^\T$. For $d\geq 2$,  $T_{|max|}(\mathbf{y}(\mathbf{Z}), \mathbf{Z})$ is not asymptotically sharp-dominant under the weak null: the reference distribution generated under the sharp null depends upon the correlation matrix corresponding to $\tilde{V}$, while the true randomization distribution is governed by the correlations encoded within $V$. The Gaussian prepivoted correction takes the form
\begin{align*}
    G_{|max|}(\mathbf{y}(\mathbf{Z}), \mathbf{Z}) =
    \gamma^{(d)}_{\bm{0},\hat{V}_{\tau\tau}}\left\{\mathbf{a}: \max_{1\leq j \leq d}\; \frac{|a_j|}{\sqrt{\hat{V}_{\tau \tau, jj}}} \leq  \max_{1\leq j \leq d}\; \frac{\sqrt{N}|\hat{{\tau}}_j|}{\sqrt{\hat{V}_{\tau\tau, jj}}}\right\},
\end{align*}
which composes $T_{|max|}(\mathbf{y}(\mathbf{Z}), \mathbf{Z})$ with the distribution function for $\max
\;|A_j|/\sqrt{\hat{V}_{\tau\tau, jj}}$, $j=1,..,d$, when $A$ is multivariate Gaussian with mean zero and covariance $\hat{V}_{\tau\tau}$.  \textcolor{black}{Gaussian prepivoting rectifies the insufficiency of the studentization in $T_{|max|}$, thereby providing an example of objective (iii).}
\end{example}

\begin{example}[Rerandomization] \label{example: rerandomization}
Let $\sqrt{N}\hat{{\tau}}$ be univariate and suppose we now consider a rerandomized design with balance criterion $\phi$ satisfying Condition \ref{cond: phi}. Consider the absolute  difference in means, $f_{\hat{\xi}}(\sqrt{N}\hat{{\tau}}) = \sqrt{N}|\hat{{\tau}}|$, such that $\hat{\xi} = 1$. Gaussian prepivoting yields the test statistic
\begin{align*}
 G_{Re}(\mathbf{y}(\mathbf{Z}), \mathbf{Z}) &= \frac{\gamma^{(1+k)}_{\bm{0}, \hat{V}}\left\{(\mathbf{a}, \mathbf{b})^\T: |a| \leq \sqrt{N}|\hat{{\tau}}|\; \wedge\; \phi(\mathbf{b})=1\right \}}{\gamma^{(k)}_{\bm{0}, \hat{V}_{\delta \delta}}\left\{\mathbf{b}: \phi(\mathbf{b})=1\right\}}\end{align*}
\end{example}

For completely randomized designs with $\phi(\cdot)=1$ deterministically, Gaussian prepivoting is equivalent to studentizing as described in Example 1. In general rerandomized designs however, observe that the transformation depends upon the particular form of the balance criterion $\phi$, and that the reference distribution will depend upon the relationship between the potential outcomes and the covariates used in the balance criterion. As a result, it will generally not be the case that the reference distribution of $G_{Re}(\mathbf{y}(\mathbf{Z}), \mathbf{Z})$ under the sharp null yields equivalent inference to that attained using $\sqrt{N}|\hat{{\tau}}|/\sqrt{\hat{V}_{\tau\tau}}$. This suggests that in rerandomized designs, studentization alone is insufficient for attaining an asymptotically sharp-dominant test statistic. In \S \ref{subsec: studentization in rerandomization}, we show this through an example in the case of Mahalanobis rerandomization.  Lemmas A15 and A16 of \citet{asymptoticsOfRerand} show that under our conditions, probability limits for estimators $\hat{V}$ derived under complete randomization are generally preserved under rerandomized designs. Once again, Theorem \ref{thm: gaussian} ensures that  $G_{Re}(\mathbf{y}(\mathbf{Z}), \mathbf{Z})$ will be asymptotically sharp-dominant, such that the randomization distribution assuming the sharp null may be employed for inference for the weak null.  \textcolor{black}{The development of a finite sample exact method for testing $H_{F}$ which is asymptotically valid for testing $H_{N}$ in rerandomized designs is novel, but its construction is extremely simple within the framework of Gaussian prepivoting; this highlights Gaussian prepivoting's portability to designs outside of just completely randomized experiments.  In the appendix we provide two more examples of this portability: one for matched-pair designs and one for experiments with any finite number of treatment arms.  This highlights objective (iv).}

For the interested reader, in the appendix we include this same collection of examples written directly in the form of Gaussian integrals, and we include verification of Conditions~\ref{cond: f}-\ref{cond: var}.

\section{Gaussian comparison, stochastic dominance, and the probability integral transform}\label{sec: illustration}
\subsection{Gaussian comparison and  Anderson's Theorem}\label{sec:anderson}
We now highlight the essential technical ingredients underpinning the success of Gaussian prepivoting. Consider two mean-zero multivariate Gaussian vectors $(\mathbf{A}_1, \mathbf{B}_1)^\T$ and $(\mathbf{A}_2,\mathbf{B}_2)^\T$, with covariances
\begin{align*} M_1 = \begin{pmatrix}  \Lambda^{(1)}_{aa} & \Lambda_{ab}\\ \Lambda_{ba} & \Lambda_{bb}\end{pmatrix};\;\;\; M_2 = \begin{pmatrix}  \Lambda^{(2)}_{aa} & \Lambda_{ab}\\ \Lambda_{ba} & \Lambda_{bb},\end{pmatrix},\end{align*}
satisfying $\Lambda^{(2)}_{aa} - \Lambda^{(1)}_{aa} \succeq 0$ and $\Lambda_{bb} \succ 0$; the inequalities are stated with respect to the Loewner partial order on positive semidefinite matrices. Let the dimensions of $\mathbf{A}_j$ and $\mathbf{B}_j$ be $d$ and $k$ respectively for $j=1,2$. Compare the tail probabilities for
\begin{align*} f(\mathbf{A}_1) \mid  \phi\left(\mathbf{B}_1\right) = 1\;\; \text{and}\;\; f(\mathbf{A}_2) \mid  \phi\left(\mathbf{B}_2\right) = 1, \end{align*} where $\phi$ and $f$ satisfy Conditions \ref{cond: phi} and Condition \ref{cond: f} respectively. The following result is a straightforward corollary of \citeauthor{conservativeCovarianceDominance}'s (\citeyear{conservativeCovarianceDominance}) theorem for multivariate Gaussians; see also Theorem 4.2.5 of \citet{mvtNormal_Tong}.
\begin{lemma}\label{prop:anderson}
			Under the stated conditions, for any scalar $v$,
			\begin{align*}
			\mathbb{P}\left\{f(\mathbf{A}_1)\geq v \mid \phi(\mathbf{B}_1)=1\right\}\leq \mathbb{P}\left\{f(\mathbf{A}_2)\geq v\mid \phi(\mathbf{B}_2)=1\right\}.
			\end{align*}
		\end{lemma}

The result follows immediately from Anderson's theorem after noting that the set $\mathcal{B}_v = \{(\mathbf{a}, \mathbf{b})^\T: f(\mathbf{a}) \leq v\; \wedge\; \phi(\mathbf{b})=1\}$ is convex and mirror-symmetric for any $v$. This can be seen through our assumption that $f(\cdot)$ is quasi-convex and mirror-symmetric, such that its sublevel sets are convex and mirror symmetric. We further have that $\mathbb{P}(\phi(\mathbf{B}_1)=1) = \mathbb{P}(\phi(\mathbf{B}_2)=1) > 0$ given the structure of the covariance matrices $M_1$ and $M_2$ and Condition \ref{cond: phi}, completing the proof.

 \subsection{Stochastic dominance and the probability integral transform}
For two real valued random variables $S$ and $T$,  $S$ \textit{(first order) stochastically dominates} $T$ if $F_{S}(a) \leq F_{T}(a)$ for all $a \in \mathbb{R}$, where $F_{S}$ and $F_T$ are the distribution functions of $S$ and $T$ respectively.

Suppose now that  $S$ and $T$ are continuous and that $S$ stochastically dominates $T$. By the probability integral transform, the distribution of $F_T(T)$ would be standard uniform. The following proposition considers transforming the random variable $T$ not by its own distribution function, but rather by the distribution function of $S$, its stochastically dominating random variable.
\begin{lemma}\label{prop:unif}
Suppose that $S$, $T$ are continuous random variables and that $S$ stochastically dominates $T$.  Then, $F_S(T)$ is stochastically dominated by a standard uniform random variable.
\end{lemma}
\begin{proof}
For any $t\in [0,1]$, $\mathbb{P}\{F_S(T) \leq t\} = \mathbb{P}\{T \leq F_S^{-1}(t)\}\geq \mathbb{P}\{S \leq F_S^{-1}(t)\} = t.$\end{proof}

In the setup of \S \ref{sec:anderson}, under Conditions \ref{cond: phi} and \ref{cond: f} we have by Proposition \ref{prop:anderson} that $f(\mathbf{A}_2)\mid \phi(\mathbf{B}_2)=1$ stochastically dominates $f(\mathbf{A}_1)\mid \phi(\mathbf{B}_1)=1$. Consequently, composing $f(\mathbf{A}_1)\mid \phi(\mathbf{B}_1)=1$ with the distribution function of $f(\mathbf{A}_2)\mid \phi(\mathbf{B}_2)=1$ would yield a random variable that is stochastically dominated by a standard uniform.

\subsection{A proof sketch for Theorem \ref{thm: gaussian}}\label{sec: sketch}
While a formal proof of Theorem \ref{thm: gaussian} is deferred to the appendix, here we provide an informal sketch in light of Lemmas \ref{prop:anderson} and \ref{prop:unif}. Under Assumptions \ref{asm: non-degen sampling limit} - \ref{asm: bounded fourth moment} and Condition \ref{cond: phi}, $\sqrt{N}(\hat{\bm{\tau}} - \bar{\bm{\tau}})$ converges in distribution to $\mathbf{A}_1\mid \phi(\mathbf{B_1}) = 1$, where  $(\mathbf{A}_1, \mathbf{B}_1)^\T$ are jointly multivariate normal with covariance $V$. Recall that $T(\yobs, \mathbf{Z}) = f_{\hat{\xi}}(\sqrt{N}\hat{\bm{\tau}})$ for some $f_{\eta}$ satisfying Condition \ref{cond: f} for all $\eta\in \Xi$, some $\hat{\xi}$ satisfying Condition \ref{cond: xi}, and with a balance criterion $\phi$ satisfying Condition \ref{cond: phi}. By Condition \ref{cond: xi} and the assumption of the weak null, we have that $\hat{\xi}(\yobs, \bfZ)$ converges in probability to $\xi$. Therefore, under the weak null, by Lemma \ref{prop:anderson} the limiting distribution of $T(\yobs, \mathbf{Z})$ would be stochastically dominated by that of $f_\xi(\mathbf{A}_2) \mid \phi(\mathbf{B}_2)=1$ for any $(\mathbf{A}_2,\mathbf{B}_2)^\T$ multivariate Gaussian with covariance matrix
\begin{align*}
\bar{\bar{V}} &= V + \begin{pmatrix}\Delta & 0_{d,k}\\  0_{k,d} & 0_{k,k} \end{pmatrix}
\end{align*}
with $\Delta\succeq 0$. The transformation
\begin{align*}
\bar{\bar{G}}(\mathbf{y}(\mathbf{Z}), \mathbf{Z}) = \frac{\gamma^{(d+k)}_{\bm{0}, \bar{\bar{V}}}\left\{(\mathbf{a}, \mathbf{b})^\T: {f}_{\hat{\xi}}(a) \leq f_{{\hat{\xi}}}(\sqrt{N}\hat{\bm{\tau}})\; \wedge\; \phi(\mathbf{b})=1\right \}}{\gamma^{(k)}_{\bm{0}, \bar{\bar{V}}_{\delta \delta}}\left\{\mathbf{b}: \phi(\mathbf{b})=1\right\}}
\end{align*}
transforms $T(\yobs, \mathbf{Z})$ by the distribution function of a random variable which stochastically dominates its limiting distribution. By Lemma \ref{prop:unif} and the continuous mapping theorem, asymptotically $\bar{\bar{G}}(\mathbf{y}(\mathbf{Z}), \mathbf{Z})$ is stochastically dominated by a standard uniform. By Condition \ref{cond: var}, the covariance estimator $\hat{V}$ used in forming ${G}(\mathbf{y}(\mathbf{Z}), \mathbf{Z})$ has a probability limit of the required form for stochastic dominance. Therefore, another application of the continuous mapping theorem yields that  ${G}(\mathbf{y}(\mathbf{Z}), \mathbf{Z}) - \bar{\bar{G}}(\mathbf{y}(\mathbf{Z}), \mathbf{Z}) = o_p(1)$, such that by Slutsky's Theorem ${G}(\mathbf{y}(\mathbf{Z}), \mathbf{Z})$ is itself stochastically dominated by a standard uniform.

Meanwhile, Proposition \ref{prop:permute} and Condition \ref{cond: phi} yield that under the weak null the distribution of $\sqrt{N}\hat{\tau}(\yobs, \mathbf{W}) \mid \bfZ$ converges weakly in probability to the distribution of $\tilde{\mathbf{A}}\mid \phi(\tilde{\mathbf{B}}) = 1$, where $(\tilde{\mathbf{A}}, \tilde{\mathbf{B}})^\T$ are jointly multivariate Gaussian with mean zero and covariance $\tilde{V}$. The distribution of $f_{\hat{\xi}(\yobs, \mathbf{W})}\{\sqrt{N}\hat{\tau}({\yobs}, \mathbf{W})\} \mid \bfZ$ is precisely $\hatPerm_T$, the reference distribution assuming the sharp null holds for the test statistic $T(\yobs, \bfZ) = f_{\hat{\xi}}(\sqrt{N}\hat{\bm{\tau}})$. By Condition \ref{cond: var}, $\hat{V}(\yobs, \mathbf{W})$ converges in probability to $\tilde{V}$ itself. Further, by Condition \ref{cond: xi} $\hat{\xi}(\yobs, \mathbf{W})$ converges in probability to $\tilde{\xi}$. Applying the continuous mapping theorem and Slutsky's Theorem for randomization distributions \citep[Lemmas A5-A6]{chungRomano16}, one sees that Gaussian prepivoting furnishes a transformation that amounts to, asymptotically, an application of the probability integral transform. As a result, $\hatPerm_G(t)$ converges in probability to $t$, the distribution function of the standard uniform, for all $t\in [0,1]$.

\section{Extensions to asymptotically linear estimators}\label{sec: linear}

		Theorem \ref{thm: gaussian} may be extended to estimators other than the difference in means. Consider an estimator $\breve{\tau}(\mathbf{y}(\mathbf{Z}), \mathbf{Z})$ such that 
		\begin{align*}
		   \sqrt{N}\{\breve{\tau}(\mathbf{y}(\mathbf{Z}), \mathbf{Z}) - \bar{\bm{\tau}}\} &= \sqrt{N}\left(\frac{1}{n_1}\sum_{i=1}^NZ_i\mathbf{r}_i(Z_i) - \frac{1}{n_0}\sum_{i=1}^N(1-Z_i)\mathbf{r}_i(Z_i)\right) + o_p(1)
		\end{align*} for some constants $\{\mathbf{r}_i(0), \mathbf{r}_i(1)\}_{i=1}^N$ which may change with $N$ and that satisfy $(1/N)\sum_{i=1}^N(\mathbf{r}_i(1)-\mathbf{r}_i(0)) = 0$ along with Assumptions \ref{asm: means and covs stabilize} and \ref{asm: bounded fourth moment}. Suppose further that $\breve{\tau}(\tilde{\mathbf{y}}(\mathbf{Z}), \mathbf{W})$, $\mathbf{W}$ independent from $\mathbf{Z}$ and drawn uniformly from $\Omega$, satisfies 
			\begin{align*}
		   \sqrt{N}\breve{\tau}(\tilde{\mathbf{y}}(\mathbf{Z}), \mathbf{W}) &= \sqrt{N}\left(\frac{1}{n_1}\sum_{i=1}^NW_i\tilde{\mathbf{r}}_i(Z_i) - \frac{1}{n_0}\sum_{i=1}^N(1-W_i)\tilde{\mathbf{r}}_i(Z_i)\right) + o_p(1)
		\end{align*}for potentially distinct constants $\{\tilde{\mathbf{r}}_i(0), \tilde{\mathbf{r}}_i(1)\}_{i=1}^N$ which may change with $N$ that satisfy $(1/N)\sum_{i=1}^N(\tilde{\mathbf{r}}_i(1)-\tilde{\mathbf{r}}_i(0)) = 0$ along with Assumptions \ref{asm: means and covs stabilize} and \ref{asm: bounded fourth moment}. Observe that the difference in means estimator satisfies these conditions with $\mathbf{r}_i(z) = \tilde{\mathbf{r}}_i(z) = \mathbf{y}_i(z) - z\bar{\bm{\tau}}$ for $z\in\{0,1\}$. Let $\bm{\tau}_{r i} = \mathbf{r}_i(1) - \mathbf{r}_i(0)$. Let $\Sigma_{r(z)}, \Sigma_{\tau_r}, \Sigma_{r(z)x}, \Sigma_{\tau_r x}$ be the analogues of $\Sigma_{y(z)}$, $\Sigma_{\tau}$, $\Sigma_{y(z)x}$ and $\Sigma_{\tau x}$ for $z\in\{0,1\}$, and let the same hold with $r$ replaced by $\tilde{r}$. Define $V^{(r)}$ and $\tilde{V}^{(\tilde{r})}$ as the analogues of $V$ and $\tilde{V}$, computed now based upon $\mathbf{r}(z)$ and $\tilde{\mathbf{r}}(z)$ instead of $\mathbf{y}(z)$ and $\tilde{\mathbf{y}}(z)$ for $z\in\{0,1\}$.
      Consider a test statistic for the weak null of the form $\breve{T}(\yobs, \bfZ) = f_{\hat{\xi}}(\sqrt{N}\breve{\tau})$ for some $f_\eta$ satisfying Condition \ref{cond: f} and $\hat{\xi}$ satisfying Condition \ref{cond: xi}, and suppose that there exists a covariance estimator $\breve{V}$ satisfying Condition \ref{cond: var} with $V$ and $\tilde{V}$ replaced by $V^{(r)}$ and $\tilde{V}^{(r)}$.  The Gaussian prepivoted test statistic is
\begin{align*}
\breve{G}(\mathbf{y}(\mathbf{Z}), \mathbf{Z}) = \frac{\gamma^{(d+k)}_{\bm{0}, \breve{V}}\left\{(\mathbf{a}, \mathbf{b})^\T: f_{\hat{\xi}}(\mathbf{a}) \leq \breve{T}(\yobs, \mathbf{Z})\; \wedge\; \phi(\mathbf{b})=1\right \}}{\gamma^{(k)}_{\bm{0}, \breve{V}_{\delta \delta}}\left\{\mathbf{b}:  \phi(\mathbf{b})=1\right\}}
\end{align*}
\begin{theorem}\label{thm: brevegaussian}
Suppose that Neyman's null, $H_N: \bar{\bm{\tau}} = 0$, holds. Then, under the described restrictions on $\breve{T}(\yobs, \mathbf{Z})$ and $\breve{V}$ and under Assumption \ref{asm: non-degen sampling limit} and with Assumptions \ref{asm: means and covs stabilize} and \ref{asm: bounded fourth moment} applied to $\mathbf{r}_i(z)$, $z=\{0,1\}$, $\breve{G}(\mathbf{y}(\mathbf{Z}), \mathbf{Z})$ converges in distribution to a random variable $\breve{U}$ taking values in $[0,1]$ satisfying
\begin{align*}
    \mathbb{P}(\breve{U} \leq t) \geq t,
\end{align*}
for all $t \in [0,1]$. Furthermore, the distribution $\hatPerm_{\Breve{G}}(t)$ satisfies $\hatPerm_{\breve{G}}(t)\overset{p}{\rightarrow} t$ for all $t\in [0,1]$.\looseness=-1
\end{theorem}

In the appendix, we illustrate that the regression-adjusted average treatment effect estimator and its corresponding estimated variance presented in \citet{agnosticRegAdj} can be viewed in this form. As a result, Theorem \ref{thm: brevegaussian} provides justification for the use of the prepivoted randomization distribution of a regression-adjusted estimator.

\section{Simulation studies}\label{sec: simulations}
\subsection{Studentization and prepivoting in rerandomized designs}\label{subsec: studentization in rerandomization}
In the $b$th of $B$ iterations, we draw, for $i=1,...,N$, covariates $iid$ as
\begin{align*} \mathbf{x}_i \overset{iid}{\sim} \mathcal{N}\left(0, \begin{pmatrix}1.0&0.8&0.2\\ 0.8&1&0.3\\0.2 & 0.3 & 1\end{pmatrix}\right). \end{align*}
Given these covariates, we draw $r_i(0)$ and $r_i(1)$ as
\begin{align*}
r_i(0) = \mathbf{x}_i^\T\bm{\beta}_0 + \epsilon_i(0);\;\;
r_i(1) = \mathbf{x}_i^\T\bm{\beta}_1 + \epsilon_i(1),
\end{align*}
where $\beta_0 = -(6.4, -4.0, -2.4)$, $\beta_1 = c(0.2, 0.4, 0.6)^\T$, $\epsilon_i(0)\overset{iid}{\sim}-\mathcal{E}(1)+1$, $\epsilon_i(1)\overset{iid}{\sim}-\mathcal{E}(1/10) + 10$, $\epsilon_i(0)$ independent of $\epsilon_i(1)$, and $\mathcal{E}(\lambda)$ representing an exponential distribution with rate $\lambda$.

We form the potential outcomes under treatment and control in two distinct ways, one in which the sharp null holds and one in which only the weak null holds:
\begin{itemize}
\item[]\textit{Sharp Null}: $y_i(1) = y_i(0) = r_i(1)$
\item[]\textit{Weak Null}: $y_i(1) = r_i(1)$; $y_i(0) = r_i(0) + \bar{r}(1) - \bar{r}(0)$
\end{itemize}
 Of the $N$ individuals, $n_1=0.2N$ receive the treatment and $n_0 = 0.8N$ receive the control. We use a Mahalanobis-based rerandomized design, with criterion $\phi(\sqrt{N}\hat{\bm{\delta}}) = \indicatorFunction{(\sqrt{N}\hat{\bm{\delta}})^\T V_{\delta \delta}^{-1}(\sqrt{N}\hat{\bm{\delta}}) \leq 1 }$. This balance criterion reduces the cardinality of $\Omega$ by roughly 80\% relative to a completely randomized design. For each $b$, we draw a single $\mathbf{Z}\in\Omega$, and proceed with inference using the reference distribution of the following test statistics under the incorrect assumption that the sharp null holds:
\begin{enumerate}
\item Absolute difference in means, unstudentized
\item Absolute difference in means, studentized
\item Gaussian prepivoting the absolute difference in means, studentized
\end{enumerate}
The true reference distributions assuming the sharp null are replaced by Monte-Carlo estimates with $1000$ draws from $\Omega$ for each $b$, and the desired Type I error rate is $\alpha = 0.05$. We also perform inference using the large-sample reference distribution for the absolute studentized difference in means in a rerandomized design; see \citet{asymptoticsOfRerand} for more details. As a covariance estimator $\hat{V}$, we use the conventional unpooled covariance estimator for $(\sqrt{N}\hat{\tau}, \sqrt{N}\hat{\bm{\delta}})^\T$ in a two-sample design. For the generative models reflecting the sharp and weak nulls, we proceed with both $N=50$ and $N=1000$ to compare performance in small and large sample regimes. For each $N$, we conduct $B=5000$ simulations.\looseness=-1

\begin{table}
	\caption{\label{tab: rerand} Inference after rerandomization. The rows describe the simulation settings, which vary between the sharp and weak nulls holding and between small and large sample sizes. The first three columns represent the performance of randomization tests assuming the sharp null hypothesis and using the unstudentized absolute difference in means, absolute studentized difference in means, and Gaussian prepivoted absolute difference in means respectively to perform inference. The last column is a large-sample test which is asymptotically valid for the weak null, based upon \citet{asymptoticsOfRerand}. The desired Type I error rate in all settings is  $\alpha = 0.05$.}
	\centering
	\begin{tabular}{l c c c c c c c}
	&\multicolumn{3}{c}{Randomization Test} 			& Large-Sample\\
									&	No Stu. & 	Stu. 	& 	Pre. 	& 						\\
	Sharp, $N=50$ 	& 	0.053	& 	0.050	&		0.051	& 	0.110			\\ 
	Sharp, $N=1000$ &  	0.052 &		0.048 &		0.048 & 	0.054			\\ 
	Weak, $N=50$ 		&  	0.073 & 	0.114 & 	0.037 & 	0.058 		\\ 
	Weak, $N= 1000$ &  	0.070	& 	0.083 &		0.018 &		0.019			\\ 
	\end{tabular}
\end{table}

\begin{figure}[ht]
	\centering
	\includegraphics[scale=.5]{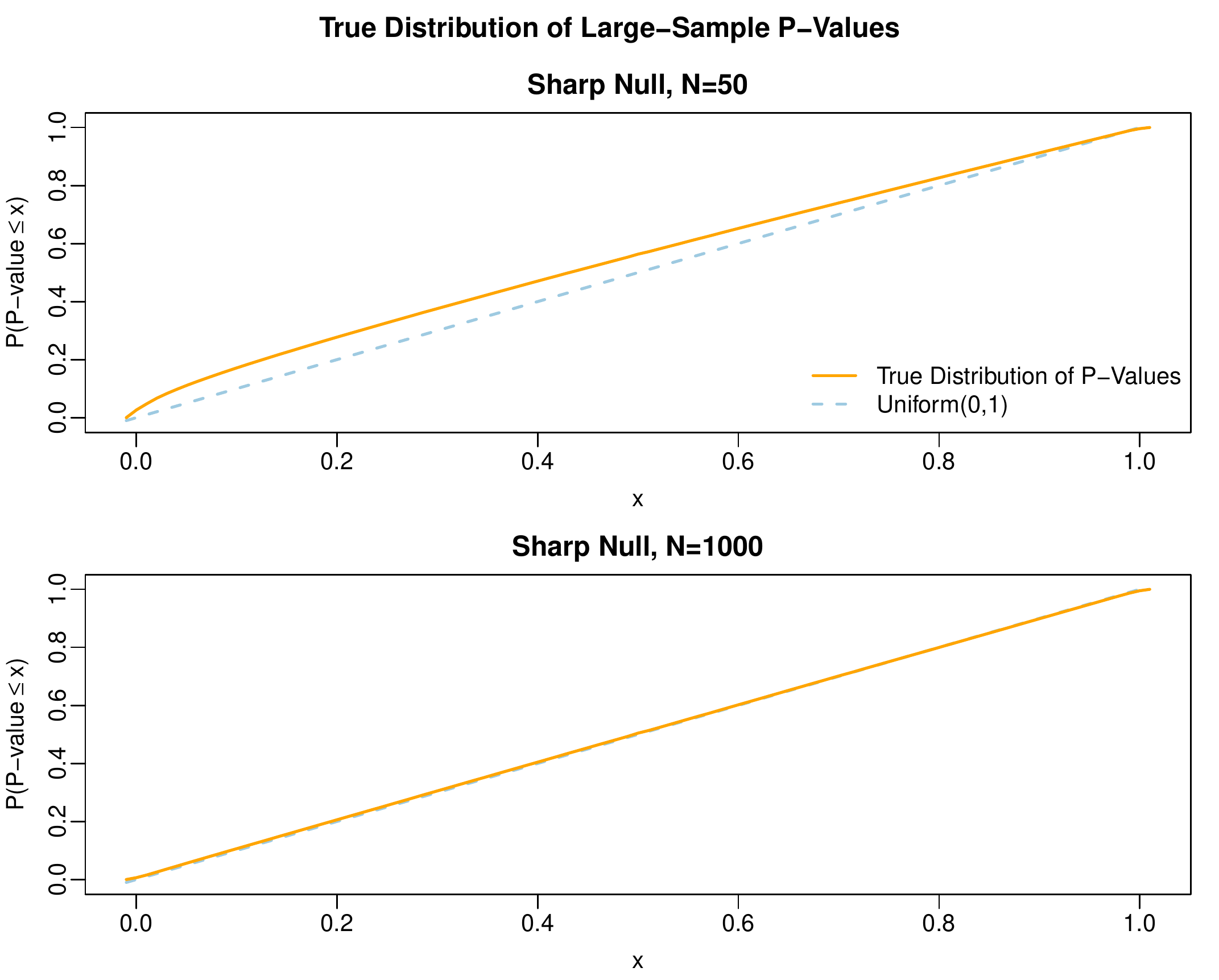}
	\caption{Randomization distribution of the large-sample $p$-values under the sharp null (solid) compared to a standard uniform distribution (dashed) at $N=50$ (top) and $N=1000$ (bottom). At $N=50$, it is more likely to observe a small $P$-value than what the uniform distribution would suggest, yielding the inflated Type I error rate.}
	\label{fig: rerand_ls}
\end{figure}

\begin{figure}[ht]
	\centering
	\includegraphics[scale=.5]{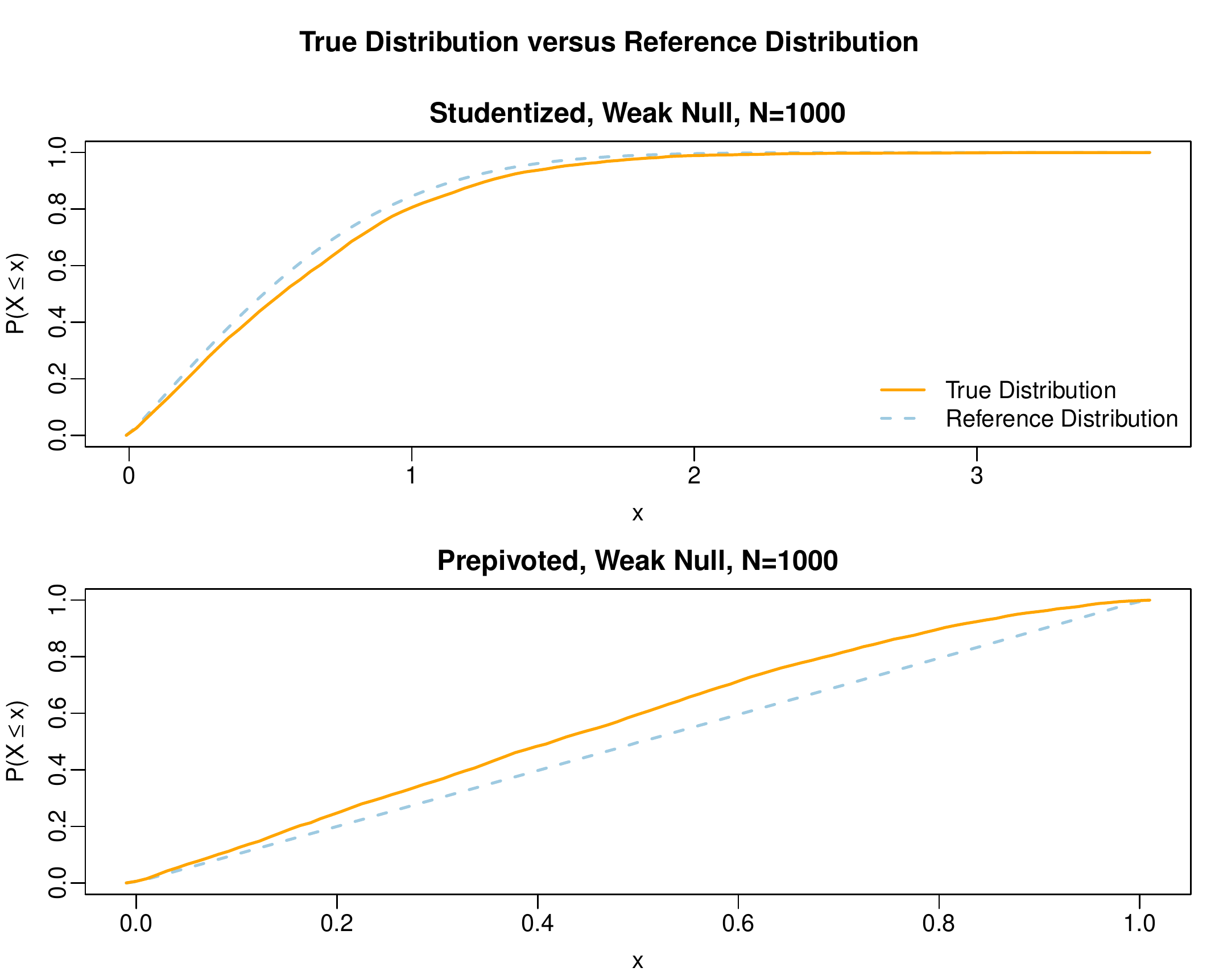}
	\caption{True randomization distribution under the weak null (solid) versus the reference distribution assuming the sharp null (dashed) for the studentized (top) and Gaussian prepivoted (bottom) test statistics with a rerandomized design. To yield valid randomization tests under the weak null, the solid line needs to lie above the dotted line, such that the solid line attributes less mass in the right tail than the dotted line does}
	\label{fig: rerand_stu_pre}
\end{figure}

Table \ref{tab: rerand} contains the results of the simulation study. Under the sharp null with $N=50$, we see the benefits of using a randomization test: the randomization tests based upon the unstudentized, studentized, and prepivoted absolute difference in means all resulted in a Type I error rate of \textcolor{black}{0.05} (up to noise from the Monte-Carlo simulation) as desired. Contrast this with the large-sample test, which had an estimated Type I error rate of \textcolor{black}{0.110} under the sharp null hypothesis. Figure \ref{fig: rerand_ls} explains the deficiency of the large-sample test by comparing the true distribution for the large-sample $p$-values to the standard uniform distribution. As is seen, at $N=50$ small $p$-values are more likely to occur than what the standard uniform would predict at any point $t \in [0,1]$, resulting in inflated Type I error rates. By $N=1000$, the asymptotic approximation performs much better, as the true distribution of $p$-values lies on top of the standard uniform. Gaussian prepivoting uses 1 minus these large-sample $p$-values as the test statistic whose randomization distribution is enumerated, such that the solid line in Figure \ref{fig: rerand_ls} reflects 1 minus the randomization distribution of the Gaussian prepivoted test statistic. As Gaussian prepivoting uses a randomization test under the sharp null, the solid line also reflects the reference distribution employed for performing inference. That these coincide is a consequence of the sharp null holding, such that the randomization tests are exact tests for any sample size.

Under the weak null, we see in Table \ref{tab: rerand} that even at $N=1000$, the unstudentized and studentized randomization tests erroneously assuming the sharp null have inflated Type I error rates. This pattern will persist even asymptotically, as in this simulation setup these test statistics are not asymptotically sharp-dominant. This may come as a surprise, as in completely randomized designs studentizing \textit{does} furnish asymptotic sharp dominance. As evidenced here, the impact of covariates on the limiting distribution in rerandomized experiments invalidates studentization as a mechanism for attaining asymptotic sharp dominance. Figure \ref{fig: rerand_stu_pre} illustrates this in the case of the studentized test statistic. We see in the top panel that the true distribution function for the studentized test statistic lies below that of the reference distribution assuming the sharp null, such that the right-tail probabilities are larger for the true randomization distribution than they are for the reference distribution. This yields anti-conservative inference. We see in the bottom panel of Figure \ref{fig: rerand_stu_pre} that through use of Gaussian prepivoting, asymptotic conservativeness has been restored: the true randomization distribution of the prepivoted test statistic is stochastically dominated by the reference distribution assuming the sharp null, as predicted by Theorem \ref{thm: gaussian}. We further see that the cumulative distribution assuming the sharp null is converging to the distribution function of the standard uniform (a straight line between 0 and 1), again reflecting Theorem \ref{thm: gaussian}. Table \ref{tab: rerand} further shows that the Gaussian prepivoted test and the large-sample test have very similar rejection rates at $N=1000$, reflecting the asymptotic equivalence of the two methods under the weak null. \looseness=-1

\subsection{A comparison of multivariate tests}
In each iteration $b = 1,...,B$, we draw $\{\mathbf{r}_i(1)\}_{i=1}^N$ and $\{\mathbf{r}_i(0)\}_{i=1}^N$ independent from one another and $iid$ from mean zero equicorrelated multivariate normals of dimension $k=25$ with marginal variances one. The correlation coefficients governing $\mathbf{r}_i(1)$ and $\mathbf{r}_i(0)$ are 0 and 0.95 respectively. We will have two simulation settings, one each for the sharp and weak null:
\begin{itemize}\item[]\textit{Sharp Null}: $\mathbf{y}_i(1) = \mathbf{y}_i(0) = \mathbf{r}_i(1)$.
\item[]\textit{Weak Null}: $\mathbf{y}_i(1) = \mathbf{r}_i(1)$; $\mathbf{y}_i(0) = \mathbf{r}_i(0) + \bar{\mathbf{r}}(1) - \bar{\mathbf{r}}(0)$.
\end{itemize}

In both settings, $n_1=0.2N$ individuals receive the treatment and $n_0 = 0.8N$ receive the control. We consider a completely randomized design, and proceed with inference using the reference distribution of the following test statistics under the (erroneous) assumption that the sharp null holds:
\begin{enumerate}
\item Hotelling's $T$-squared, unpooled covariance
\item Hotelling's $T$-squared, pooled covariance
\item Max absolute $t$-statistic, unpooled standard error
\end{enumerate}
For each candidate test, we proceed with the randomization distribution both of the untransformed test statistic and the Gaussian prepivoted test statistic. These tests are conducted using Monte-carlo simulation to generate the reference distributions, with $1000$ draws from $\Omega$ for each iteration $b$. In addition to the two types of randomization tests, we also compute a large-sample $p$-value for each test which is asymptotically valid under the weak null hypothesis. As a covariance estimator $\hat{V}$, we use the conventional unpooled covariance estimator for $\sqrt{N}\hat{\bm{\tau}}$.  For each test, we seek to maintain the Type I error rate at or below $\alpha= 0.05$.  For the generative models reflecting the sharp and weak nulls we proceed with both $N=300$ and $N=5000$ to compare performance as $N$ increases. For each $N$, we conduct $B=5000$ simulations.\looseness=-1

    \begin{table}
        \caption{\label{tab:mult} Inference in completely randomized designs with multiple outcomes. The rows describe the simulation settings, which vary between the sharp and weak nulls holding and between small and large sample sizes. There are three sets of columns, one corresponding to each of the three test statistics under consideration. For each set of columns, the column labeled ``FRT'' represents the Fisher randomization test using that test statistic. The column labeled ``Pre.'' instead reflects the Fisher randomization test after applying Gaussian prepivoting to the original test statistic. The last column, labeled ``LS,'' is a large-sample test which is asymptotically valid for the weak null. The desired Type I error rate in all settings is  $\alpha=0.05$.}
        \centering
        \begin{tabular}{l c c c c c c c c c} 
                &\multicolumn{3}{c}{Hotelling, Unpooled} & \multicolumn{3}{c}{Hotelling, Pooled} & \multicolumn{3}{c}{Max $t$-stat}\\
                & FRT & Pre. & LS &  FRT & Pre. & LS & FRT & Pre. & LS\\
                Sharp, $N=300$  &   0.050 & 0.050 & 0.321 & 0.052 & 0.047 & 0.086 & 0.051 & 0.050 & 0.068 \\ 
                Sharp, $N=5000$ &   0.044 & 0.044 & 0.053 & 0.047 & 0.042 & 0.045 & 0.046 & 0.045 & 0.048 \\ 
                Weak, $N=300$   &   0.117 & 0.117 & 0.270 & 0.975 & 0.166 & 0.157 & 0.020 & 0.006 & 0.008 \\ 
                Weak, $N=5000$  &   0.003 & 0.003 & 0.003 & 0.951 & 0.005 & 0.005 & 0.021 & 0.005 & 0.005 \\ 
        \end{tabular}


    \end{table}

Table \ref{tab:mult} gives the estimated Type I error rates for the candidate tests. We first note the poor performance of the large-sample tests under both the sharp and weak null with $N=300$. For instance, the large-sample $p$-values constructed using the unpooled, Hotelling procedure are attained using a $\chi^2_{25}$ distribution and have estimated Type I error rates of 0.321 under the sharp null for $N=300$, and of 0.270 under the weak null for $N=300$ despite the desired control at $\alpha=0.05$. By $N=5000$, the large-sample tests all have estimated Type I error rates approaching the nominal level under the sharp null, and below the nominal level under the weak null.

Naturally, all randomization tests attain (up to Monte-Carlo error) the desired Type I error rate under the sharp null at both $N=300$ and $N=5000$, highlighting the appeal of the randomization tests. Under the weak null, we see that the randomization test based upon the Hotelling $T$-statistic with a pooled covariance fails to control the Type I error rate even at $N=5000$, reflecting that the test statistic is not asymptotically sharp-dominant.  While the randomization test based on the max $t$-statistic controls the Type I error rate in these simulations, this is not guaranteed in general: in the appendix we conduct this simulation at $\alpha=0.25$, where anti-conservativeness of the max $t$-statistic arises.  For both of these test statistics, applying Gaussian prepivoting restores guaranteed asymptotic conservativeness and results in test statistics whose performance closely aligns with the large-sample tests, a reflection of Theorem \ref{thm: gaussian}. For the test based upon Hotelling's $T$ statistic with an unpooled covariance estimator, observe that the Type I error rates for the randomization tests with and without Gaussian prepivoting are identical in all four scenarios tested. As discussed in Example \ref{example: mult} of \S \ref{sec: examples}, this is because Gaussian prepivoting is unnecessary for this particular test statistic: Hotelling's $T$ statistic with an unpooled covariance estimator is already asymptotically sharp-dominant as proven in \citet{RandTestsWeakNulls}. Applying Gaussian prepivoting recovers an equivalent randomization test, furnishing identical $p$-values for any observed outcomes $\mathbf{y}(\mathbf{Z})$ for completely randomized designs.
\color{black}

In the appendix we provide a theoretical analysis of the statistical power of Gaussian prepivoting and include simulations to demonstrate the power in practice.  We also provide analysis of real-world data from the Student Achievement and Retention experiment of \citet{ALO}.

\section{Discussion}\label{sec: conclusion}
	\subsection{An open question: multivariate one-sided testing in finite population causal inference}
	The restrictions on the function $f_{\eta}$ outlined in Condition \ref{cond: f} require a quasi-convex, continuous function that is mirror-symmetric about the origin. This restriction results in convex, mirror-symmetric sublevel sets for $f_\eta$ and facilitates the application of Anderson's theorem, such that dominance in the Loewner order on covariance matrices translates to the stochastic dominance under the weak null. While the restrictions on $f_\eta$ are sensible with two-sided alternatives, they preclude testing directional alternatives because of the mirror symmetry condition. For instance, suppose one wanted to test the null hypothesis $\bar{\tau}_i\leq 0$ for all $i=1,..,d$ versus the alternative that for at least one $i$ $(i=1,..,d)$, $\bar{\tau}_i > 0$. In the univariate case, choosing $T(\yobs, \bfZ) = \hat{\tau}/\hat{V}_{\tau\tau}^{1/2}$ does not provide a valid one-sided test for all $\alpha$. That said, it \textit{does} provide a valid test for $\alpha \leq 0.5$, such that for any reasonable value for $\alpha$ to be deployed in practice a one-sided test is possible. \looseness=-1

	Suppose we have multivariate potential outcomes and consider the test statistic $T_{max}(\yobs, \bfZ) = \max_{1\leq i\leq d}\;\;\hat{\tau}_i/\hat{V}_{\tau\tau,ii}^{1/2}$, with $\hat{V}_{\tau\tau}$ satisfying Condition \ref{cond: var}. Consider the Gaussian prepivoted test statistic $G_{max}(\yobs, \bfZ)$. The following is, to the best of our knowledge, an open question: is it the case that, for any $\alpha\leq 0.5$, $G_{max}$ is asymptotically sharp-dominant, in that $\limsup\;\;\mathbb{E}\{\varphi_{G_{max}}(\alpha)\}\leq \alpha$? Under the assumptions imposed in this work, the answer would be true should the following conjecture on Gaussian comparisons hold:
	\begin{conjecture}
	Let $\mathbf{X} = (X_1,...,X_d)$, and $\mathbf{Y} = (Y_1,...,Y_d)$ be $d$-dimensional multivariate Gaussian vectors, with a common mean $\bm{\mu} = (\mu_1,...,\mu_d)$ but distinct covariances $\Sigma^X$ and $\Sigma^Y$, with $ij$ entries $\sigma_{ij}^X$ and $\sigma_{ij}^Y$, respectively. Let $\gamma_{ij}^X = \mathbb{E}\{(X_i-X_j)^2\}$ and $\gamma_{ij}^Y = \mathbb{E}\{(Y_i-Y_j)^2\}$.  Define $med\left(\max_i Y_i\right)$ as the median of $\underset{1\leq i\leq d}{\max}Y_i$, i.e. the value $a$ such that $\Prob{\underset{1\leq i\leq d}{\max}Y_i\leq a } = 0.5$. Suppose that $\sigma^Y_{ii} \geq \sigma^X_{ii}$ for all $i$ and that $\gamma_{ij}^Y \geq \gamma_{ij}^X$ for all $i,j$. Consider any point $c \geq med\left(\max_i Y_i\right)$. Then,
	\begin{equation*}
\Prob{\underset{1\leq i\leq d}{\max}\;X_i\geq c} \leq (?)\;\; \Prob{\underset{1\leq i\leq d}{\max}\;Y_i\geq c}.
\end{equation*}
	\end{conjecture}
	The conjecture is true in the univariate case. Under the assumptions of this conjecture, the Sudakov-Fernique inequality \citep[Theorem 2.2.5]{adl09} asserts that $\mathbb{E}\{\max_{1\leq i\leq d}\;X_i\}\leq \mathbb{E}\{\max_{1\leq i\leq d}\;Y_i\}$. Should we further assume $\sigma_{ii}^X=\sigma_{ii}^Y$, the result holds for all points $c$ through Slepian's lemma \citep[Theorem 5.1.7]{Slepian, mvtNormal_Tong}. Unfortunately,  a refined result about tail probabilities above the median does not appear to be available in the literature under the conditions outlined in the conjecture. A potential path forward may be a modification of the soft-max proof of the Sudakov-Fernique inequality found in \citet{cha05}.

\subsection{Summary}
In this work, we present a general framework for designing randomization tests that are both exact for Fisher's sharp null and are asymptotically conservative for Neyman's weak null in completely randomized experiments and rerandomized designs. Loosely stated, the approach may be summarized as follows: if one has access to a large-sample test that is asymptotically conservative under Neyman's weak null, then a Fisher randomization test using the $p$-value produced by that large-sample test will maintain asymptotic correctness under the weak null while additionally restoring exactness should the sharp null be true. As the Fisher randomization distribution of these $p$-values converges weakly in probability to a uniform, the resulting randomization test assuming the sharp will have the same large-sample performance under the weak null as large-sample test itself, and will further have the same asymptotic power under local alternatives as the large-sample test.  We show that Gaussian prepivoting exactly recovers several randomization tests known to be valid under the weak null, while providing a general approach to restore asymptotic correctness to randomization tests for a large class of test statistics. Importantly, our framework immediately provides valid randomization tests of the weak null hypothesis in rerandomized designs, absent from the literature until now.

\section{Additional Materials}
\textcolor{black}{Example code to implement Algorithm~\ref{alg: Gaussian prepivoting} is available at \url{https://github.com/PeterLCohen/PrepivotingCode}.} \looseness=-1

\delineate
\begin{center}
    \textbf{\huge{Appendix}}
\end{center}
\appendix

\setcounter{lemma}{0}
\setcounter{assumption}{0}
\setcounter{condition}{0}

\section{Useful lemmas}

\renewcommand{\thetheorem}{\Alph{theorem}}
\renewcommand{\thelemma}{\Alph{lemma}}

\begin{lemma}\label{supp: lem: gaussian measure continuous}
    For any Borel measurable set $B \subseteq \R^{\ell}$, the centered Gaussian measure of $B$ is a continuous function in terms of the covariance parameter.  In other words, $\gamma_{\mathbf{0}, \Sigma}^{\ell}(B)$ is a continuous function of $\Sigma$ over the positive definite cone of $\ell \times \ell$ real matrices with metric induced by the Frobenius norm.
\end{lemma}
\begin{proof}
    Denote the space of positive definite real $\ell \times \ell$ matrices by $S_{++}^{\ell}$; this is a metric space under the metric induced by the Frobenius norm.  Consider a sequence of matrices $\Sigma_{N} \in S_{++}^{\ell}$ for which $\Sigma_{N} \rightarrow \Sigma$. By definition for any Borel measurable set $B \subseteq \R^{\ell}$
    \begin{equation*}
        \gamma_{\mathbf{0}, \Sigma_{N}}^{\ell}(B) = \int_{B}\frac{1}{\sqrt{2\pi}^{\ell}}\frac{1}{\sqrt{\det(\Sigma_{N})}}\exp\left(\frac{-\mathbf{x}^{T}\Sigma_{N}^{-1}\mathbf{x}}{2}\right)\,d\mathbf{x}.
    \end{equation*}
    The function $f(M) = \det(M)^{-1/2}$ is continuous over the positive definite cone of $\ell \times \ell$ matrices.  Thus, since $\Sigma_{N} \rightarrow \Sigma$ it follows that \begin{equation}\label{supp: eqn: scalar limit}
        \frac{1}{\sqrt{2\pi}^{\ell}}\frac{1}{\sqrt{\det(\Sigma_{N})}} \rightarrow \frac{1}{\sqrt{2\pi}^{\ell}}\frac{1}{\sqrt{\det(\Sigma)}}.
    \end{equation}

    All that remains to be examined is the limiting behavior of
    \begin{equation*}
        \int_{B}\exp\left(\frac{-\mathbf{x}^{\T}\Sigma_{N}^{-1}\mathbf{x}}{2}\right)\,d\mathbf{x}.
    \end{equation*}
    For $(M, \mathbf{x}) \in S_{++}^{\ell} \times \R^{\ell}$ the function $g(M, \mathbf{x}) = \exp(-\mathbf{x}^{\T}M^{-1}\mathbf{x}/2)$ is a jointly continuous of both $\mathbf{x}$ and $M$.  Consequently, for all $\mathbf{x} \in \R^{\ell}$
    \begin{equation*}
        \exp\left(\frac{-\mathbf{x}^{\T}\Sigma_{N}^{-1}\mathbf{x}}{2}\right) \rightarrow \exp\left(\frac{-\mathbf{x}^{\T}\Sigma^{-1}\mathbf{x}}{2}\right).
    \end{equation*}

    Since all convergent sequences are bounded there exits a positive semidefinite matrix $\Sigma_{*}$ that is greater than or equal to (in the Loewner partial order) all $\Sigma_{N}$.  Thus, $\Sigma_{N}^{-1} \succeq \Sigma_{*}^{-1}$ for all $N \in \N$.  Consequently $g(\Sigma_{N}, \mathbf{x})$ is dominated by $g(\Sigma_{*}, \mathbf{x})$ for all $N$ and all $\mathbf{x} \in \R^{\ell}$.  Thus, Lebesgue's dominated convergence theorem implies that
    \begin{equation}\label{supp: eqn: integral limit}
        \int_{B}\exp\left(\frac{-\mathbf{x}^{\T}\Sigma_{N}^{-1}\mathbf{x}}{2}\right)\,d\mathbf{x} \rightarrow \int_{B}\exp\left(\frac{-\mathbf{x}^{\T}\Sigma^{-1}\mathbf{x}}{2}\right)\,d\mathbf{x}.
    \end{equation}

    Combining \eqref{supp: eqn: scalar limit} and \eqref{supp: eqn: integral limit} implies that for all sequences $\Sigma_{N} \in S_{++}^{\ell}$ such that $\Sigma_{N} \rightarrow \Sigma$
    \begin{equation}\label{supp: eqn: sequential continuity}
        \gamma_{\mathbf{0}, \Sigma_{N}}^{\ell}(B) \rightarrow \gamma_{\mathbf{0}, \Sigma}^{\ell}(B).
    \end{equation}

    \eqref{supp: eqn: sequential continuity} establishes that $\gamma_{\mathbf{0}, \Sigma}^{\ell}(B)$ is a sequentially continuous function of the parameter $\Sigma$ for all $\Sigma \in S_{++}^{\ell}$.  Sequential continuity in a metric space is equivalent to continuity \citep[Theorem 5.31]{metricSpacesContinuity}; so $\gamma_{\mathbf{0}, \Sigma}^{\ell}(B)$ is a continuous function of the parameter $\Sigma$ for all $\Sigma \in S_{++}^{\ell}$.

\end{proof}

\delineate

\begin{lemma}\label{supp: lem: prepivoting is a continuous function}
    Let a function $f_{\eta}(\cdot)$ satisfy Condition~\ref{supp: cond: f} and let $\phi(\cdot)$ satisfy Condition~\ref{supp: cond: phi}.  Let the matrix $V \in S_{++}^{(d + k)}$ be defined blockwise as
    \begin{equation*}
        V = \begin{pmatrix} V_{\tau \tau} & V_{\tau \delta}\\ V_{\delta \tau} & V_{\delta \delta}\end{pmatrix}.
    \end{equation*}
    Consider
    \begin{equation*}
        h(V, \eta, x) = \frac{\gamma^{(d+k)}_{\bm{0}, V}\left\{(\mathbf{a}, \mathbf{b})^\T: f_{\eta}(\mathbf{a}) \leq x\; \wedge\; \phi(\mathbf{b})=1\right \}}{\gamma^{(k)}_{\bm{0},V_{\delta \delta}}\left\{\mathbf{b}:  \phi(\mathbf{b})=1\right\}}.
    \end{equation*}
    The function $h(V, \eta, x)$ is a continuous function of $V$, $\eta$, and $x$ jointly.
\end{lemma}
\begin{proof}
   Because $V$ is positive definite, $V_{\delta \delta}$ must be as well.  Thus, the centered Gaussian measure $\gamma^{(k)}_{\bm{0},V_{\delta \delta}}(\cdot)$ is non-singular.  Furthermore, because $\phi(\cdot)$ satisfies Condition~\ref{supp: cond: phi}, the set $\left\{\mathbf{b}:  \phi(\mathbf{b})=1\right\}$ is Borel measurable with positive Lebesgue measure.  Thus, $\gamma^{(k)}_{\bm{0},V_{\delta \delta}}\left\{\mathbf{b}:  \phi(\mathbf{b})=1\right\}$ is positive.  Moreover, Lemma~\ref{supp: lem: gaussian measure continuous} establishes that $\gamma^{(k)}_{\bm{0},V_{\delta \delta}}\left\{\mathbf{b}:  \phi(\mathbf{b})=1\right\}$ is a continuous function of $V_{\delta \delta}$, and thus of $V$.

   Consider the function $\kappa: (\eta, x) \mapsto \left\{(\mathbf{a}, \mathbf{b})^\T: f_{\eta}(\mathbf{a}) \leq x\; \wedge\; \phi(\mathbf{b})=1\right \}$.  The range of $\kappa$ is the set of Borel measurable sets in $\R^{(d + k)}$.  This space can be imbued with the metric\footnote{Actually, $d(B, B')$ is a pseudo-metric unless one considers two sets equal if their symmetric difference is of measure zero.  We take this convention since -- by absolute continuity -- sets of Lebesgue measure zero are of Gaussian measure zero as well.}
   $$d(B, B') = \mu(B \triangledown B')$$
   where $B \triangledown B'$ is the symmetric difference of $B$ and $B'$ and $\mu(\cdot)$ is Lebesgue measure on $\R^{(d + k)}$; this is sometimes called the \textit{ Fr\'{e}chet–Nikod\'{y}m–Aronszajn distance} \citep[Section 4]{setMetrics}.  Consider sequences of $\eta_{N}$ which converge to $\eta$ and $x_{N}$ which converge to $x$.  Let $B_{N}$ denote $\kappa(\eta_{N}, x_{N})$; the set-theoretic limit of $B_{N}$ converges to $B$ under $d(B, B')$.  This relies upon the continuity of $f_{\eta}(\mathbf{a})$ in $\eta$.  Thus, $\kappa$ is sequentially continuous in $\eta$ and $x$ jointly.  Sequential continuity in a metric space is equivalent to continuity \citep[Theorem 5.31]{metricSpacesContinuity}; so $\kappa$ is jointly continuous in $\eta$ and $x$.

   The numerator of $h(V, \eta, x)$ is the composition of $\gamma^{(d+k)}_{\bm{0}, V}(B)$ with $\kappa(\eta, x)$; the former is continuous in $V$ by Lemma~\ref{supp: lem: gaussian measure continuous} and in $B$ by the absolute continuity of Gaussian measure, and the later is jointly continuous in $\eta$ and $x$.  Thus, the numerator of $h(V, \eta, x)$ is jointly continuous in $V$, $\eta$, and $x$.  Since the denominator of $h(V, \eta, x)$ is a continuous function of $V$ that is always positive, the function $h(V, \eta, x)$ itself is a jointly continuous function of $V$, $\eta$, and $x$.
\end{proof}

\section{Proof of main results}
\subsection{A reminder: assumptions and conditions}

As in the main text, we rely on some regularity conditions which we restate below for convenience.

\setcounter{assumption}{0}
\renewcommand{\theassumption}{\arabic{assumption}}
\begin{assumption}\label{supp: asm: non-degen sampling limit}
	The proportion $n_{1} / N$ limits to $p \in (0, 1)$ as $N \rightarrow \infty$.  
\end{assumption}

\begin{assumption}\label{supp: asm: means and covs stabilize}
	All finite population means and covariances have limiting values for both the potential outcomes and the covariates. For instance, $\lim_{N \rightarrow \infty}\bar{\mathbf{y}}(z) = \bar{\mathbf{y}}_{\infty}(z)$ for $z \in \{0, 1\}$ and $\lim_{N \rightarrow \infty}\Sigma_{y(1)} = \Sigma_{y(1),\infty}$.
\end{assumption}

\textcolor{black}{Assumption~3 of the main text is known to be stronger than necessary for certain results.  Here we split Assumption~3 into two parts.  We do this to show exactly which results can rely upon a weaker assumption and which results seem to rely upon the stronger assumption.}

\edef\oldassumption{\the\numexpr\value{assumption}+1}

\setcounter{assumption}{0}
\renewcommand{\theassumption}{\oldassumption (\alph{assumption})}
\begin{assumption}\label{supp: asm: worst case vanishes}
 The worst-case squared distance from the average potential outcome  is $o(N)$; i.e.,
	\begin{equation*}
		\lim_{N \rightarrow \infty}\max_{\substack{z \in \{0, 1\}\\ j\in \{1,...,d\}}}\max_{i \in \{1, \ldots, N\}}\frac{\left({y}_{ij}(z) - \bar{{y}}_j(z)\right)^{2}}{N} = 0.
	\end{equation*}
	Further, the above holds for the covariates with ${x}_{ij}$ replacing ${y}_{ij}(z)$ above for $j=1,..,k$.
\end{assumption}
\begin{assumption}\label{supp: asm: bounded fourth moment}
 There exists some $C < \infty$ for which, for all $z\in\{0,1\}$, all $j=1,..,d$ and all $N$,
	\begin{equation*}
		\frac{\sum_{i = 1}^{N}\left(y_{ij}(z) - \overline{y}_j(z)\right)^{4}}{N} < C
	\end{equation*}
	Further, the above holds for the covariates with ${x}_{ij}$ replacing ${y}_{ij}(z)$ above for $j=1,..,k$.
\end{assumption}

\textcolor{black}{Assumption~\ref{supp: asm: bounded fourth moment} implies Assumption \ref{supp: asm: worst case vanishes} \citep[Proposition 1]{RandTestsWeakNulls}.  Assumption~\ref{supp: asm: bounded fourth moment} is made at times for mathematical convenience to simplify the analysis of certain random distributions; though it remains an open question whether such results hold under weaker assumptions.}

\delineate

Recall the $\tilde{\mathbf{y}}(Z_i)$ is defined as
\begin{align*}
\tilde{\mathbf{y}}_i(Z_i) &= \mathbf{y}_i(Z_i) - Z_i\bm{\bar{\tau}},
\end{align*} such that $\tilde{\mathbf{y}}(\mathbf{Z}) = \mathbf{y}(\mathbf{Z}) - \mathbf{Z}\bm{\bar{\tau}}^\T$. Further recall the following conditions from the main text.

\begin{condition}
    \label{supp: cond: phi}$\phi: \mathbb{R}^k \mapsto \{0,1\}$ is an indicator function such that the set $M = \{\mathbf{b}: \phi(\mathbf{b}) = 1\}$ is closed, convex, and mirror-symmetric about the origin (i.e., $\mathbf{b} \in M \Leftrightarrow -\mathbf{b} \in M$) with non-empty interior.
\end{condition}

\begin{condition}\label{supp: cond: f}
    For any $\eta \in \Xi$, $f_{\eta}(\cdot): \mathbb{R}^d\mapsto \mathbb{R}_+$ is continuous, quasi-convex, and nonnegative with $f_\eta(\mathbf{t}) = f_\eta(-\mathbf{t})$ for all $\mathbf{t}\in \mathbb{R}^d$. Furthermore, $f_{\eta}(\mathbf{t})$ is jointly continuous in $\eta$ and $\mathbf{t}$.
\end{condition}

\begin{condition}\label{supp: cond: xi}
With $\mathbf{W}, \mathbf{Z}$ independent and each uniformly distributed over $\Omega$,\begin{align*}
\hat{\xi}(\tilde{\mathbf{y}}(\mathbf{Z}), \mathbf{Z})  \overset{p}{\rightarrow}\xi;\;\;
\hat{\xi}(\tilde{\mathbf{y}}(\mathbf{Z}), \mathbf{W}) \overset{p}{\rightarrow}\tilde{\xi}, 
\end{align*}for some $\xi, \tilde{\xi} \in \Xi$.
\end{condition}

\begin{condition}\label{supp: cond: var}
  With $\mathbf{W}, \mathbf{Z}$ independent, both uniformly distributed over $\Omega$, and for some  $\Delta \succeq 0$, $\Delta \in \mathbb{R}^{d\times d}$,\begin{align*} \hat{V}(\tilde{\mathbf{y}}(\mathbf{Z}), \mathbf{Z}) - V \overset{p}{\rightarrow} \begin{pmatrix}\Delta & 0_{d,k}\\  0_{k,d} & 0_{k,k} \end{pmatrix};\;\;
\hat{V}(\tilde{\mathbf{y}}(\mathbf{Z}), \mathbf{W}) - \tilde{V} \overset{p}{\rightarrow} 0_{(d+k), (d+k)}.\end{align*}
\end{condition}
Oftentimes in the proofs it will implicitly be assumed that the weak null holds. For that reason, $\hat{\xi}$ and $\hat{V}$ may be written with $\yobs$ as inputs rather than $\tilde{\mathbf{y}}(\mathbf{Z})$. Let $\Omega_{CRE}$ denote the set of allowable treatment allocation vectors $\bfz$ for a completely randomized experiment.  Formally $$\Omega_{CRE} = \left\{ \bfz \in \{0, 1\}^{N} \,\Bigg|\, \sum_{i = 1}^{N}z_{i} = n_{1} \right\}.$$

\subsection{A remark on limiting distributions for rerandomized designs}
    A completely randomized experiment can be considered a rerandomized experiment for which $\phi(\cdot)$ is identically one.  This trivial balance criterion satisfies Condition~\ref{supp: cond: phi}.\footnote{When no covariate information is collected, this statement is then vacuous, but in such a context the comparison to a rerandomized experiment is also missing.}  When $\phi(\cdot)$ is not vacuous, the interesting case for rerandomized designs, limiting distributions in completely randomized designs continue to provide corresponding limiting distributions after rerandomization under Condition \ref{supp: cond: phi}.

    By the finite population central limit theorem of \citet{FiniteCLT}, $\sqrt{N}(\hat{\bm{\tau}}  - \bar{\bm{\tau}} , \hat{\bm{\delta}})^\T$ is asymptotically distributed according to a mean-zero multivariate Gaussian distribution with covariance matrix $V$, where
	\begin{align*}
    	V\;\; &= \begin{pmatrix} V_{\tau \tau} & V_{\tau \delta}\\ V_{\delta \tau} & V_{\delta \delta}\end{pmatrix};\\
		V_{\tau\tau} & = p^{-1}\Sigma_{y(1),\infty} + (1 - p)^{-1}\Sigma_{y(0),\infty} - \Sigma_{\tau,\infty};       \\
		V_{\delta\delta}       & = \{p(1 - p)\}^{-1}\Sigma_{x,\infty};\\
		V_{\tau \delta}   & = p^{-1}\Sigma_{y(1)x,\infty} + (1-p)^{-1}\Sigma_{y(0)x,\infty} = V_{\delta\tau}^\T.
	\end{align*}

   Conditioning according to appropriate balance holding requires that $V_{\delta\delta} \succ 0$.  In this case, the conditional probability of $\sqrt{N}\hat{\tau}(\tilde{\mathbf{y}}(\mathbf{Z}), \bfZ) \in B$ subject to $\phi(\sqrt{N}\hat{\delta}(\mathbf{x}, \bfZ)) = 1$ limits to
    \begin{equation}\label{supp: eqn: conditional measure randomization}
        \frac{\gamma^{(d+k)}_{\bm{0}, V}\left\{(\mathbf{a}, \mathbf{b})^\T: \mathbf{a} \in B\; \wedge\; \phi(\mathbf{b})=1\right \}}{\gamma^{(k)}_{\bm{0}, V_{\delta \delta}}\left\{\mathbf{b}:  \phi(\mathbf{b})=1\right\}}
    \end{equation}
    for any Borel measurable set $B$.

    Likewise, by Proposition 1 of the main text and Lemma 4.1 of \citet{lowDim_HighDim}, the conditional probability of $\sqrt{N}\hat{\tau}(\tilde{\mathbf{y}}(\mathbf{Z}), \bfW) \in B$ subject to $\phi(\sqrt{N}\hat{\delta}(\mathbf{x}, \bfW)) = 1$ limits to
    \begin{equation}\label{supp: eqn: conditional measure permutation}
        \frac{\gamma^{(d+k)}_{\bm{0}, \tilde{V}}\left\{(\mathbf{a}, \mathbf{b})^\T: \mathbf{a} \in B \; \wedge\; \phi(\mathbf{b})=1\right \}}{\gamma^{(k)}_{\bm{0}, \tilde{V}_{\delta \delta}}\left\{\mathbf{b}:  \phi(\mathbf{b})=1\right\}}.
    \end{equation}

 The finite population central limit theorem of \citet{FiniteCLT} and Proposition 1 are statements about joint convergence in distribution for the scaled differences in means for the observed outcomes and for the covariates. Passing to convergence in distribution conditional upon $\phi(\sqrt{N}\hat{\delta}(\mathbf{x}, \bfZ)) = 1$ or $\phi(\sqrt{N}\hat{\delta}(\mathbf{x}, \bfW)) = 1$ described in \eqref{supp: eqn: conditional measure randomization} and \eqref{supp: eqn: conditional measure permutation} rests upon the continuity-set argument used in the proof of Proposition A1 in \citet{asymptoticsOfRerand}.  Condition~\ref{supp: cond: phi} guarantees that such arguments remain valid: in particular the set $M$ defined within Condition~\ref{supp: cond: phi} is of positive Lebesgue measure. This allows results for completely randomized designs to provide asymptotics when $\Omega_{CRE}$ is replaced with $\Omega$ from a general rerandomized design.

\subsection{Proof of Theorem 1}

\renewcommand{\thetheorem}{\arabic{theorem}}
\setcounter{theorem}{0}
\begin{theorem}\label{supp: thm: gaussian}
    Suppose we have either a completely randomized design or a rerandomized design with balance criterion $\phi$ satisfying Condition \ref{supp: cond: phi}. Suppose ${T}(\yobs, \mathbf{Z})$ is of the form $f_{\hat{\xi}}(\sqrt{N}\hat{\bm{\tau}})$ for some $f_\xi$ and $\hat{\xi}$ satisfying Conditions \ref{supp: cond: f} and \ref{supp: cond: xi}. Suppose further that we employ a covariance estimator $\hat{V}$ satisfying Condition \ref{supp: cond: var} when forming the prepivoted test statistic $G(\yobs, \bfZ)$. Then, under Neyman's null $H_N: \bar{\bm{\tau}} = 0$ and under Assumptions \ref{supp: asm: non-degen sampling limit} - \ref{supp: asm: worst case vanishes}, $G(\mathbf{y}(\mathbf{Z}), \mathbf{Z})$ converges in distribution to a random variable $\tilde{U}$ taking values in $[0,1]$ satisfying
    \begin{align*}
        \mathbb{P}(\tilde{U} \leq t) \geq t,
    \end{align*}
    for all $t \in [0,1]$. Furthermore, strengthening Assumption \ref{supp: asm: worst case vanishes} to Assumption \ref{supp: asm: bounded fourth moment}, the distribution $\hatPerm_G(t)$ satisfies
    \begin{align*}
    \hatPerm_G(t)\overset{p}{\rightarrow} t
    \end{align*}
    for all $t\in [0,1]$.
\end{theorem}
\begin{proof}[Proof of Theorem~\ref{supp: thm: gaussian}]
    A completely randomized experiment can be viewed as a rerandomized experiment for which $\phi(\mathbf{b}) = 1$ for all $\mathbf{b} \in \R^{k}$; this $\phi$ satisfies Condition~\ref{supp: cond: phi}.  As such, the proof below proceeds with general $\phi$ satisfying Condition~\ref{supp: cond: phi} -- making no distinction between rerandomized designs and completely randomized design.




    First, we focus on the randomization distribution of the prepivoted test statistic; in other words, we examine the limiting distribution of $G(\mathbf{y}(\mathbf{Z}), \mathbf{Z})$ under $H_{N}$.  By the finite population central limit theorem of \citet{FiniteCLT} in a completely randomized design or a rerandomized design with $\phi$ satisfying Condition~\ref{supp: cond: phi} the $\sqrt{N}$-scaled difference in means, $\sqrt{N}(\bm{\hat{\tau}}-\bm{\bar{\tau}}, \hat{\bm{\delta}})^{T}$, converges in distribution to $\Normal{0}{V}$ with
    \begin{align*}
	    V\;\; &= \begin{pmatrix} V_{\tau \tau} & V_{\tau \delta}\\ V_{\delta \tau} & V_{\delta \delta}\end{pmatrix};\\
		V_{\tau\tau} & = p^{-1}\Sigma_{y(1),\infty} + (1 - p)^{-1}\Sigma_{y(0),\infty} - \Sigma_{\tau,\infty};       \\
		V_{\delta\delta}       & = \{p(1 - p)\}^{-1}\Sigma_{x,\infty};\\
		V_{\tau \delta}   & = p^{-1}\Sigma_{y(1)x,\infty} + (1-p)^{-1}\Sigma_{y(0)x,\infty} = V_{\delta\tau}^\T.
	\end{align*}

	Furthermore, by Condition \ref{supp: cond: phi} and Corollary A1 of \citet{asymptoticsOfRerand}, we have that for $\bfZ$ instead uniform over $\Omega$ (accounting for the rerandomized design),
	$\sqrt{N}(\hat{\bm{\tau}}-\bm{\bar{\tau}}) \overset{d}{\rightarrow} \mathbf{C}$, where $\mathbf{C}$ follows the distribution of $\mathbf{A}\mid \phi(\mathbf{B}) =1$ for $\mathbf{A}\in \mathbb{R}^d$, $\mathbf{B}\in \mathbb{R}^k$, and $(\mathbf{A}, \mathbf{B})^\T$ multivariate Gaussian with covariance $V$ and mean zero.

	By Condition~\ref{supp: cond: xi} $\hat{\xi}(\tilde{\mathbf{y}}(\mathbf{Z}), \mathbf{Z}) \convP[] \xi$ and by Condition~\ref{supp: cond: var}
	$$\hat{V}(\tilde{\mathbf{y}}(\mathbf{Z}), \mathbf{Z}) \convP[] V + \begin{pmatrix}\Delta & 0_{d,k}\\  0_{k,d} & 0_{k,k} \end{pmatrix} =: \bar{\bar{V}}.$$

	Leveraging Lemma~\ref{supp: lem: prepivoting is a continuous function} and the continuous mapping theorem, under $H_{N}$
	\begin{equation*}
	    h\left(\hat{V}({\mathbf{y}}(\mathbf{Z}), \mathbf{Z}), \hat{\xi}({\mathbf{y}}(\mathbf{Z}), \mathbf{Z}), \sqrt{N}\tauhat \right) \convD[] h\left(V + \begin{pmatrix}\Delta & 0_{d,k}\\  0_{k,d} & 0_{k,k} \end{pmatrix}, \xi, \mathbf{C} \right)
	\end{equation*}
	where $\mathbf{C}$ distributed as before. Unwinding the notation of $h(\cdot, \cdot, \cdot)$ gives that $G(\mathbf{y}(\mathbf{Z}), \mathbf{Z})$ converges in distribution to
	\begin{equation}\label{supp: eqn: randomization limit of prepiv}
	    \frac{\gamma^{(d+k)}_{\bm{0}, \bar{\bar{V}}}\left\{(\mathbf{a}, \mathbf{b})^\T: f_{\xi}(\mathbf{a}) \leq f_{\xi}(\mathbf{C})\; \wedge\; \phi(\mathbf{b})=1\right \}}{\gamma^{(k)}_{\bm{0}, \bar{\bar{V}}_{\delta \delta}}\left\{\mathbf{b}:  \phi(\mathbf{b})=1\right\}}.
	\end{equation}

    If we had known to plug in $V$ for $\hat{V}$, \eqref{supp: eqn: randomization limit of prepiv} would exactly amount to applying the $f_{\xi}$-pushforward of the Gaussian measure $\gamma^{(d+k)}_{\bm{0}, V}$ conditional on $\phi(\bm{b}) = 1$, which would result in a uniform random variable since this is just the asymptotic probability integral transform for $T(\yobs, \bfZ)$ given that $\phi(\sqrt{N}\bm{\hat{\delta}}) = 1$.  However, we do not know $V$ and instead estimate it conservatively using a $\hat{V}$ that satisfies Condition~\ref{supp: cond: var}; this results in the discrepancy between the covariance of $\mathbf{C}$ versus the covariance used in the Gaussian measure $\gamma^{(d+k)}_{\bm{0}, \bar{\bar{V}}}$ in \eqref{supp: eqn: randomization limit of prepiv}.  Consequently, \eqref{supp: eqn: randomization limit of prepiv} amounts to $f_{\xi}$-pushforward of the Gaussian measure $\gamma^{(d+k)}_{\bm{0}, \bar{\bar{V}}}$ in the numerator (the denominator stays the same in both cases since the bottom right block of both $\bar{\bar{V}}$ and $V$ is $V_{\delta\delta}$).  Since $\bar{\bar{V}} \succeq V$, it follows by Lemma 1 of the main text (and Anderson's theorem more generally) that the numerator of \eqref{supp: eqn: randomization limit of prepiv} is no larger than the numerator of \eqref{supp: eqn: randomization limit of prepiv} with $\bar{\bar{V}}$ replaced by $V$.  Then, since applying the $f_{\xi}$-pushforward of the Gaussian measure $\gamma^{(d+k)}_{\bm{0}, V}$ conditional on $\phi(\bm{b}) = 1$ results in a uniform random variable, it follows that \eqref{supp: eqn: randomization limit of prepiv} is stochastically dominated by a uniform random variable from Lemma 2 in the text.  In other words, $G(\mathbf{y}(\mathbf{Z}), \mathbf{Z})$ converges in distribution to a random variable $\tilde{U}$ taking values in $[0,1]$ satisfying $\mathbb{P}(\tilde{U} \leq t) \geq t$ for all $t \in [0,1]$.

    Now we turn our attention to the limiting value of $\hatPerm_{G}(t)$ for any $t$.  Relying upon the result of Proposition 1 in the main text -- which requires Assumptions \SLLNAssumptionsSupp\ -- in a completely randomized design the distribution of $\{\sqrt{N}\hat{\tau}(\tilde{\mathbf{y}}(\mathbf{Z}), \bfW),\sqrt{N}\hat{\delta}(\mathbf{x}, \bfW)\}^\T \mid \mathbf{Z}$ converges weakly in probability to a multivariate Gaussian measure, with mean zero and covariance
    \begin{align*}
    \tilde{V} = \begin{pmatrix} \tilde{V}_{\tau \tau} & \tilde{V}_{\tau \delta}\\ \tilde{V}_{\delta \tau} & \tilde{V}_{\delta \delta}\end{pmatrix}.
    \end{align*}

    By \citet[Lemma 4.1]{lowDim_HighDim}, this is equivalent to
    \begin{equation}\label{supp: eqn: two copies go to normals}
         \begin{bmatrix}\{\sqrt{N}\hat{\tau}(\tilde{\mathbf{y}}(\mathbf{Z}), \bfW),\sqrt{N}\hat{\delta}(\mathbf{x}, \bfW)\}^\T \\
         \{\sqrt{N}\hat{\tau}(\tilde{\mathbf{y}}(\mathbf{Z}), \bfW'),\sqrt{N}\hat{\delta}(\mathbf{x}, \mathbf{W'})\}^\T \end{bmatrix}\convD[] \{(\tilde{\mathbf{A}},\tilde{\mathbf{B}}), (\tilde{\mathbf{A}}',\tilde{\mathbf{B}'})\}^\T
    \end{equation}
    where $\bfZ$, $\bfW$, and $\bfW'$ are independent and uniformly distributed over $\Omega_{CRE}$ and $(\tilde{\mathbf{A}},\tilde{\mathbf{B}})^\T$ and $(\tilde{\mathbf{A}}',\tilde{\mathbf{B}'})^\T$ are independent and identically distributed multivariate Gaussians with mean zero and covariance $\tilde{V}$. By the conditions on $\phi$ outlined in Condition \ref{supp: cond: phi}, we further have that for $\bfZ$, $\bfW$, and $\bfW'$ independently drawn from $\Omega$ (now accounting for the restrictions imposed by rerandomization),

    \begin{equation}\label{supp: eqn: two copies go to cond normals}
         \begin{bmatrix}\sqrt{N}\hat{\tau}(\tilde{\mathbf{y}}(\mathbf{Z}), \bfW) \\
         \sqrt{N}\hat{\tau}(\tilde{\mathbf{y}}(\mathbf{Z}), \bfW) \end{bmatrix}\convD[] ({\mathbf{D}}, {\mathbf{D}}'),
    \end{equation}
    where $({\mathbf{D}}, \mathbf{D}')$ are independent and identically distributed from the conditional distribution of $\tilde{\mathbf{A}}\mid \phi(\tilde{\mathbf{B}}) = 1$.

    By Conditions~\ref{supp: cond: xi} and \ref{supp: cond: var}
    \begin{equation}\label{supp: eqn: permutation distribution constants converge}
        \begin{bmatrix}
            \hat{\xi}(\tilde{\mathbf{y}}(\mathbf{Z}), \bfW)\\
            \hat{V}(\tilde{\mathbf{y}}(\mathbf{Z}), \bfW) \\
            \hat{\xi}(\tilde{\mathbf{y}}(\mathbf{Z}), \bfW')\\
            \hat{V}(\tilde{\mathbf{y}}(\mathbf{Z}), \bfW')
        \end{bmatrix} \convP[] \begin{bmatrix}
                                \tilde{\xi}\\
                                \tilde{V} \\
                                \tilde{\xi}\\
                                \tilde{V}
                            \end{bmatrix}.
    \end{equation}
    Moreover, \eqref{supp: eqn: two copies go to cond normals} and \eqref{supp: eqn: permutation distribution constants converge} hold jointly.  Thus, the continuous mapping theorem implies that
    \begin{gather}
        \begin{bmatrix} h\left(\hat{V}(\tilde{\mathbf{y}}(\mathbf{Z}), \bfW), \hat{\xi}(\tilde{\mathbf{y}}(\mathbf{Z}), \bfW), \sqrt{N}\hat{\tau}(\tilde{\mathbf{y}}(\mathbf{Z}), \bfW)\right) \\
         h\left(\hat{V}(\tilde{\mathbf{y}}(\mathbf{Z}), \bfW'), \hat{\xi}(\yobs, \bfW'), \sqrt{N}\hat{\tau}(\tilde{\mathbf{y}}(\mathbf{Z}), \bfW')\right)
         \end{bmatrix}\nonumber \\
         \verticalConvD \label{supp: eqn: joint permutation result}\\
         \begin{bmatrix} h\left(\tilde{V}, \tilde{\xi}, \mathbf{D}\right) \\
        h\left(\tilde{V}, \tilde{\xi}, \mathbf{D}' \right)
         \end{bmatrix}\nonumber
    \end{gather}
    where $\mathbf{D}$ and $\mathbf{D}'$ are distributed as before.

    Recall that under the weak null, $\tilde{\mathbf{y}}(\mathbf{Z}) = \mathbf{y}(\mathbf{Z})$ and $h\left(\hat{V}(\tilde{\mathbf{y}}(\mathbf{Z}), \bfW), \hat{\xi}(\tilde{\mathbf{y}}(\mathbf{Z}), \bfW), \sqrt{N}\hat{\tau}(\tilde{\mathbf{y}}(\mathbf{Z}), \bfW)\right)$ is precisely $G(\yobs, \bfW)$ as previously defined. Observe that $H(\tilde{V}, \tilde{\xi}, \mathbf{D})$ takes the form

    \begin{equation}\label{supp: eqn: permutation limit of prepiv}
        h\left(\tilde{V}, \tilde{\xi}, \mathbf{D} \right) = \frac{\gamma^{(d+k)}_{\bm{0}, \tilde{V}}\left\{(\mathbf{a}, \mathbf{b})^\T: f_{\tilde{\xi}}(\mathbf{a}) \leq f_{\tilde{\xi}}(\mathbf{D})\; \wedge\; \phi(\mathbf{b})=1\right \}}{\gamma^{(k)}_{\bm{0}, \tilde{V}_{\delta \delta}}\left\{\mathbf{b}:  \phi(\mathbf{b})=1\right\}}.
    \end{equation}

    The logic applied to \eqref{supp: eqn: randomization limit of prepiv} applies similarly to \eqref{supp: eqn: permutation limit of prepiv} except for the fact that the mismatch in the covariance of $\mathbf{C}$ and $\gamma^{(d+k)}_{\bm{0}, \bar{\bar{V}}}$ of \eqref{supp: eqn: randomization limit of prepiv} no longer exists in \eqref{supp: eqn: permutation limit of prepiv} since $\mathbf{D}$ is derived from $(\tilde{\mathbf{A}},\tilde{\mathbf{B}})^\T \sim \Normal{0}{\tilde{V}}$ and the Gaussian measure $\gamma^{(d+k)}_{\bm{0}, \tilde{V}}$ is applied.  As remarked earlier, since the internal covariance matches the external covariance $h\left(\tilde{V}, \tilde{\xi}, \mathbf{D} \right)$ is uniformly distributed over $[0, 1]$.  Applying Lemma 4.1 of \citet{lowDim_HighDim} to \eqref{supp: eqn: joint permutation result} thus implies that $\hatPerm_{G}$ converges weakly in probability to  $\text{Unif}[0, 1]$.  In other words, $\hatPerm_G(t)\overset{p}{\rightarrow} t$ for all $t\in [0,1]$.
\end{proof}

\subsection{Theorem 2}

Theorem 2 reduces to the proof of Theorem~\ref{supp: thm: gaussian} by recognizing the $r_{i}$ and $\tilde{r}_{i}$ as potential outcomes satisfying the required assumptions.  The asymptotically vanishing factor $o_{P}(1)$ in the definitions of $ \sqrt{N}\{\breve{\tau}(\mathbf{y}(\mathbf{Z}), \mathbf{Z}) - \bar{\bm{\tau}}\}$ and $\sqrt{N}\breve{\tau}(\mathbf{y}(\mathbf{Z}) - \mathbf{Z}\bar{\bm{\tau}}^\T, \bfW)$ plays no role in the analysis of their limiting distributions, thereby allowing for application of the same proofs used to show Proposition 1 and Theorem~\ref{supp: thm: gaussian}.

\section{Gaussian prepivoting after regression adjustment}\label{sec: reg adj}
\subsection{Regression adjustment in completely randomized experiments}
		In completely randomized experiments with covariate information, a common practice is to use regression-based estimators for treatment effects to improve efficiency.  Assume that $k$ is fixed and smaller than $N$, and let the potential outcomes be univariate. Define $\tauhat_{reg}(\yobs, \mathbf{Z})$ to be the estimated coefficient on $Z_{i}$ in an ordinary least squares regression of $y_i(Z_{i})$ on $Z_{i}$, $(\mathbf{x}_{i} - \bar{\mathbf{x}})$, and $Z_i(\mathbf{x}_i-\bar{\mathbf{x}})$. \citet{agnosticRegAdj} shows that under suitable regularity conditions, $\tauhat_{reg}$ is $\sqrt{N}$-consistent for $\bar{{\tau}}$ and has an asymptotic variance that is no larger than that of $\hat{{\tau}}$. Importantly, this result holds true \textit{without} assuming that the linear model inspiring $\tauhat_{reg}$ is actually true.



		Let
		$$Q_{1} = \lim_{N \rightarrow \infty}\left(\sum_{ i = 1}^{N}(\mathbf{x}_{i} - \overline{\mathbf{x}})(\mathbf{x}_{i} - \overline{\mathbf{x}})^{\T}\right)^{-1}\left(\sum_{i = 1}^{N}(\mathbf{x}_{i} - \overline{\mathbf{x}})^\T ({y}_i(1) - \overline{y}(1))\right)$$
		be the limit of the OLS slopes for potential outcome under treatment regressed upon covariates, and define $Q_{0}$ analogously for the potential outcomes under control.  The population level treatment residuals based upon the limiting slopes are then defined as
		\begin{align*}\varepsilon_{i}(1) &=
		({y}_i(1) - \overline{y}(1)) - (\mathbf{x}_{i} - \overline{\mathbf{x}})^{\T}Q_{1};\\
		\varepsilon_i(0) &= ({y}_i(0) - \overline{y}(0)) - (\mathbf{x}_{i} - \overline{\mathbf{x}})^{\T}Q_{0}.\end{align*}
		  Let $\tilde{Q} = pQ_{1} + (1 - p)Q_{0}$ and further define
		\begin{align}\label{eqn: mixed slope residuals}
    	    \tilde{\varepsilon}_{i}(1) &=
                            	({y}_i(1) - \overline{y}(1)) - (\mathbf{x}_{i} - \overline{\mathbf{x}})^{\T}\tilde{Q}; \\
                            	\tilde{\varepsilon}_i(0) &= ({y}_i(0) - \overline{y}(0)) - (\mathbf{x}_{i} - \overline{\mathbf{x}})^{\T}\tilde{Q}.\nonumber
    	\end{align}
\setcounter{proposition}{1}
 \begin{proposition}\label{supp: prop: reg adj vs residuals}
 Suppose Assumption \ref{supp: asm: non-degen sampling limit} holds, and suppose further that Assumptions \ref{supp: asm: means and covs stabilize} and \ref{supp: asm: bounded fourth moment}  hold for the potential outcomes and covariates. Then,
     \begin{align*}
         \sqrt{N}\left\{\tauhat_{reg}(\yobs, \mathbf{Z}) - \bar{\bm{\tau}}\right\} &= \sqrt{N}\left(\frac{1}{n_1}\sum_{i=1}^NZ_i\varepsilon_i(Z_i) - \frac{1}{n_0}\sum_{i=1}^N(1-Z_i)\varepsilon_i(Z_i)\right) + o_p(1)\\
         \sqrt{N}\left\{\tauhat_{reg}(\yobs - \mathbf{Z}\bar{\bm{\tau}}^\T, \mathbf{W})\right\} &= \sqrt{N}\left(\frac{1}{n_1}\sum_{i=1}^NW_i\tilde{\varepsilon}_i(Z_i) - \frac{1}{n_0}\sum_{i=1}^N(1-W_i)\tilde{\varepsilon}_i(Z_i)\right) + o_p(1)
     \end{align*}
 \end{proposition}

 Let $\hat{\varepsilon}_i(\yobstilde, \mathbf{W})$ be the $i$th sample residual from a regression of $\yobstilde$ on $W_{i}$, $(\mathbf{x}_{i} - \bar{\mathbf{x}})$, and $W_i(\mathbf{x}_i-\bar{\mathbf{x}})$. Using the sample residuals $\hat{\varepsilon}_{i}(\yobstilde, \bfW)$ form the variance estimators

\begin{align*}
    \hat{\sigma}^2_{0}(\tilde{\mathbf{y}}(\bfZ), \bfW) &= \frac{1}{n_0 - 1}\sum_{i = 1}^{N}(1 - W_i)\left\{\hat{\varepsilon}_i(\tilde{\mathbf{y}}(\bfZ), \mathbf{W}) - \frac{1}{n_0}\sum_{j = 1}^{N}(1 - W_j)\hat{\varepsilon}_j(\tilde{\mathbf{y}}(\bfZ), \mathbf{W})\right\}^2\\
    \hat{\sigma}^2_{1}(\tilde{\mathbf{y}}(\bfZ), \bfW) &= \frac{1}{n_1 - 1}\sum_{i = 1}^{N}W_i\left\{\hat{\varepsilon}_i(\tilde{\mathbf{y}}(\bfZ), \mathbf{W}) - \frac{1}{n_1}\sum_{j = 1}^{N}W_j\hat{\varepsilon}_j(\tilde{\mathbf{y}}(\bfZ), \mathbf{W})\right\}^2
\end{align*}

For the $\hat{\varepsilon}_{i}(\tilde{\mathbf{y}}(\mathbf{Z}), \bfZ)$'s form $\hat{\sigma}^2_{0}(\tilde{\mathbf{y}}(\mathbf{Z}), \bfZ)$ and $\hat{\sigma}^2_{1}(\tilde{\mathbf{y}}(\mathbf{Z}), \bfZ)$ analogously but replace $\bfW$ with $\bfZ$.

Consider the variance estimators
\begin{align*}
    \hat{V}_{reg}(\tilde{\mathbf{y}}(\mathbf{Z}), \bfZ) &= \frac{N}{n_1}\hat{\sigma}^2_{1}(\tilde{\mathbf{y}}(\mathbf{Z}), \bfZ) + \frac{N}{n_0}\hat{\sigma}^2_{0}(\tilde{\mathbf{y}}(\mathbf{Z}), \bfZ)\\
    \hat{V}_{reg}(\tilde{\mathbf{y}}(\bfZ), \mathbf{W}) &=\frac{N}{n_1}\hat{\sigma}^2_{1}(\tilde{\mathbf{y}}(\bfZ), \bfW) + \frac{N}{n_0}\hat{\sigma}^2_{0}(\tilde{\mathbf{y}}(\bfZ), \bfW).
\end{align*}
Observe that $\hat{\sigma}^2_{j}({\mathbf{y}}(\mathbf{Z}), \bfZ) = \hat{\sigma}^2_{j}(\tilde{\mathbf{y}}(\mathbf{Z}), \bfZ)$ for $j=0,1$ regardless of whether or not the weak null holds, but that $\hat{\sigma}^2_{j}(\tilde{\mathbf{{y}}}(\mathbf{Z}), \bfW)\neq \hat{\sigma}^2_{j}({\mathbf{{y}}}(\mathbf{Z}), \bfW)$ unless the weak null holds.


\begin{proposition}
    $\hat{V}_{reg}(\tilde{\mathbf{y}}(\mathbf{Z}), \mathbf{W})$ satisfies Condition \ref{supp: cond: var} with $V_{\tau\tau}$ replaced by $V^{(\varepsilon)}_{\tau\tau}$ and $\tilde{V}_{\tau\tau}$ replaced by $\tilde{V}^{(\tilde{\varepsilon})}_{\tau\tau}$. The particular form of $\Delta$, the degree to which $\hat{V}_{reg}(\tilde{\mathbf{y}}(\mathbf{Z}), \mathbf{Z})$ is asymptotically conservative, is
    \begin{align*}\Delta &= \underset{N\rightarrow \infty}{\lim} \frac{1}{N}\sum_{i=1}^N(\tau_i - \bar{\tau} -  (\mathbf{x}_i - \mathbf{x})^\T(Q_1-Q_0))^2.
        \end{align*}
\end{proposition}
 By Theorem 2, one may apply Gaussian prepivoting to $\sqrt{N}\hat{\tau}_{reg}$ using $\hat{V}_{reg}$ and any function $f_{\hat{\xi}}$ satisfying Condition \ref{supp: cond: f} and \ref{supp: cond: xi}; for instance, take $f_{\hat{\xi}}(\sqrt{N}\hat{\tau}_{reg}) = \sqrt{N}|\hat{\tau}_{reg}|$. Note that other asymptotically equivalent forms for $\hat{V}_{reg}$ to the one given here exist. For example, Section 5 of \citet{agnosticRegAdj} suggests using the sandwich variance estimator corresponding to $\hat{\tau}_{reg}$.

\subsection{Proof of Proposition 2}
We begin with the following Lemma:
{\renewcommand{\thelemma}{\Alph{lemma}}
}
 \begin{lemma}\label{supp: prop: residual properties implicitly hold}
     If Assumptions \ref{supp: asm: means and covs stabilize}	and \ref{supp: asm: worst case vanishes} hold for the potential outcomes and covariates, then Assumptions \ref{supp: asm: means and covs stabilize}	and \ref{supp: asm: worst case vanishes} hold for the collection of $\varepsilon_{i}(z)$.  Likewise, if Assumptions \ref{supp: asm: means and covs stabilize}	and \ref{supp: asm: bounded fourth moment} hold for the potential outcomes and covariates, then Assumptions \ref{supp: asm: means and covs stabilize} and \ref{supp: asm: bounded fourth moment} hold for the collection of $\tilde{\varepsilon}_i(z)$.
 \end{lemma}
 \begin{proof}
    For each $N$, expanding by the definition of $\varepsilon_{i}(1)$ yields
    \begin{align*}
         \bar{\varepsilon}(1) &= N^{-1}\sum_{i = 1}^{N}\left(({y}_i(1) - \overline{y}(1)) - (\mathbf{x}_{i} - \overline{\mathbf{x}})^\T Q_{1} \right) = 0;\\
         \Sigma_{\varepsilon(1)} &= (N-1)^{-1}\sum_{i=1}^N\left({y}_i(1) - \overline{y}(1) - (\mathbf{x}_{i} - \overline{\mathbf{x}})^\T Q_{1} \right)^2.
    \end{align*}
  By inspection, Assumption \ref{supp: asm: means and covs stabilize} holds for the collection of $\varepsilon_{i}(1)$ so long as the potential outcomes and covariates satisfy Assumption \ref{supp: asm: means and covs stabilize}.  Similar proofs establish Assumption \ref{supp: asm: means and covs stabilize} for $\varepsilon_{i}(0)$, $\tilde{\varepsilon}_i(0)$, and $\tilde{\varepsilon}_i(1)$.

    Suppose that Assumption \ref{supp: asm: worst case vanishes} holds for the potential outcomes and covariates. Then
	\begin{gather}
	    \lim_{N \rightarrow \infty}\max_{\substack{z \in \{0, 1\}}}\max_{i \in \{1, \ldots, N\}}\frac{\left({y}_{i}(z) - \bar{{y}}(z)\right)^{2}}{N} = 0\label{supp: eqn: y limit}\\
	    \text{and}\nonumber\\
	    \lim_{N \rightarrow \infty}\max_{j\in \{1,...,k\}}\max_{i \in \{1, \ldots, N\}}\frac{\left({x}_{ij} - \bar{{x}}_j\right)^{2}}{N} = 0.\nonumber
	\end{gather}
	As a consequence of the second statement,
	\begin{equation*}
	    \lim_{N \rightarrow \infty}\max_{i \in \{1, \ldots, N\}}\frac{\sum_{j = 1}^{d}\left({x}_{ij} - \bar{{x}}_j\right)^{2}}{N} = 0
	\end{equation*}
    and so, by the Cauchy-Schwarz inequality,
    \begin{multline}\label{supp: eqn: xQ limit}
	    \lim_{N \rightarrow \infty}\max_{i \in \{1, \ldots, N\}}\frac{\left(\mathbf{x}_{i}^\T Q_{1}- \bar{\mathbf{x}}^\T Q_{1}\right)^{2}}{N}
	    \leq \lim_{N \rightarrow \infty}\max_{i \in \{1, \ldots, N\}}\frac{||Q_{1}||_{2}^{2}\sum_{j = 1}^{d}\left({x}_{ij} - \bar{{x}}_{j}\right)^{2}}{N} = 0.
	\end{multline}

	Because $\left(({y}_i(1) - \overline{y}(1)) - (\mathbf{x}_{i} - \overline{\mathbf{x}})^\T Q_{1} \right)^{2} \leq 2({y}_i(1) - \overline{y}(1))^{2} + 2\left((\mathbf{x}_{i} - \overline{\mathbf{x}})^\T Q_{1} \right)^{2}$ it follows from \eqref{supp: eqn: y limit} and \eqref{supp: eqn: xQ limit} that Assumption \ref{supp: asm: worst case vanishes} holds for the collection of $\varepsilon_{i}(z)$.

	Now suppose that Assumption \ref{supp: asm: bounded fourth moment} holds for the potential outcomes and covariates: there exists some $C < \infty$ for which, for all $z\in\{0,1\}$ and all $N$,
	\begin{gather*}
		\frac{\sum_{i = 1}^{N}\left(y_{ij}(z) - \overline{y}_j(z)\right)^{4}}{N} < C \quad \forall j=1,..,d \\
		\text{and}\\
		\frac{\sum_{i = 1}^{N}\left({x}_{ij} - \overline{{x}}_j\right)^{4}}{N} < C \quad \forall j=1,..,k.
	\end{gather*}

	Modifying the argument from above to accommodate $\tilde{Q}$ instead of $Q_{1}$ and applying H\"{o}lder's inequality gives the desired result.  Specifically,  H\"{o}lder's inequality implies that
	\begin{equation*}
	    \left(\mathbf{x}_{i}^\T \tilde{Q}- \bar{\mathbf{x}}^\T \tilde{Q}\right)^{4} \leq C_Q ||\mathbf{x}_{i} - \overline{\mathbf{x}}||^{4}_{4}.
	\end{equation*}
	where $C_Q$ is a constant that does not change with $N$ and depends only upon $\tilde{Q}$.  Combining this inequality with Assumption~\ref{supp: asm: bounded fourth moment} on the potential outcomes then gives that Assumption \ref{supp: asm: bounded fourth moment} holds for the collection of $\tilde{\varepsilon}_i(z)$.
\end{proof}

We split the proof of Proposition~\ref{supp: prop: reg adj vs residuals} into two: Proposition~\ref{supp: prop: reg adj vs residuals_randomization} and Proposition~\ref{supp: prop: reg adj vs residuals_permutation}.
\setcounter{proposition}{2}
\edef\oldproposition{\the\numexpr\value{proposition}}
\setcounter{proposition}{0}
\renewcommand{\theproposition}{\oldproposition (\alph{proposition})}
\begin{proposition}\label{supp: prop: reg adj vs residuals_randomization}
    \begin{equation*}
        \sqrt{N}\left\{\tauhat_{reg}(\yobs, \mathbf{Z}) - \bar{{\tau}}\right\} = \sqrt{N}\left(\frac{1}{n_1}\sum_{i=1}^NZ_i\varepsilon_i(Z_i) - \frac{1}{n_0}\sum_{i=1}^N(1-Z_i)\varepsilon_i(Z_i)\right) + o_p(1)
    \end{equation*}
\end{proposition}
\begin{proof}
    By Lemma A.3 of \citet{agnosticRegAdj},
    \begin{equation*}
        \tauhat_{reg}(\yobs, \mathbf{Z}) - \bar{{\tau}} = \frac{1}{n_1}\sum_{i=1}^NZ_i\hat{\varepsilon}_i(\yobs, \bfZ) - \frac{1}{n_0}\sum_{i=1}^N(1-Z_i)\hat{\varepsilon}_i(\yobs, \bfZ)
    \end{equation*}
    where the sample residuals $\hat{\varepsilon}_i(\yobs, \bfZ)$ are derived from the regression of $y_i(Z_{i})$ on $Z_{i}$, $(\mathbf{x}_{i} - \bar{\mathbf{x}})$, and $Z_i(\mathbf{x}_i-\bar{\mathbf{x}})$.  Let $\hat{Q}_{1}(\yobs, \bfZ)$ be the sample slope coefficient in the OLS regression of $y_{i}(Z_{i})$ on $\mathbf{x}_{i}$ in the group of individuals for which $Z_{i} = 1$; similarly, let $\hat{Q}_{0}(\yobs, \bfZ)$ be the sample slope coefficient in the population OLS regression of $y_{i}(Z_{i})$ on $\mathbf{x}_{i}$ in the group of individuals for which $Z_{i} = 0$ \citep{agnosticRegAdj}.  

    Define
    \begin{align*}
        \hat{\varepsilon}_{i}(1) &= ({y}_i(1) - \overline{y}(1)) - (\mathbf{x}_{i} - \overline{\mathbf{x}})^\T \hat{Q}_{1}(\yobs, \bfZ);\\
		\hat{\varepsilon}_i(0) &= ({y}_i(0) - \overline{y}(0)) - (\mathbf{x}_{i} - \overline{\mathbf{x}})^\T \hat{Q}_{0}(\yobs, \bfZ);
	\end{align*}
	these are random and depend upon $\bfZ$.  The sample residual $\hat{\varepsilon}_i(\yobs, \bfZ)$ is $\hat{\varepsilon}_{i}(Z_{i})$.

	By standard OLS theory the slope coefficient matrix $\hat{Q}_{1}(\yobs, \bfZ)$ is defined by \begin{multline*}
		\hat{Q}_{1}(\yobs, \bfZ) = \left(\frac{1}{n_{1} - 1}\sum_{i = 1}^{N}Z_{i}\left(\mathbf{x}_{i} - n_{1}^{-1}\sum_{j = 1}^{N}Z_{j}\mathbf{x}_{j}\right)\left(\mathbf{x}_{i} - n_{1}^{-1}\sum_{j = 1}^{N}Z_{j}\mathbf{x}_{j}\right)^{T} \right)^{-1}\times \\ \left(\frac{1}{n_{1} - 1}\sum_{i = 1}^{N}Z_{i}\left(\mathbf{x}_{i} - n_{1}^{-1}\sum_{j = 1}^{N}Z_{j}\mathbf{x}_{j}\right)\left(y_i(Z_{i}) - n_{1}^{-1}\sum_{j = 1}^{N}Z_{j}y(Z_{j})\right) \right)
	\end{multline*}
	$\hat{Q}_{0}(\yobs, \bfZ)$ is defined analogously.

	By weak laws of large numbers for covariance matrices in finite populations, $\hat{Q}_{0}(\yobs, \bfZ)$ and $\hat{Q}_{1}(\yobs, \bfZ)$ converge in probability to $Q_{0}$ and $Q_{1}$, respectively \citep[Lemma A.5]{agnosticRegAdj}.  Thus,
	\begin{align*}
        \hat{\varepsilon}_{i}(1) -\varepsilon_{i}(1) &= o_{P}(1)\\
		\hat{\varepsilon}_{i}(0) -\varepsilon_{i}(0) &= o_{P}(1)\\
    \end{align*}

	From this, it follows that
	\begin{equation*}
        \sqrt{N}\left\{\tauhat_{reg}(\yobs, \mathbf{Z}) - \bar{\tau}\right\} = \sqrt{N}\left(\frac{1}{n_1}\sum_{i=1}^NZ_i\varepsilon_i(Z_i) - \frac{1}{n_0}\sum_{i=1}^N(1-Z_i)\varepsilon_i(Z_i)\right) + o_p(1)
    \end{equation*}
    This proof closely parallels the logic used in the proof for Theorem 1 of \citet{agnosticRegAdj}.
\end{proof}

Before proving the remaining component of Proposition~\ref{supp: prop: reg adj vs residuals} we provide a convenient lemma.
\delineate

\renewcommand{\theproperty}{\Alph{property}}
Consider a function $g: \Omega\times\Omega \rightarrow \R$.  Let $\bfZ$ and $\bfW$ independently distributed uniformly over $\Omega$.  Define two properties:
\begin{property}\label{supp: property: conditional almost sure conv. in prob}
     The random variable $g(\bfZ, \bfW) \given \bfZ = \bfz$ converges in probability to $c$ for all conditioning sets $\{\bfz\}_{N \in \N}$ except for a set of measure zero.
\end{property}
\begin{property}\label{supp: property: joint conv. in prob}
     The random variable $g(\bfZ, \bfW)$ converges in probability to $c$ with respect to randomness in both $\bfZ$ and $\bfW$.
\end{property}

\begin{lemma}\label{supp: lem: condition convergence in prob to joint convergence in prob}
    Consider a function $g: \Omega\times\Omega \rightarrow \R$.  For $\bfZ$ and $\bfW$ independently distributed uniformly over $\Omega$ Property~\ref{supp: property: conditional almost sure conv. in prob} implies Property~\ref{supp: property: joint conv. in prob}.
\end{lemma}
\begin{proof}
    Assume that Property~\ref{supp: property: conditional almost sure conv. in prob} holds.  Fix $\varepsilon > 0$; then
    \begin{equation}\label{supp: eqn: property A}
        \Prob[\bfW\given \bfZ]{|g(\bfZ, \bfW) - c| \geq \varepsilon \given \bfZ } \convAS[] 0.
    \end{equation}
    Consider $\Prob[\bfW, \bfZ]{|g(\bfZ, \bfW) - c| \geq \varepsilon}$; by the law of total probability
    \begin{align*}
        \Prob[\bfZ, \bfW]{|g(\bfZ, \bfW) - c| \geq \varepsilon} &= \sum_{\bfz \in \Omega} \Prob[\bfW\given \bfZ = \bfz]{|g(\bfZ, \bfW) - c| \geq \varepsilon \given \bfZ = \bfz} \Prob[\bfZ]{\bfZ = \bf\z}\\
        &= \Expectation[\bfZ]{ \Prob[\bfW\given \bfZ]{|g(\bfZ, \bfW) - c| \geq \varepsilon \given \bfZ }}
    \end{align*}

    Since $\Prob[\bfW\given \bfZ]{|g(\bfZ, \bfW) - c| \geq \varepsilon \given \bfZ } \convAS[] 0$ and $ {\Prob[\bfW\given \bfZ]{|g(\bfZ, \bfW) - c| \geq \varepsilon \given \bfZ } \in [0, 1]}$ the bounded convergence theorem implies that
    \begin{equation*}
        \lim_{N \rightarrow \infty} \Expectation[\bfZ]{ \Prob[\bfW\given \bfZ]{|g(\bfZ, \bfW) - c| \geq \varepsilon \given \bfZ }} = \Expectation[\bfZ]{0} = 0
    \end{equation*}

    Thus, $g(\bfZ, \bfW)$ converges in probability to $c$ with respect to randomness in both $\bfZ$ and $\bfW$.
\end{proof}

\begin{proposition}\label{supp: prop: reg adj vs residuals_permutation}
    \begin{equation*}
        \sqrt{N}\left\{\tauhat_{reg}(\yobs - \mathbf{Z}\bar{\tau}, \mathbf{W})\right\} = \sqrt{N}\left(\frac{1}{n_1}\sum_{i=1}^NW_i\tilde{\varepsilon}_i(Z_i) - \frac{1}{n_0}\sum_{i=1}^N(1-W_i)\tilde{\varepsilon}_i(Z_i)\right) + o_p(1)
    \end{equation*}
\end{proposition}
\begin{proof}


    By definition $\tauhat_{reg}(\yobs - \mathbf{Z}\bar{\tau}, \mathbf{W})$ is the estimated coefficient on $W_{i}$ in an ordinary least squares regression of $y_i(Z_{i}) - Z_{i}\bar{\tau}$ on $W_{i}$, $(\mathbf{x}_{i} - \bar{\mathbf{x}})$, and $W_i(\mathbf{x}_i-\bar{\mathbf{x}})$.

    By the same logic that gave rise to Lemma A.3 of \citet{agnosticRegAdj},
    \begin{equation*}
        \tauhat_{reg}(\yobs - \mathbf{Z}\bar{\tau}, \mathbf{W}) = \frac{1}{n_1}\sum_{i=1}^NW_i\hat{\varepsilon}_i(\yobstilde, \bfW) - \frac{1}{n_0}\sum_{i=1}^N(1-W_i)\hat{\varepsilon}_i(\yobstilde, \bfW)
    \end{equation*}
     where the sample residuals $\hat{\varepsilon}_i(\yobstilde, \bfW)$ are derived from the regression of $y_i(Z_{i}) - Z_{i}\bar{\tau}$ on $W_{i}$, $(\mathbf{x}_{i} - \bar{\mathbf{x}})$, and $W_i(\mathbf{x}_i-\bar{\mathbf{x}})$.  Let $\hat{Q}_{1}(\yobstilde, \bfW)$ be the sample slope coefficient in the OLS regression of $y_i(Z_{i}) - Z_{i}\bar{\tau}$ on $\mathbf{x}_{i}$ in the group of individuals for which $W_{i} = 1$; similarly, let $\hat{Q}_{0}(\yobstilde, \bfW)$ be the sample slope coefficient in the OLS regression of $y_i(Z_{i}) - Z_{i}\bar{\tau}$ on $\mathbf{x}_{i}$ in the group of individuals for which $W_{i} = 0$.  For convenience of notation, denote $N^{-1}\sum_{i = 1}^{N} \tilde{y}_{i}(Z_i)$ by $\overline{\tilde{y}}(\mathbf{Z})$. 

    Consequently
    \begin{align*}
        \hat{\varepsilon}_i(\yobstilde, \bfW) = \begin{cases}
     \tilde{y}_{i}(Z_i) - \overline{\tilde{y}}(\mathbf{Z}) - (\mathbf{x}_{i} - \overline{\mathbf{x}})^\T \hat{Q}_{1}(\yobstilde, \bfW); & \text{if } W_{i} = 1\\
     \tilde{y}_{i}(Z_i) - \overline{\tilde{y}}(\mathbf{Z}) - (\mathbf{x}_{i} - \overline{\mathbf{x}})^\T \hat{Q}_{0}(\yobstilde, \bfW); & \text{if } W_{i} = 0.
    \end{cases}
\end{align*}
	these are random and depend upon both $\bfZ$ and $\bfW$.

    By standard OLS theory the slope coefficient matrix $\hat{Q}_{1}(\yobstilde, \bfW)$ is
	\begin{multline*}
		\hat{Q}_{1}(\yobstilde, \bfW) = \left(\frac{1}{n_{1} - 1}\sum_{i = 1}^{N}W_{i}\left(\mathbf{x}_{i} - n_{1}^{-1}\sum_{j = 1}^{N}W_{j}\mathbf{x}_{j}\right)\left(\mathbf{x}_{i} - n_{1}^{-1}\sum_{j = 1}^{N}W_{j}\mathbf{x}_{j}\right)^{\T} \right)^{-1}\times \\ \left(\frac{1}{n_{1} - 1}\sum_{i = 1}^{N}W_{i}\left(\mathbf{x}_{i} - n_{1}^{-1}\sum_{j = 1}^{N}W_{j}\mathbf{x}_{j}\right)\left(\left(y_i(Z_{i}) - Z_{i}\bar{\tau}\right) - n_{1}^{-1}\sum_{j = 1}^{N}W_{j}\left(y_j(Z_{j}) - Z_{j}\bar{\tau} \right)\right) \right)
	\end{multline*}

    In Lemma A.5 of \citet{agnosticRegAdj}, it is shown that the first term of $\hat{Q}_{1}(\yobstilde, \bfW)$ converges in probability to $\Sigma_{x, \infty}^{-1}$.  Now we turn our analysis to the second term of $\hat{Q}_{1}(\yobstilde, \bfW)$; denote this term by $M_{1}(\yobstilde, \bfW)$.

    The centering of the potential outcomes under treatment that occurred when translating $\y_{i}(z)$ to $\tilde{y}_{i}(z)$ does not impact Assumptions~\ref{supp: asm: non-degen sampling limit}, \ref{supp: asm: means and covs stabilize}, \ref{supp: asm: worst case vanishes}, and \ref{supp: asm: bounded fourth moment}.  Thus, the finite population strong law for second moments \citep[Lemma A.3, Part ii]{RandTestsWeakNulls}  applies to the sample covariances
    \begin{gather*}
        \hat{\Sigma}_{\tilde{y}(1)x}  = \frac{1}{n_{1} - 1}\sum_{i = 1}^{N}Z_i\left(\mathbf{x}_i - n_1^{-1}\sum_{j=1}^NZ_j\mathbf{x}_j\right)\left(\tilde{y}_i(1) - n_1^{-1}\sum_{j=1}^NZ_j\tilde{y}_j(1)\right)^\T \\
        \text{and}\\
        \hat{\Sigma}_{\tilde{y}(0)x}  = \frac{1}{n_{0} - 1}\sum_{i = 1}^{N}(1 - Z_i)\left(\mathbf{x}_i - n_0^{-1}\sum_{j=1}^N(1-Z_j)\mathbf{x}_j\right)\left(\tilde{y}_i(0) - n_0^{-1}\sum_{j=1}^N(1 - Z_j)\tilde{y}_j(0)\right)^\T.
    \end{gather*}

    Since the centering of the potential outcomes under treatment that occurred when translating $\y_{i}(z)$ to $\tilde{y}_{i}(z)$ does not impact the above covariance structure, it follows from Lemma A.3 of \citet{RandTestsWeakNulls} that $\hat{\Sigma}_{\tilde{y}(1)x} \convAS[] \Sigma_{y(1)x, \infty}$ and $\hat{\Sigma}_{\tilde{y}(0)x} \convAS[] \Sigma_{y(0)x, \infty}$ (This statement relies upon Assumptions \SLLNAssumptionsSupp.).  Condition on a sequence of treatment allocations $\{\bfZ\}_{N \in \N}$ for the growing sequence of experiments such that $\hat{\Sigma}_{\tilde{y}(1)x} \given \bfZ \rightarrow \Sigma_{y(1)x, \infty}$ and $\hat{\Sigma}_{\tilde{y}(0)x} \given \bfZ \rightarrow \Sigma_{y(0)x, \infty}$; this requirement is met for all $\bfZ$ except for a set of measure zero.

    Fix the treatment allocations $\{\bfZ\}_{N \in \N}$; after this conditioning we are left with fully determined ``imputed potential outcomes'':
    \begin{itemize}
        \item $\{\tilde{y}_{i}(Z_{i})\}_{i = 1}^{N}$ for the ``imputed treatment potential outcomes''
        \item $\{\tilde{y}_{i}(Z_{i})\}_{i = 1}^{N}$ for the ``imputed control potential outcomes''
    \end{itemize}
    The imputed population can be envisioned as the population that an experiment would imagine to exist if she observed outcomes $\tilde{y}(\bfZ)$ and believed that Fisher's sharp null held.  Consider $\bfW$ as a treatment allocation for this imputed population.  Under this interpretation $M_{1}(\yobstilde, \bfW)$ is the sample covariance between covariates and the imputed outcomes observed under ``treatment'' $W_{i} = 1$.  Instead of working with $M_{1}(\yobstilde, \bfW)$, we first focus attention to the underlying quantity that $M_{1}(\yobstilde, \bfW)$ seeks to estimate: the covariance between covariates and the imputed potential outcomes $\{\tilde{y}_{i}(Z_{i})\}_{i = 1}^{N}$; we proceed with analysis based upon a fixed sequence of treatment allocations $\bfZ$.  This quantity is
    \begin{align*}
        \Sigma_{imputed,\, \tilde{y}(1)x} &= \frac{1}{N - 1}\sum_{i = 1}^{N}\left(x_{i} - N^{-1}\sum_{j = 1}^{N}\mathbf{x}_{j} \right)\left(\tilde{y}_{i}(Z_{i}) - N^{-1}\sum_{j = 1}^{N}\tilde{y}_{j}(Z_{j}) \right)^\T\\
        &= \frac{1}{N - 1}\sum_{i \given Z_{i} = 1}\left(\mathbf{x}_{i} - N^{-1}\sum_{j = 1}^{N}\mathbf{x}_{j} \right)\left(\tilde{y}_{i}(1) - N^{-1}\sum_{j = 1}^{N}\tilde{y}_{j}(Z_{j}) \right)^\T\\
        &\quad + \frac{1}{N - 1}\sum_{i \given Z_{i} = 0}\left(\mathbf{x}_{i} - N^{-1}\sum_{j = 1}^{N}\mathbf{x}_{j} \right)\left(\tilde{y}_{i}(0) - N^{-1}\sum_{j = 1}^{N}\tilde{y}_{j}(Z_{j}) \right)^\T.
    \end{align*}
    By the strong laws for the sample means, this shares the same limit as
    \begin{multline*}
        \frac{1}{N - 1}\sum_{i \given Z_{i} = 1}\left(\mathbf{x}_{i} - n_1^{-1}\sum_{j = 1}^{N}Z_j\mathbf{x}_{j} \right)\left(\tilde{y}_{i}(1) - n_1^{-1}\sum_{j=1}^NZ_j\tilde{y}_j(1) \right)^\T\\
        + \frac{1}{N - 1}\sum_{i \given Z_{i} = 0}\left(\mathbf{x}_{i} - n_0^{-1}\sum_{j = 1}^{N}(1-Z_j)\mathbf{x}_{j} \right)\left(\tilde{y}_{i}(0) - n_0^{-1}\sum_{j=1}^N(1 - Z_j)\tilde{y}_j(0) \right)^\T.
    \end{multline*}
    In turn, these two terms can be rewritten as
    \begin{equation*}
        \frac{n_{1} - 1}{N - 1} \hat{\Sigma}_{\tilde{y}(1)x} + \frac{n_{0} - 1}{N - 1} \hat{\Sigma}_{\tilde{y}(0)x}
    \end{equation*}
    which limits to $p\Sigma_{y(1)x, \infty} + (1 - p)\Sigma_{y(0)x, \infty}$ for all $\bfZ$ except for a set of measure zero.  Since the centering of the potential outcomes under treatment that occurred when translating $\y_{i}(z)$ to $\tilde{y}_{i}(z)$ does not impact Assumptions~\ref{supp: asm: non-degen sampling limit}, \ref{supp: asm: means and covs stabilize}, \ref{supp: asm: worst case vanishes}, and \ref{supp: asm: bounded fourth moment} it follows from Lemma 1 of \citet{agnosticRegAdj} that
    $$M_{1}(\yobstilde, \bfW) \given \bfZ \convP[] p\Sigma_{y(1)x, \infty} + (1 - p)\Sigma_{y(0)x, \infty}$$
    almost surely in $\bfZ$; combining this with Lemma~\ref{supp: lem: condition convergence in prob to joint convergence in prob} implies that
    $$M_{1}(\yobstilde, \bfW) \convP[] p\Sigma_{y(1)x, \infty} + (1 - p)\Sigma_{y(0)x, \infty}.$$

    Thus
    $$\hat{Q}_{1}(\yobs, \bfW) \convP[] \Sigma_{x, \infty}^{-1}\left( p\Sigma_{y(1)x, \infty} + (1 - p)\Sigma_{y(0)x, \infty}\right) = \tilde{Q}.$$

    The remainder of the proof proceeds in direct analogy with the proof used for Proposition~\ref{supp: prop: reg adj vs residuals_randomization}.
\end{proof}

\begin{remark}
    The utility of Lemma~\ref{supp: lem: condition convergence in prob to joint convergence in prob} in the proof of Proposition~\ref{supp: prop: reg adj vs residuals_permutation} arose from our choice to analyze $M_{1}(\yobstilde, \bfW)$ through conditioning upon treatment allocation $\bfZ$.  With this conditioning argument, Assumption~\ref{supp: asm: bounded fourth moment} is leveraged to attain strong laws with respect to randomness in $\bfZ$; these guarantee that arguments based upon conditioning on $\bfZ = \bfz$ hold for all but a set of measure zero.  An alternative approach to arrive at the statement $M_{1}(\yobstilde, \bfW) \convP[] p\Sigma_{y(1)x, \infty} + (1 - p)\Sigma_{y(0)x, \infty}$ may be to work unconditionally: appealing to a suitable weak law while allowing for randomness in both $\bfZ$ and $\bfW$.  With an approach of this nature, Assumption~\ref{supp: asm: bounded fourth moment} may be stronger than necessary. 

\end{remark}

\subsection{Proof of Proposition 3}
    First we show that
    \begin{equation}\label{supp: eqn: cond var randomization reg adj}
        \hat{V}_{reg}(\yobs, \bfZ)  - V^{(\varepsilon)}_{\tau\tau} \overset{p}{\rightarrow} \Delta,
    \end{equation} with $\Delta$ defined in the statement of Proposition 3. From the proof of Proposition~\ref{supp: prop: reg adj vs residuals_randomization}
    \begin{multline*}
        \hat{V}_{reg}(\yobs, \bfZ) = \frac{N}{n_1}\frac{1}{n_1 - 1}\sum_{i = 1}^{N}Z_i\left(\varepsilon_{i}(1) - \frac{1}{n_1}\sum_{j = 1}^{N}Z_j\varepsilon_{j}(1)\right)^2 \\
        + \frac{N}{n_0}\frac{1}{n_0 - 1}\sum_{i = 1}^{N}(1 - Z_i)\left(\varepsilon_{i}(0) - \frac{1}{n_0}\sum_{j = 1}^{N}(1 - Z_j)\varepsilon_{j}(0)\right)^2 \\
        + o_{P}(1).
    \end{multline*}
    Since $N/n_1 \rightarrow p^{-1}$ and $N/n_0 \rightarrow (1 - p)^{-1}$ this has the same limit as $N \rightarrow \infty$ as
    \begin{multline*}
        \frac{1}{p}\frac{1}{n_1 - 1}\sum_{i = 1}^{N}Z_i\left(\varepsilon_{i}(1) - \frac{1}{n_1}\sum_{j = 1}^{N}Z_j\varepsilon_{j}(1)\right)^2 \\
        + \frac{1}{1 - p}\frac{1}{n_0 - 1}\sum_{i = 1}^{N}(1 - Z_i)\left(\varepsilon_{i}(0) - \frac{1}{n_0}\sum_{j = 1}^{N}(1 - Z_j)\varepsilon_{j}(0)\right)^2.
    \end{multline*}

    Thus, \eqref{supp: eqn: cond var randomization reg adj} holds by the weak law of large numbers for second moments \citep[Lemma A.1]{agnosticRegAdj} and second part of Theorem 2 from \citet{agnosticRegAdj}.

    Next we show that
    \begin{equation}\label{supp: eqn: cond var permutation reg adj}
        \hat{V}_{reg}(\yobs, \mathbf{W}) - \tilde{V}^{(\tilde{\varepsilon})}_{\tau\tau} \overset{p}{\rightarrow} 0.
    \end{equation}

    By the proof of Proposition~\ref{supp: prop: reg adj vs residuals_permutation}
    \begin{multline}\label{eqn: asymptotically equiv variance}
        \hat{V}_{reg}(\yobs, \bfW) = \frac{N}{n_1}\frac{1}{n_1 - 1}\sum_{i = 1}^{N}W_i\left(\tilde{\varepsilon}_{i}(Z_i) - \frac{1}{n_1}\sum_{j = 1}^{N}W_j\tilde{\varepsilon}_{j}(Z_j)\right)^2 \\
        + \frac{N}{n_0}\frac{1}{n_0 - 1}\sum_{i = 1}^{N}(1 - W_i)\left(\tilde{\varepsilon}_{i}(Z_i) - \frac{1}{n_0}\sum_{j = 1}^{N}(1 - W_j)\tilde{\varepsilon}_{j}(Z_j)\right)^2 \\
        + o_{P}(1).
    \end{multline}


    By conditioning upon $\bfZ$, an argument similar to that used to analyze $M_{1}(\yobstilde, \bfW)$ in the proof of Proposition~\ref{supp: prop: reg adj vs residuals_permutation} can then be applied to compute the probability limits of the first two terms in \eqref{eqn: asymptotically equiv variance} almost surely with respect to the conditioning variable $\bfZ$.  Then leveraging Lemma~\ref{supp: lem: condition convergence in prob to joint convergence in prob} yields that the probability limit is the same when considering randomness in both $\bfZ$ and $\bfW$.  Finally, using $N/n_1 \rightarrow p^{-1}$ and $N/n_0 \rightarrow (1 - p)^{-1}$ yields that

    \begin{equation*}
        \hat{V}_{reg}(\yobs, \bfW) \convP[]\frac{1}{p}\Sigma_{\tilde{\varepsilon}(1), \infty} + \frac{1}{1 - p}\Sigma_{\tilde{\varepsilon}(1), \infty} = \tilde{V}^{{(\tilde{\varepsilon})}}_{{\tau}{\tau}}.
    \end{equation*}



\section{An example for paired designs} \label{supp: sec: paired designs}
    The main text focuses upon rerandomized experimental designs.  Since a completely randomized experiment is simply a rerandomized experiment with trivial balance criterion, the results of the main text automatically apply to completely randomized experiments as well.  However, Gaussian prepivoting is not limited to just these contexts.  Here we illustrate the utility of Gaussian prepivoting for a matched-pair experimental design.  Before describing the exact details of Gaussian prepivoting for paired designs, we prove a generalization of Theorem~\ref{supp: thm: gaussian} and Theorem~2.

    Suppose that $\mathbf{A}(\cdot, \cdot)$ is a function such that under the weak null $H_{N}$ and for some positive definite matrices $V$ and $\tilde{V}$,
    \begin{align}
        \begin{pmatrix} \mathbf{A}(\yobs, \bfZ) \\ \sqrt{N}\hat{\bm{\delta}}(\mathbf{x}, \bfZ) \end{pmatrix} \convD[] \Normal{\mathbf{0}}{\begin{pmatrix} V_{AA} & V_{A\delta} \\ V_{\delta A} & V_{\delta \delta} \end{pmatrix}};\label{eqn: limit cov for A rand}\\
        \begin{pmatrix} \mathbf{A}(\tilde{\mathbf{y}}(\mathbf{Z}), \bfW) \\ \sqrt{N}\hat{\bm{\delta}}(\mathbf{x}, \bfW) \end{pmatrix} \convD[] \Normal{\mathbf{0}}{\begin{pmatrix} \tilde{V}_{AA} & \tilde{V}_{A\delta} \\ \tilde{V}_{\delta A} & \tilde{V}_{\delta \delta} \end{pmatrix}}\label{eqn: limit cov for A reference}.
    \end{align}
    \begin{remark}
       With the introduction of the covariance matrices in \eqref{eqn: limit cov for A rand} and \eqref{eqn: limit cov for A reference} we must modify Condition~\ref{supp: cond: var} slightly.  It now becomes: with $\mathbf{W}, \mathbf{Z}$ independent, both uniformly distributed over $\Omega$, and for some  $\Delta \succeq 0$, $\Delta \in \mathbb{R}^{d\times d}$,
        \begin{align*}
            \hat{V}({\mathbf{y}}(\mathbf{Z}), \mathbf{Z}) - \begin{pmatrix} V_{AA} & V_{A\delta} \\ V_{\delta A} & V_{\delta \delta} \end{pmatrix} &\convP[] \begin{pmatrix}\Delta & 0_{d,k}\\  0_{k,d} & 0_{k,k} \end{pmatrix};\\
            \hat{V}(\tilde{\mathbf{y}}(\mathbf{Z}), \mathbf{W}) - \begin{pmatrix} \tilde{V}_{AA} & \tilde{V}_{A\delta} \\ \tilde{V}_{\delta A} & \tilde{V}_{\delta \delta} \end{pmatrix}  &\convP[] 0_{(d+k), (d+k)}.
        \end{align*}
    \end{remark}
    \setcounter{theorem}{2}
    \begin{theorem}\label{supp: thm: generalized gaussian prepiv}
        Suppose that $T(\yobs, \bfZ) = f_{\hat{\xi}}(\mathbf{A}(\yobs, \bfZ))$ for some $f_\eta$ and $\hat{\xi}$ satisfying Conditions \ref{supp: cond: f} and \ref{supp: cond: xi} .  If we employ a covariance estimator $\hat{V}$ satisfying the revised Condition \ref{supp: cond: var} when forming the prepivoted test statistic $G(\yobs, \bfZ)$ then, under $H_N: \bar{\bm{\tau}} = 0$ and the assumption that \eqref{eqn: limit cov for A rand} holds, $G(\mathbf{y}(\mathbf{Z}), \mathbf{Z})$ converges in distribution to a random variable $\tilde{U}$ taking values in $[0,1]$ satisfying
        \begin{align*}
            \mathbb{P}(\tilde{U} \leq t) \geq t,
        \end{align*}
        for all $t \in [0,1]$. Furthermore, if \eqref{eqn: limit cov for A reference} holds, then the distribution $\hatPerm_G(t)$ satisfies
        \begin{align*}
        \hatPerm_G(t)\overset{p}{\rightarrow} t
        \end{align*}
        for all $t\in [0,1]$.
    \end{theorem}
    \begin{proof}
       The proof of this theorem proceeds exactly as that of Theorem~\ref{supp: thm: gaussian}, but with $\sqrt{N}\hat{\bm{\tau}}(\cdot, \cdot)$ replaced by $\mathbf{A}(\cdot, \cdot)$.  
    \end{proof}

    The structure assumed in Theorem ~\ref{supp: thm: generalized gaussian prepiv} is commonly encountered in finite population causal inference across a host of experimental designs. Armed with Theorem~\ref{supp: thm: generalized gaussian prepiv} we turn to the problem of Gaussian prepivoting in paired designs.  
    Consider a population with $I$ matched pairs of individuals, so that the total population size is $N = 2I$.  Attributes of the $j$\textsuperscript{th} unit in the $i$\textsuperscript{th} pair are subscripted with $ij$; e.g. potential outcomes are $\bfy_{ij}(0)$ and $\bfy_{ij}(1)$.  For simplicity take $d = 1$, though these results are not bound to the univariate case.  For a paired design $\Omega = \Omega_{pair}$ with
    \begin{equation*}
        \Omega_{pair} := \left\{\bfz \in \{0, 1\}^{N} \; \Bigg| \; \sum_{j}\bfz_{ij} = 1 \; \forall\, i = 1, \ldots, I\right\}.
    \end{equation*}
    In words, allowable treatment allocations assign one unit of each pair to treatment and the remaining unit of the pair to control.  The average observed treated-minus-control difference in outcomes is
    \begin{equation*}
        \hat{\bm{\tau}}_{pair}(\yobs, \bfZ) := \frac{1}{I}\sum_{i = 1}^{I}\Bigg(\underbrace{(2Z_{i1}-1)\bigg(y_{i1}(Z_{i1}) - y_{i2}(Z_{i2})  \bigg)}_{\mathscr{T}_{i}(\yobs, \bfZ)}\Bigg),
    \end{equation*}
    with $\mathscr{T}_{i}(\yobs, \bfZ)$ representing the treated-minus-control difference in pair $i$. Subject to Conditions~1 and 2 of \citet{fog18reg} -- which are the paired-design analogues of our Assumptions~2 and 3(b) -- the random variable $\hat{\bm{\tau}}_{pair}(\yobs, \bfZ)$ obeys a finite population central limit theorem. This finite population central limit theorem for $\sqrt{I}\left( \hat{\bm{\tau}}_{pair}(\yobs, \bfZ) - \overline{\bm{\tau}}\right)$ can be derived from Theorem~1 of \cite{fog18reg} by dropping the regression-assisting terms.

    We estimate the variance of $\hat{\bm{\tau}}_{pair}(\yobs, \bfZ)$ via its classical Neyman-style estimator
    \begin{equation*}
        \hat{V}_{pair}(\yobs, \bfZ) := \frac{1}{I-1}\sum_{i = 1}^{I}\Big( \mathscr{T}_{i}(\yobs, \bfZ) - \hat{\bm{\tau}}_{pair}(\yobs, \bfZ)\Big)^{2}.
    \end{equation*}
    \citet{imaiPairsVariance} shows that $\hat{V}_{pair}(\yobs, \bfZ)$ is conservative with respect to the true variance of $\hat{\bm{\tau}}_{pair}(\yobs, \bfZ)$.  Under standard regularity conditions \citep[Appendix Lemma 8]{fog19student} there exists a constant $\nu^{2}  > 0$ such that $\sqrt{N}  \hat{\bm{\tau}}_{pair}(\tilde{\bfy}(\bfZ), \bfW)$ converges in distribution to $\Normal{0}{\nu^{2}}$ and by \citet[Appendix Lemma 11]{fog19student}
    \begin{equation*}
        \hat{V}_{pair}(\tilde{\bfy}(\bfZ), \bfW) := \frac{1}{I-1}\sum_{i = 1}^{I}\Big(\mathscr{T}_{i}(\tilde{\bfy}(\bfZ), \bfW) - \hat{\bm{\tau}}_{pair}(\tilde{\bfy}(\bfZ), \bfW)\Big)^{2} \convP[]  \nu^{2}.
    \end{equation*}

    Select a function $f_{\eta}(\cdot)$ which satisfies Conditions~\ref{supp: cond: f} and \ref{supp: cond: xi}, then form the prepivoted test statistic
    \begin{equation*}
        G_{pair}(\yobs, \bfZ) = \gamma^{(1)}_{0, \hat{V}_{pair}}\left\{\mathbf{a}: {f}_{\hat{\xi}}(\mathbf{a}) \leq {f}_{\hat{\xi}}(\hat{\bm{\tau}}_{pair}(\yobs, \bfZ))\right \}.
    \end{equation*}
    Theorem~\ref{supp: thm: generalized gaussian prepiv} applies to $G_{pair}(\yobs, \bfZ)$ and so prepivoting naturally extends to paired experimental designs.  In fact, for paired designs there are numerous candidates for the variance estimator $\hat{V}$ that extend beyond the Neyman-style estimator $\hat{V}_{pair}$; for examples, see the regression-assisted variance estimators of \citet{fog18reg} or the pairs of pairs estimator discussed in \citet{abadieImbensPairsofPairs} and \citet{fog19encourage}.

\color{black}

\section{Experiments with many treatments}\label{supp: sec: many arms}
Theorems~\ref{supp: thm: gaussian} and 2 are not limited to experiments with only two treatment arms (e.g. treatment versus control).  In this section, we show that these results extend naturally to experiments with an arbitrary finite number of treatment arms. For simplicity we present notation and results for completely randomized designs, but extensions are available to rerandomized designs as in the two-armed case. Special cases of balance criteria for general multi-armed rerandomized designs are discussed in \citet[Section 5.2]{reRandMorganRubin}; for rerandomization in factorial experiments \citet{rerandForFactorial} and \citet{branson2016} provide extensive literature.

 Consider an experiment with $A$ arms. The treatment indicator for each individual, $Z_i$, now takes values in $\{0,\ldots, A-1\}$. The potential outcomes under the various treatment options are $\bfy_{i}(0), \ldots, \bfy_{i}(A - 1)$.  For convenience denote $\{0, \ldots, A - 1\}$ by $[A-1]$.  For fixed values $n_{0}, \dots, n_{A - 1} \in \N$ which sum to $N$ let $\Omega_{CRE, A}$ be the set of treatment allocation vectors
\begin{equation*}
    \Omega_{CRE, A} := \left\{\bfz \in [A - 1]^{N} \,\Bigg|\, \sum_{i = 1}^{N}\indicatorFunction{z_{i} = a} = n_{a}, \;\forall\, a \in [A - 1] \right\}.
\end{equation*}
Modify Assumption~\ref{supp: asm: non-degen sampling limit} to be that $N^{-1}n_{a} \rightarrow p_{a}$ with $p_{a} \in (0, 1)$ for all $a$.  In words, no treatment arm is asymptotically degenerate.  The remaining assumptions are modified to hold for $z \in [A - 1]$ instead of just $z \in \{0, 1\}$.  In the multi-arm setting we redefine Fisher's sharp null to be
\begin{equation*}
     H_{F}^{(A)}:\, \bfy_{i}(0) = \bfy_{i}(1) = \cdots = \bfy_{i}(A - 1)\; \forall\, i = 1, \ldots, N.
\end{equation*}
The corresponding generalization of Neyman's weak null is
\begin{equation*}
     H_{N}^{(A)}:\, N^{-1}\sum_{i = 1}^{N}\bfy_{i}(0) = N^{-1}\sum_{i = 1}^{N}\bfy_{i}(1) = \cdots = N^{-1}\sum_{i = 1}^{N}\bfy_{i}(A - 1).
\end{equation*}
See \citet{din17ls} for discussion of this generalization of the sharp and weak nulls.  Further generalizations of these nulls can be found in \citet{RandTestsWeakNulls}.

Denote the vector of the average observed outcome in treatment group $a$ by
\begin{equation*}
    \hat{\bar{\bfy}}(a) = n_a^{-1}\sum_{i \st Z_i = a}\bfy_{i}(a)
\end{equation*}
and the $d \times A$ matrix of all such averages by
\begin{equation*}
    \hat{\bar{\mathbf{Y}}} = \begin{bmatrix} \hat{\bar{\bfy}}(0) &
    \cdots     \hat{\bar{\bfy}}(A - 1)
                            \end{bmatrix}.
\end{equation*}

Consider a matrix $C_{\bfy}$ of dimensions $A \times d'$ for some $d'$.  We stipulate that this matrices is comprised of column-wise contrasts; i.e., each column contains some non-zero element but sums to zero. In place of $\hat{\bm{\tau}}(\yobs, \bfZ)$ we now turn to the weighted treatment effect estimator
\begin{equation*}
    \hat{\bm{\tau}}_{C}(\yobs, \bfZ) = \text{vec}\left(\hat{\bar{\mathbf{Y}}}C_{\bfy}\right)
\end{equation*}
where the $\text{vec}(\cdot)$ operator reshapes $d$-by-$d'$ matrices to $(dd')$-length vectors by vertically concatenating the columns.  In the classical two-armed experiment $C_{\mathbf{y}} = [-1, 1]^{\T}$ returns the standard difference in means as $\hat{\bm{\tau}}_{C}(\yobs, \bfZ)$.  

We extend the notation $\hat{\Sigma}_{y(a)}$ to denote the sample variance estimator in the population which received treatment arm $a$.  Mimicking the argument of \citet[Section 2.3]{RandTestsWeakNulls} with the natural extension to multivariate outcomes gives that an asymptotically conservative covariance estimator for $\sqrt{N}\text{vec}\left(\hat{\bar{\mathbf{Y}}}\right)$ is
\begin{equation*}
    \hat{D}(\yobs, \bfZ) = \bigoplus_{a \in [A - 1]}\left(\frac{N}{n_{a}}\hat{\Sigma}_{y(a)}\right),
\end{equation*}
where $\oplus$ denotes the direct sum of matrices, resulting in a block-diagonal matrix of dimension $Ad\times Ad$ with $(N/n_{a-1})\hat{\Sigma}_{y(a-1)}$ in the $a$\textsuperscript{th} of $A$ blocks.
\citet{vecFacts} present numerous algebraic properties of the Kronecker product and the vectorization operator; exploiting their equation (6) yields
\begin{equation*}
        \text{vec}\left( \hat{\bar{\mathbf{Y}}}C_{\bfy}\right) = (C_{\bfy}^{\T} \tensor I)\text{vec}(\hat{\bar{\mathbf{Y}}})
\end{equation*}
Consequently, to produce a Neyman-style conservative covariance estimator for $\sqrt{N}\hat{\bm{\tau}}_{C}(\yobs, \bfZ)$ we form
\begin{align*}
    \hat{V}_{C}(\yobs, \bfZ) &:= (C_{\bfy}^{\T} \tensor I)\left(\bigoplus_{a \in [A - 1]}\left(\frac{N}{n_{a}}\hat{\Sigma}_{y(a)}\right) \right)(C_{\bfy}^{\T} \tensor I)^{\T}\\
    &= (C_{\bfy}^{\T} \tensor I) \hat{D}(\yobs, \bfZ)(C_{\bfy}^{\T} \tensor I)^{\T}\\
\end{align*}

Since $\hat{D}(\yobs, \bfZ)$ is an asymptotically conservative covariance estimator for $\sqrt{N}\text{vec}\left(\hat{\bar{\mathbf{Y}}}\right)$, the covariance estimator $\hat{V}_{C}(\yobs, \bfZ)$ is conservative in the flavor of Condition~\ref{supp: cond: var} but with the natural modifications taken to account for our focus on $\sqrt{N}\hat{\bm{\tau}}_{C}(\yobs, \bfZ)$.

An analysis of of $\hat{D}(\yobs, \bfW)$ under $H_{N}^{(A)}$ in the univariate outcomes case is included in Appendix A2 of \citet[Appendix page 6]{RandTestsWeakNulls}.  The extension to multivariate outcomes follows the same reasoning, replacing scalar variance estimators with matrix-valued covariance estimators.  Their analysis takes a perspective conditional upon $\bfZ$, but our Lemma~\ref{supp: lem: condition convergence in prob to joint convergence in prob} ports their results into our unconditional framework.  Furthermore, \citet[Appendix A2]{RandTestsWeakNulls} provides a detailed analysis of the asymptotic behavior of $\sqrt{N}(\hat{\bm{\tau}}_{C}({\bfy}(\bfZ), \bfW)$ conditional upon $\bfZ$ and $H_{N}^{(A)}$.  Using Hoeffding's Lemma -- see for instance \citet[Lemma 4.1]{lowDim_HighDim} -- in an approach mirroring that used in our proof of Theorem~\ref{supp: thm: gaussian} the analysis of \citet{RandTestsWeakNulls} provides an unconditional understanding of $\sqrt{N}(\hat{\bm{\tau}}_{C}({\bfy}(\bfZ), \bfW)$ under the weak null. Combining their results gives that
\begin{equation}\label{eqn: permutation variance estimator exact for multiarm}
    \hat{V}_{C}(\yobs, \bfW) - \mathbb{V}\left(\sqrt{N}(\hat{\bm{\tau}}_{C}({\bfy}(\bfZ), \bfW)\right) \convP[] \mathbf{0}_{(dd') \times (dd')}
\end{equation}
where $\mathbb{V}\left(\sqrt{N}(\hat{\bm{\tau}}_{C}({\bfy}(\bfZ), \bfW)\right)$ denotes the variance of $\sqrt{N}(\hat{\bm{\tau}}_{C}({\bfy}(\bfZ), \bfW)$.


These results lay the basis for applying Gaussian prepivoting to test statistics $T(\yobs, \bfZ)$ of the form $f_{\hat{\xi}}\left(\hat{\bm{\tau}}_{C}(\yobs, \bfZ) \right)$ with $f_{\hat{\xi}}$ satisfying Conditions~\ref{supp: cond: f} and \ref{supp: cond: xi}.  The prepivoted test statistic takes its usual form
\begin{equation*}
    G_{C}(\mathbf{y}(\mathbf{Z}), \mathbf{Z}) = \gamma^{(dd')}_{\bm{0}, \hat{V}_{C}}\left\{\mathbf{a}: {f}_{\hat{\xi}}(\mathbf{a}) \leq f_{\hat{\xi}}\left(\hat{\bm{\tau}}_{C}(\yobs, \bfZ) \right)\right \}.
\end{equation*}
The central limit behavior of $\sqrt{N}\hat{\bm{\tau}}_{C}(\yobs, \bfZ)$ and asymptotic conservativeness of the Neyman-style variance estimator $\hat{V}_{C}(\yobs, \bfZ)$ apply to Theorem~\ref{supp: thm: generalized gaussian prepiv} to show that the true distribution $\Rand_{G_{C}}$ is asymptotically stochastically dominated by the standard uniform distribution.  The central limit behavior of $\sqrt{N}\hat{\bm{\tau}}_{C}(\yobs, \bfW)$ in conjunction with \eqref{eqn: permutation variance estimator exact for multiarm} implies that the reference distribution $\Perm_{G_{C}}$ limits weakly in probability to the standard uniform distribution; thereby furnishing inferences which are exact for $H_{F}^{(A)}$ and asymptotically conservative for $H_{N}^{(A)}$.

The only major difference between the results for $A > 2$ and $A = 2$ is the use of more general finite population central limit theorems for $\hat{\bm{\tau}}_{C}(\yobs, \bfZ)$ in place of those for $\hat{\bm{\tau}}(\yobs, \bfZ)$; such a central limit theorem is given by \citet[Theorem 5]{FiniteCLT}. Through particular choices for $C_{\mathbf{y}}$, $f_\eta$, and $\hat{\xi}$, we can recover the test statistic proposed in \citet{din17ls} for testing the weak null in multi-armed trials while additionally providing an alternative fix to the usual $F$-statistic to restore asymptotic sharp-dominance. Similarly, the same reasoning can be applied to asymptotically linear estimators (cf. Section~7 of the main text) and the proof of Theorem~2 does not change substantially.

\section{Exact and asymptotically valid confidence sets}
Our results readily extend to constructing confidence intervals which are both asymptotically conservative for the sample average treatment effect and exact if a constant treatment effect model holds.  To this end, we first describe how to test the hypotheses $H_{F, \mathbf{c}}:\, \bm{\tau}_{i} = \mathbf{c} \, \forall\, i = 1, \ldots, N$ and $H_{N, \mathbf{c}}:\, \overline{\bm{\tau}} = \mathbf{c}$ for any fixed $\mathbf{c} \in \R^{d}$.  Define $\tilde{\bfy}^{\mathbf{c}}(\bfZ)$ to be $\tilde{\mathbf{y}}^{\mathbf{c}}_i(Z_i) = \mathbf{y}_i(Z_i) - Z_i\mathbf{c}$.  Then by replacing $\hat{\bm{\tau}}(\yobs, \bfZ)$ and $\hat{\bm{\tau}}(\tilde{\bfy}(\bfZ), \bfW)$ with $\hat{\bm{\tau}}(\tilde{\mathbf{y}}^{\mathbf{c}}(\mathbf{Z}), \bfZ)$ and $\hat{\bm{\tau}}(\tilde{\mathbf{y}}^{\mathbf{c}}(\mathbf{Z}), \bfW)$, respectively, in Equation~(4) of the main paper the test for $H_{F}$ and $H_{N}$ developed in Section~5 of the main paper yields a single procedure that is exact for for $H_{F, \mathbf{c}}$ and asymptotically conservative for $H_{N, \mathbf{c}}$.

           Now, first suppose that one is willing to assume a constant effect model and desires a confidence set for the value $\mathbf{c}$ such that $\bm{\tau}_{i} = \mathbf{c}$ for all $i = 1, \ldots, N$.  Second, suppose that one wants a confidence set for the average treatment effect without assuming a constant effect; i.e., a confidence set for the value $\mathbf{c}$ for which $\overline{\bm{\tau}} = \mathbf{c}$.  Fix a confidence level $1-\alpha \in (0, 1)$ and let $C(T, \yobs, \bfZ)$ be the acceptance region for the test of $H_{F, \mathbf{c}}$ conducted at level $\alpha$ based upon the test statistic $T(\cdot, \cdot)$ evaluated over $\mathbf{c}\in \mathbb{R}^d$.

            Theorem~1 in the main text implies that randomization inference using the prepivoted test statistic $G(\cdot, \cdot)$ yields exact tests for $H_{F, \mathbf{c}}$ and asymptotically valid tests for $H_{N, \mathbf{c}}$. Leveraging the duality between hypothesis testing and confidence sets \citep[Theorem 9.2.2]{casellaBerger} implies the following corollary of Theorem 1.

            \begin{corollary}\label{cor: Gaussian prepiv confidence sets}
                Assume that the regularity conditions of Theorem~1 hold and consider the set $C(G, \yobs, \bfZ)$ formed by inverting the randomization test of $H_{F,\mathbf{c}}$ conducted using the prepivoted test statistic $G(\tilde{\bfy}^{\mathbf{c}}(\bfZ), \cdot)$. $C(G, \yobs, \bfZ)$ is both an asymptotically conservative confidence set for the sample average treatment effect and an exact confidence set under the additional assumption of constant treatment effects.
            \end{corollary}

            Corollary~\ref{cor: Gaussian prepiv confidence sets} implies that inverting hypothesis tests based on Gaussian prepivoting yields a single confidence set with both $1 - \alpha$ coverage for under a constant effect model at all finite $N$ and also at least $1 - \alpha$ asymptotic coverage for the sample average treatment effect.  The extension to generating confidence sets based upon asymptotically linear estimators of the form $\breve{T}$ follows similarly but rests upon Theorem~2 instead of Theorem~1.  Consequently, regression adjustment can be incorporated into the confidence set generating procedure.
\color{black}

\section{Additional simulations}
\subsection{The generative model}\label{sec:gen}
Theorem~\ref{supp: thm: gaussian} and its generalizations concern the finite sample and asymptotic Type I error rates of testing $H_{F}$ and $H_{N}$.   Here we provide additional simulations to highlight the potential for anti-conservative inference in the absence of prepivoting and to investigate the statistical power of the Fisher Randomization Test based upon prepivoted test statistics.

Our simulations proceed similarly to those of Section~9.2 in the main text.  For completeness, we detail the simulation set-up here.  In each iteration $b = 1,...,B$, we draw $\{\mathbf{r}_i(1)\}_{i=1}^N$ and $\{\mathbf{r}_i(0)\}_{i=1}^N$ independent from one another and $iid$ from mean zero equicorrelated multivariate normals of dimension $k=25$ with marginal variances one. The correlation coefficients governing $\mathbf{r}_i(1)$ and $\mathbf{r}_i(0)$ are 0 and 0.95 respectively. For both our type I error and power simulations we will have two simulation settings, one with constant treatment effects and one with heterogeneous treatment effects:
\begin{itemize}
    \item[]\textit{Constant Effects}: $\mathbf{y}_i(1) - \bm{\tau} = \mathbf{y}_i(0) = \mathbf{r}_i(1)$.
    \item[]\textit{Heterogeneous Effects}: $\mathbf{y}_i(1) = \mathbf{r}_i(1)$; $\mathbf{y}_i(0) + \overline{\bm{\tau}} = \mathbf{r}_i(0) + \bar{\mathbf{r}}(1) - \bar{\mathbf{r}}(0)$.
\end{itemize}
The experimental design is that of a completely randomized experiment in which $n_1=0.2N$.  We test for treatment effect using randomization inference based upon the following three statistics:
\begin{enumerate}
    \item Hotelling's $T$-squared, unpooled covariance,
    \item Hotelling's $T$-squared, pooled covariance,
    \item Max absolute $t$-statistic, unpooled standard error.
\end{enumerate}

\subsection{Type I error rates}
When $\bm{\tau} = \mathbf{0}$ in the constant effects simulation setting in \S \ref{sec:gen}, Fisher's sharp null holds; likewise taking $\overline{\bm{\tau}} = \mathbf{0}$ under heterogeneous effects enforces Neyman's weak null.  In these contexts, we reexamine the Type I error rate simulations of the main paper, but at $\alpha = 0.25$ instead of $0.05$.  While this value of $\alpha$ is larger than typical in scientific practice, the larger value of $\alpha$ allows frequent rejections of the null hypothesis despite the inherent conservativeness of inference under the finite population model.  We stress that the conservativeness of these tests rests upon the finite population inference framework and not upon the mechanics of Gaussian prepivoting: the non-identifiability of $\Sigma_{\tau, \infty}$ forces conservativeness of \textit{any} procedure which asymptotically guarantees Type I error rate control under $H_{N}$.  We conduct simulations with $N=300$ and $N=5000$; for each $N$, we conduct $B=5000$ simulations.  These tests are conducted using Monte-carlo simulation to generate the reference distributions, with $1000$ draws from $\Omega_{CRE}$ for each iteration $b$.  Table~\ref{tab: error rates} presents Type I error rates under simulation with $\bm{\tau} = \overline{\bm{\tau}} = \mathbf{0}$. 

\begin{table}
    \caption{\label{tab: error rates} Type I error rates in completely randomized designs with multiple outcomes. The rows describe the simulation settings, which vary between the sharp and weak nulls holding and between small and large sample sizes. There are three sets of columns, one corresponding to each of the three test statistics under consideration. For each set of columns, the column labeled ``FRT'' represents the Fisher Randomization Test using that test statistic. The column labeled ``Pre.'' instead reflects the Fisher Randomization Test after applying Gaussian prepivoting to the original test statistic. The last column, labeled ``LS,'' is a large-sample test which is asymptotically valid for the weak null. The desired Type I error rate in all settings is  $\alpha=0.25$.  For all columns $\bm{\tau} = \overline{\bm{\tau}} = \mathbf{0}$.}
    \centering

        \begin{tabular}{l c c c c c c c c c}
            &\multicolumn{3}{c}{Hotelling, Unpooled} & \multicolumn{3}{c}{Hotelling, Pooled} & \multicolumn{3}{c}{Max $t$-stat}\\
            & FRT & Pre. & LS &  FRT & Pre. & LS & FRT & Pre. & LS\\
            Sharp, $N=300$  &   0.244 & 0.244 & 0.630 & 0.251 & 0.249 & 0.365 & 0.254 & 0.252 & 0.300 \\ 
            Sharp, $N=5000$ &   0.247 & 0.247 & 0.270 & 0.248 & 0.243 & 0.257 & 0.251 & 0.247 & 0.255 \\ 
            Weak, $N=300$   &   0.320 & 0.320 & 0.538 & 0.996 & 0.361 & 0.433 & 0.321 & 0.071 & 0.082 \\ 
            Weak, $N=5000$  &   0.049 & 0.049 & 0.056 & 0.990 & 0.064 & 0.067 & 0.308 & 0.060 & 0.064 \\ 
        \end{tabular}
\end{table}

As observed both in Table~\ref{tab: error rates} above and in Table~2 of the main text, the Fisher Randomization Test using the pooled Hotelling $T$-statistic demonstrates significant anti-conservativeness under $H_{N}$.  Moreover, we see from Table~\ref{tab: error rates} that at $\alpha = 0.25$ the Fisher Randomization Test using the max $t$-statistic is also anti-conservative under $H_{N}$, a problem that persists even when $N = 5000$.  As a demonstration of Theorem~\ref{supp: thm: gaussian}, the Fisher Randomization Tests using the prepivoted versions of $T_{pool}$ and $T_{|max|}$ control the Type I error rate under Neyman's null for large $N$.

\subsection{Power after prepivoting}
Below we provide a theoretical discussion of the power of the Fisher Randomization Test based on $G(\yobs, \bfZ)$ and use simulations to highlight key aspects of its statistical power in practice.

When prepivoting is not necessary because the test statistic being deployed is already sharp dominant and pivotal, its use does not affect the power of the test.  Suppose that the Fisher Randomization Test using the pivotal test statistic $T(\yobs, \bfZ)$ provides exact inferences under $H_{F}$ and asymptotically valid inferences under $H_{N}$.  Examples of such test statistics include the studentized absolute difference in means for $d = 1$ and its multivariate analogue $T_{\chi^{2}}(\yobs, \mathbf{Z})$; see \citet{RandTestsWeakNulls} for further examples.  Let $\mathcal{N}_{f_{\hat{\xi}}, \hat{V}}$ denote the $f_{\hat{\xi}}$-pushforward of the Gaussian measure $\gamma_{\mathbf{0}, \hat{V}}^{(d  +k)}$.  Since $f_{\hat{\xi}}$ takes values in $\R$ the pushforward measure $\mathcal{N}_{f_{\hat{\xi}}, \hat{V}}$ is a distribution on the real line, and so -- in a slight abuse of notation -- we write its corresponding cumulative distribution function evaluated at $t \in \R$ as $\mathcal{N}_{f_{\hat{\xi}}, \hat{V}}(t)$.  For a completely randomized experiment $G(\yobs, \bfZ) = \mathcal{N}_{f_{\hat{\xi}}, \hat{V}}(T(\yobs, \bfZ))$.  
If $\mathcal{N}_{f_{\hat{\xi}}, \hat{V}}$ is pivotal in the sense that its distribution does not depend upon unknown parameters requiring estimation, then $G(\yobs, \bfW)$ is a fixed continuous non-decreasing transformation of $T(\yobs, \bfW)$.  Consequently, for any fixed $\bfZ$ the pair $\left(G(\yobs, \mathbf{w}), T(\yobs, \mathbf{w}) \right)$ has rank correlation 1 when enumerated over $\mathbf{w}\in\Omega$ and so $p$-values derived under $\hatPerm_{T}$ exactly match those under $\hatPerm_{G}$.  In this case, prepivoting has no impact upon the power of the test: a test statistic $T(\yobs, \bfZ)$ with high power under the alternative will yield $G(\yobs, \bfZ)$ with high power as well.  An example of such a case is $T_{\chi^{2}}$.

However, as demonstrated in Section~5 of the main paper, there are cases for which $T(\yobs, \bfZ)$ cannot be used for randomization inference under $H_{N}$ because it is not asymptotically sharp dominant.  Examples of this include $T_{pool}(\yobs, \bfZ)$ and $T_{|max|}(\yobs, \bfZ)$.  Even in these cases, the asymptotic power of the Fisher Randomization Test using $G(\yobs, \bfZ)$ can be computed.  Regardless of pivotality, the test statistic $G(\yobs, \bfZ)$ itself is the complement of a $p$-value for an asymptotically valid test of $H_{N}$.  In other words, $1 - G(\yobs, \bfZ)$ can be used directly as a $p$-value for testing $H_{N}$ with asymptotic control of the Type I error rate (simply reject $H_{N}$ if $1 - G(\yobs, \bfZ) \leq \alpha$.)  The power of this large-sample test must be computed on a case-by-case basis since it is reliant on the structure of the underlying test statistic $T(\yobs, \bfZ)$.  Because the reference distribution employed by Gaussian prepivoting converges to a standard uniform even under the alternative,  the asymptotic power of the randomization test using $G(\yobs, \bfZ)$ converges to the power of this large-sample test of $H_{N}$, with the added benefit that the randomization test is exact for finite $N$ under $H_{F}$.  Randomization inference after Gaussian prepivoting leverages an asymptotically valid test for $H_{N}$ and furnishes exact inference for $H_{F}$ with no sacrifice in asymptotic power against $H_N$.  In other words, for test statistics $T(\yobs, \bfZ)$ satisfying Conditions~\ref{supp: cond: f} and \ref{supp: cond: xi} exactness under $H_{F}$ can be achieved for free, with limiting power remaining equal to that of the large-sample test upon which $G(\yobs, \bfZ)$ is based.

\begin{table}
        \caption{\label{tab: power} Power simulations in completely randomized designs with multiple outcomes. The rows describe the simulation settings, which vary between constant and heterogeneous effects and between small and large sample sizes. There are three sets of columns, one corresponding to each of the three test statistics under consideration. For each set of columns, the column labeled ``FRT'' represents the Fisher Randomization Test using that test statistic. The column labeled ``Pre.'' instead reflects the Fisher Randomization Test after applying Gaussian prepivoting to the original test statistic. The last column, labeled ``LS,'' is a large-sample test which is asymptotically valid for the weak null. The desired Type I error rate in all settings is  $\alpha=0.25$.  For all columns $\bm{\tau} = \overline{\bm{\tau}} = 0.05\mathbf{e}$ where $\mathbf{e}$ is the vector of all ones.}
        \centering
        \begin{tabular}{l c c c c c c c c c}
                    &\multicolumn{3}{c}{Hotelling, Unpooled} & \multicolumn{3}{c}{Hotelling, Pooled} & \multicolumn{3}{c}{Max $t$-stat}\\
            & FRT & Pre. & LS &  FRT & Pre. & LS & FRT & Pre. & LS\\
            Constant, $N=300$           & 0.378 & 0.378 & 0.748 & 0.389 & 0.384 & 0.521 & 0.339 & 0.335 & 0.393\\
            Constant, $N=5000$          & 1 & 1 & 1 & 1 & 1 & 1 & 0.969 & 0.968 &  0.971\\
            Heterogeneous, $N=300$      & 0.360 & 0.360 & 0.574 & 0.995 & 0.391 & 0.452 & 0.458 & 0.130 & 0.149\\
            Heterogeneous, $N=5000$     & 0.421 & 0.421 & 0.448 & 0.995 & 0.080 & 0.086 & 0.993 & 0.861 & 0.868
        \end{tabular}

\end{table}

Table~\ref{tab: power} presents power simulations under the set-up detailed above for $\bm{\tau} = \overline{\bm{\tau}} = 0.05\mathbf{e}$, where $\mathbf{e}$ denotes the vector of all ones.  Under constant effects, the power of all of the tests is high and the Type I error rate is controlled for $H_{F}$ because we are using Fisher Randomization Tests. Although the power of the Fisher Randomization Tests using $T_{pool}$ and $T_{|max|}$ is very high for heterogeneous effects, as observed in Table~\ref{tab: error rates} the randomization tests of $T_{pool}$ and $T_{|max|}$ do not control the Type I error rate for testing $H_{N}$ even asymptotically.  However, for the tests which do asymptotically control the Type I error rate under $H_{N}$ the randomization test of $G(\yobs, \bfZ)$ has power observed to be close to that of the large-sample test.  Furthermore, the gap in power between the two diminishes as $N$ increases. As stated above, this is because the critical value deployed by the randomization test of $G(\yobs, \bfZ)$ is converging to 0.75 as $N$ increases, while the large-sample test rejects for  $G(\yobs, \bfZ) \geq 0.75$. This further highlights the asymptotic equivalence between the two approaches.

The results for the prepivoted test based upon the usual (pooled) Hotelling test yield two interesting observations. First, the pooled test has markedly worse power than the unpooled or max-$t$ statistics, as the use of the pooled covariance matrix in forming the test statistics amounts to a choice of a suboptimal norm for constructing the test. Second, the power actually decreases for both the prepivoted randomization test and the large-sample test when going from $N=300$ to $N=5000$.  This can be attributed to the large-sample approximation being quite poor for the pooled test at $N=300$ in the setting under consideration. As shown in the simulations in the manuscript and in Table \ref{tab: error rates}, the Type I error rate exceeds the nominal level when the weak null is true at $N=300$, but falls below it at $N=5000$. As $N$ increases the large-sample approximation becomes better, hence restoring the conservativeness of the test under the null. This behavior also drives the apparent reduction in power in the above table. The power still tends to 1 for the pooled Hotelling test as $N\rightarrow\infty$ in the generative model yielding this simulation study.

\section{Gaussian integral formulation}
    In the main text we used the notation $\gamma^{(\ell)}_{\bm{\mu}, \Sigma}(\mathscr{B})$ to denote the  measure of a Borel-measurable set $\mathscr{B} \subseteq \R^{\ell}$ under Gaussian measure with mean $\bm{\mu}$ and covariance $\Sigma$.  Here we provide equivalent formulations of the example Gaussian prepivoted test statistics examined in Section~5 of the main text, but instead of using $\gamma^{(\ell)}_{\bm{\mu}, \Sigma}$ we directly write the corresponding Gaussian integrals.

    \begin{example}[Absolute difference in means]\label{supp example: difference in means for cre}
    		Let $\sqrt{N}\hat{{\tau}}$ be univariate, consider a completely randomized design with no rerandomization, and let $T_{DiM}(\yobs, \mathbf{Z}) = \sqrt{N}|\tauhat|$, such that $f_{\eta}(t) = |t|$ and $\hat{\xi} = 1$.  Gaussian prepivoting yields the test statistic
    		\begin{align*}
    		    G_{DiM}(\mathbf{y}(\mathbf{Z}), \mathbf{Z}) &= \gamma^{(1)}_{0, \hat{V}_{\tau\tau}}\{a: |a| \leq \sqrt{N}|\hat{{\tau}}|\} \\
    		    &= \frac{1}{\sqrt{2 \pi \hat{V}_{\tau\tau}}}\int_{-\sqrt{N}|\hat{{\tau}}|}^{\sqrt{N}|\hat{{\tau}}|}\exp\left( \frac{-a^{2}}{2 \hat{V}_{\tau\tau}}\right)\,da\\
    		    &=1 - 2\Phi\left(-\frac{\sqrt{N}|\hat{{\tau}}|}{\sqrt{\hat{V}_{\tau\tau}}}\right),
    		\end{align*}
    		where $\Phi(\cdot)$ is the standard normal distribution function.
    \end{example}

    \begin{example}[Multivariate studentization] \label{supp example: mult}
        Let $\sqrt{N}\hat{\bm{\tau}}$ now be multivariate and suppose we have a completely randomized design; consider the test statistic
    	\begin{align}\label{supp eqn: studentized test statistic}
    		T_{\chi^{2}}(\yobs, \mathbf{Z}) &= \left(\sqrt{N} \hat{\bm{\tau}}\right)^{\T} \hat{V}_{\tau\tau}^{-1} \left(\sqrt{N} \hat{\bm{\tau}}\right);\\
    		\hat{V}_{\tau\tau} &= \frac{N}{n_{1}}\hat{\Sigma}_{y(1)} + \frac{N}{n_{0}}\hat{\Sigma}_{y(0)}. \nonumber
    	\end{align}
    	For this test statistic, $f_\eta(\mathbf{t}) = \mathbf{t}^\T\eta^{-1}\mathbf{t}$ and $\hat{\xi} = \hat{V}_{\tau\tau}$.  Gaussian prepivoting produces
    	\begin{align}
    	    G_{\chi^2}(\mathbf{y}(\mathbf{Z}), \mathbf{Z}) &= \gamma^{(d)}_{\bm{0},\hat{V}_{\tau\tau}}\{\mathbf{a}: \mathbf{a}^\T\hat{V}_{\tau\tau}^{-1}\mathbf{a} \leq T_{\chi^2}(\yobs, \mathbf{Z})\}\nonumber \\
    	    &= \frac{1}{\sqrt{(2 \pi)^{d}\det(\hat{V}_{\tau\tau})}}\int_{\R^{d}}\indicatorFunction{\mathbf{a}^\T\hat{V}_{\tau\tau}^{-1}\mathbf{a} \leq T_{\chi^2}(\yobs, \mathbf{Z})}\exp\left(\frac{-\mathbf{a}^{\T}\hat{V}_{\tau\tau}^{-1}\mathbf{a}}{2} \right)\,d\mathbf{a}\label{eqn: gaussian integral ellipsoid}\\
    	    &= F_d\{T_{\chi^2}(\yobs, \mathbf{Z})\},\nonumber
    	\end{align}
    	where $F_d(\cdot)$ is the distribution function of a $\chi^2_d$ random variable.
    \end{example}

    \begin{example}[Max absolute $t$-statistic]
        Consider again multivariate $\sqrt{N}\hat{\bm{\tau}}$ in a completely randomized design and the test statistic
        \begin{align*}
            T_{|max|}(\mathbf{y}(\mathbf{Z}), \mathbf{Z}) = \max_{1\leq j \leq d}\frac{\sqrt{N}|\hat{{\tau}}_j|}{\sqrt{\hat{V}_{\tau\tau, jj}}},
        \end{align*}
        where $\hat{V}_{\tau\tau, jj}$ is the $jj$ element of $\hat{V}_{\tau\tau}$. For this statistic, $f_{\bm{\eta}}(\mathbf{t}) = \max_{1\leq j\leq d} |t_j|/\eta_j$, and $\hat{\bm{\xi}} = (\hat{V}^{1/2}_{\tau\tau, 11},...,\hat{V}^{1/2}_{\tau\tau, dd})^\T$. After Gaussian prepivoting we are left with
        \begin{align}
            G_{|max|}(\mathbf{y}(\mathbf{Z}), \mathbf{Z}) &=
            \gamma^{(d)}_{\bm{0},\hat{V}_{\tau\tau}}\left\{\mathbf{a}: \max_{1\leq j \leq d}\; \frac{|a_j|}{\sqrt{\hat{V}_{\tau \tau, jj}}} \leq  \max_{1\leq j \leq d}\; \frac{\sqrt{N}|\hat{{\tau}}_j|}{\sqrt{\hat{V}_{\tau\tau, jj}}}\right\}\nonumber\\
    	    &= \frac{1}{\sqrt{(2 \pi)^{d}\det(\hat{V}_{\tau\tau})}}\int_{\R^{d}}\indicatorFunction{\max_{1\leq j \leq d}\; \frac{|a_j|}{\sqrt{\hat{V}_{\tau \tau, jj}}}  \leq T_{|max|}(\mathbf{y}(\mathbf{Z}), \mathbf{Z})}\exp\left(\frac{-\mathbf{a}^{\T}\hat{V}_{\tau\tau}^{-1}\mathbf{a}}{2} \right)\,d\mathbf{a}.\label{eqn: gaussian integral hyperrectangle}
        \end{align}
        Importantly, \eqref{eqn: gaussian integral ellipsoid} and \eqref{eqn: gaussian integral hyperrectangle} differ only in the support of the Gaussian integral.  The same Gaussian measure is used; the difference is that the support of \eqref{eqn: gaussian integral ellipsoid} is an ellipsoid while the support of \eqref{eqn: gaussian integral hyperrectangle} is a hyperrectangle.
    \end{example}

    \begin{example}[Rerandomization] \label{supp example: rerandomization}
        Let $\sqrt{N}\hat{{\tau}}$ be univariate and suppose we now consider a rerandomized design with balance criterion $\phi$ satisfying Condition~1. Consider the absolute  difference in means, $f_{\hat{\xi}}(\sqrt{N}\hat{{\tau}}) = \sqrt{N}|\hat{{\tau}}|$, such that $\hat{\xi} = 1$. Gaussian prepivoting yields the test statistic
        \begin{align}
            G_{Re}(\mathbf{y}(\mathbf{Z}), \mathbf{Z}) &= \frac{\gamma^{(1+k)}_{\bm{0}, \hat{V}}\left\{(\mathbf{a}, \mathbf{b})^\T: |a| \leq \sqrt{N}|\hat{{\tau}}|\; \wedge\; \phi(\mathbf{b})=1\right \}}{\gamma^{(k)}_{\bm{0}, \hat{V}_{\delta \delta}}\left\{\mathbf{b}: \phi(\mathbf{b})=1\right\}}\nonumber\\
            &= \frac{\frac{1}{\sqrt{(2 \pi)^{(k + 1)}\det(\hat{V})}}\int_{\R^{k}}\left( \phi(\mathbf{b}) \int_{-\sqrt{N}|\hat{{\tau}}|}^{\sqrt{N}|\hat{{\tau}}|}\exp\left( \frac{-[a\, \mathbf{b}^{\T}]\hat{V}^{-1}[a\, \mathbf{b}^{\T}]^{\T}}{2 }\right)\,da\right)\,d\mathbf{b}}{\frac{1}{\sqrt{(2 \pi)^{k}\det(\hat{V}_{\delta\delta})}}\int_{\R^{k}}\phi(\mathbf{b}) \exp\left( \frac{-\mathbf{b}^{\T}\hat{V}^{-1}_{\delta \delta}\mathbf{b}}{2 }\right)\,d\mathbf{b}}.\label{eqn: gaussian integral for rerand}
        \end{align}
        Since $\phi(\cdot)$ is a boolean-valued function it directly constrains the support of the Gaussian integrals in \eqref{eqn: gaussian integral for rerand}.
    \end{example}

\section{Discussing Condition~2}
Recall Condition~2 from the main text:
\setcounter{condition}{1}
\begin{condition}\label{supp cond: f}
    For any $\eta \in \Xi$, $f_{\eta}(\cdot): \mathbb{R}^d\mapsto \mathbb{R}_+$ is continuous, quasi-convex, and nonnegative with $f_\eta(\mathbf{t}) = f_\eta(-\mathbf{t})$ for all $\mathbf{t}\in \mathbb{R}^d$. Furthermore, $f_{\eta}(\mathbf{t})$ is jointly continuous in $\eta$ and $\mathbf{t}$.
\end{condition}
Each condition on $f_{\eta}$ plays an important role in the underlying mechanics of Gaussian prepivoting, and each deserves some degree of attention.  First, the joint continuity of $f_{\eta}(\mathbf{t})$ in $\eta$ and $\mathbf{t}$ plays a critical role in the asymptotic behavior of the test statistic $T$.  When computing the asymptotic distributional behavior of $T(\yobs, \bfZ)$ we leverage the central limit theorem governing $\sqrt{N}\hat{\bm{\tau}}(\yobs, \bfZ)$ and the continuous mapping theorem to obtain the distributional limit of
$$
    f_{\hat{\xi}(\yobs, \bfZ)}\left(\sqrt{N}\hat{\bm{\tau}}(\yobs, \bfZ)\right).
$$
Without the joint continuity of $f_{\eta}(\mathbf{t})$, such a generic asymptotic result would not be feasible.  The same reasoning shows the utility of Condition~\ref{supp cond: f}'s joint continuity requirement when analyzing
$$
    f_{\hat{\xi}(\yobs, \bfW)}\left(\sqrt{N}\hat{\bm{\tau}}(\yobs, \bfW)\right).
$$
In this sense, the joint continuity assumption is of technical importance for deriving asymptotic distributional behavior.  The quasi-convexity and mirror symmetry assumptions are of a more fundamental nature to our results; they are inextricably linked to \citeauthor{conservativeCovarianceDominance}'s (\citeyear{conservativeCovarianceDominance}) theorem for multivariate Gaussians and so they play a crucial role in guaranteeing the asymptotic sharp dominance of $G(\yobs, \bfZ)$.  
A quasi-convex function is a function with convex sublevel sets; for those unfamiliar with quasi-convex functions, we suggest the excellent review of \citet[Chapter 3]{generalizedConcavity}.  A simple example function from $\R^{d} \rightarrow \R$ that is both quasi-convex and mirror symmetric about the origin is the Euclidean norm $\mathbf{t} \mapsto ||\mathbf{t}||$.  Generalizing slightly more, if $f_{\eta}(\mathbf{t})$ is any seminorm on $\R^{d}$ which is jointly continuous in $\eta$ and $\mathbf{t}$, then $f_{\eta}(\cdot)$ satisfies Condition~\ref{supp cond: f}.  In fact, this is nearly a complete characterization; we will show that the criteria of Condition~\ref{supp cond: f} stipulate that $f_{\eta}(\cdot)$ is tightly related to a seminorm; though $f_{\eta}(\cdot)$ need not be a seminorm itself.  Our discussion centers around the case of a completely randomized experiment, but this restriction is only for the sake of explication; similar reasoning applies in the rerandomized case as well.

Consider a convex set $\mathcal{U} \subset \R^{d}$
; suppose that $\mathcal{U}$ is \textit{balanced} in the sense that $c\mathcal{U} \subseteq \mathcal{U}$ for all scalars $u \in [-1, 1]$.  Such a set $\mathcal{U}$ is necessarily mirror-symmetric about the origin, and so Anderson's theorem states that:
\begin{quote}
    If $\mathcal{X} \sim \Normal{\mathbf{0}}{S_{\mathcal{X}}}$ and $\mathcal{Y} \sim \Normal{\mathbf{0}}{S_{\mathcal{Y}}}$ are non-degenerate with $S_{\mathcal{Y}} - S_{\mathcal{X}} \succeq 0$, then $\Prob{\mathcal{X} \in \mathcal{U}} \geq \Prob{\mathcal{Y} \in \mathcal{U}}$.
\end{quote}
Moreover, such a set $\mathcal{U}$ defines a seminorm on $\R^{d}$ via its Minkowski functional $\rho_{\mathcal{U}}(\mathbf{t}) = \inf_{k > 0}\{\mathbf{t} \in k\mathcal{U}\}$.
In light of this, Anderson's theorem can be rewritten as:
\begin{quote}
    If $\mathcal{X} \sim \Normal{\mathbf{0}}{S_{\mathcal{X}}}$ and $\mathcal{Y} \sim \Normal{\mathbf{0}}{S_{\mathcal{Y}}}$ are non-degenerate with $S_{\mathcal{Y}} - S_{\mathcal{X}} \succeq 0$, then $\Prob{\rho_{\mathcal{U}}(\mathcal{X}) \leq 1} \geq \Prob{\rho_{\mathcal{U}}(\mathcal{Y}) \leq 1}$.
\end{quote}

Denote the preimage of a set $S$ under $f_{\eta}$ by $f_{\eta}^{-1}(S)$. By quasi-convexity and symmetry of $f_{\eta}$, the set $f_{\hat{\xi}}^{-1}\left([-\infty, T] \right)$ is convex and symmetric about the origin for any $T \in \R$.  Specifically, taking $\mathcal{U} = f_{\hat{\xi}}^{-1}\left([-\infty, T(\yobs, \bfZ)] \right)$ yields a random seminorm $\rho_{\mathcal{U}}$.  Finally, taking $S_{\mathcal{Y}} = \bar{\bar{V}}$ and $S_{\mathcal{X}} = V$ gives exactly that the randomization distribution of the Gaussian prepivoted test statistic $G(\yobs, \bfZ)$ is asymptotically dominated by the uniform distribution.

In other words, Anderson's theorem can be rephrased to say that when $S_{\mathcal{Y}} - S_{\mathcal{X}} \succeq 0$ the random variable $\mathcal{X}$ is more concentrated in any semi-norm than $\mathcal{Y}$; Condition~\ref{supp cond: f} is designed exactly so that the random set $\mathcal{U} = f_{\hat{\xi}}^{-1}\left([-\infty, T(\yobs, \bfZ)] \right)$ generates a seminorm via its Minkowski functional.

\section{Details of examples}
In the main text, we provide several examples of test statistics which are amenable to Gaussian prepivoting.  Here we provide details to verify the conditions of Theorem~1 for these examples.

Define the standard Neyman covariance estimator
\begin{equation*}
    \hat{V}(\bfy(\bfz), \bfw) = \begin{bmatrix}
        \hat{V}_{\tau\tau}(\bfy(\bfz), \bfw) & \hat{V}_{\tau\delta}(\bfy(\bfz), \bfw)\\
        \hat{V}_{\tau\delta}(\bfy(\bfz), \bfw)^{\T} & \hat{V}_{\delta\delta}(\bfy(\bfz), \bfw)
    \end{bmatrix}
\end{equation*}
where
\begin{align*}
    \hat{V}_{\tau\tau}(\bfy(\bfz), \bfw) &= N\left(\frac{\hat{\Sigma}_{y(1)}(\bfy(\bfz), \bfw)}{n_1} + \frac{\hat{\Sigma}_{y(0)}(\bfy(\bfz), \bfw)}{n_0}\right),\\
   \hat{\Sigma}_{y(1)}(\bfy(\bfz), \bfw) &= \frac{1}{n_{1} - 1}\sum_{i \,:\, w_{i} = 1}\left(\bfy_{i}(z_{i}) - \frac{1}{n_{1}}\sum_{j \,:\, w_{j} = 1}\bfy_{j}(z_{j}) \right)\left(\bfy_{i}(z_{i}) - \frac{1}{n_{1}}\sum_{j \,:\, w_{j} = 1}\bfy_{j}(z_{j}) \right)^{\T},\\
   \hat{\Sigma}_{y(0)}(\bfy(\bfz), \bfw) &= \frac{1}{n_{0} - 1}\sum_{i \,:\, w_{i} = 0}\left(\bfy_{i}(z_{i}) - \frac{1}{n_{0}}\sum_{j \,:\, w_{j} = 0}\bfy_{j}(z_{j}) \right)\left(\bfy_{i}(z_{i}) - \frac{1}{n_{0}}\sum_{j \,:\, w_{j} = 0}\bfy_{j}(z_{j}) \right)^{\T}.
\end{align*}
and the other blocks are defined analogously.
\begin{lemma}\label{lem: Neyman covariance estimator satisfies conditions}
    The Neyman covariance estimator satisfies Condition~4 of the main text.
\end{lemma}
\begin{proof}
    The limiting conservativeness of $\hat{V}(\bfy(\bfZ), \bfZ)$ rests upon the conservativeness of $\hat{V}_{\tau\tau}(\bfy(\bfZ), \bfZ)$, a well known fact dating back to \citeauthor{NeymanVariance} himself in the scalar case.  The vector version of this result is noted in \citet[Section 2.2]{decompTreatmentEffectVar} and relies upon the consistency of the sample covariance estimators $\hat{\Sigma}_{y(0)}(\bfy(\bfZ), \bfZ)$ and $\hat{\Sigma}_{y(1)}(\bfy(\bfZ), \bfZ)$.  The consistency of $\hat{\Sigma}_{y(0)}(\bfy(\bfZ), \bfZ)$ and $\hat{\Sigma}_{y(1)}(\bfy(\bfZ), \bfZ)$ (and their related quantities for the other blocks) is a consequence of Assumptions 1-3 and \citet[Appendix Lemma 1]{lin13}.

    Verifying the second part of Condition~4 requires examining the limiting behavior of $\hat{V}(\bfy(\bfZ), \bfW)$.  Such an analysis can be found in \citet[Appendix A2]{RandTestsWeakNulls}; while their work focuses on the scalar-outcome many-treatment case, the techniques convert straightforwardly to the vector-outcome treated-versus-control case.
\end{proof}

Lemma~\ref{lem: Neyman covariance estimator satisfies conditions} establishes that the covariance estimators used for prepivoting in Section~5 of the main text are indeed in accordance with Condition~4.  Next, we examine each example provided in the main text to establish that Conditions~2 and 3 are met.

\begin{example}[Absolute Difference in Means]
    In a completely randomized experiment, we consider $T_{DiM}(\yobs, \mathbf{Z}) = \sqrt{N}|\tauhat|$, with $f_{\eta}(t) = |t|$ and $\hat{\xi} = 1$.  Condition~3 is trivially satisfied since $\hat{\xi}$ is not stochastic.  Condition~2 follows from the continuity, convexity, non-negativity, and symmetry of the absolute value function.
\end{example}

\begin{example}[Multivariate studentization]
    Let $\sqrt{N}\hat{\bm{\tau}}$ now be multivariate and suppose we have a completely randomized design. We examine the statistic
	\begin{align}
		T_{\chi^{2}}(\yobs, \mathbf{Z}) &= \left(\sqrt{N} \hat{\bm{\tau}}\right)^{\T} \hat{V}_{\tau\tau}^{-1} \left(\sqrt{N} \hat{\bm{\tau}}\right),
	\end{align}
	with $\hat{V}_{\tau\tau} = \frac{N}{n_{1}}\hat{\Sigma}_{y(1)} + \frac{N}{n_{0}}\hat{\Sigma}_{y(0)}$.  For this test statistic, $f_\eta(\mathbf{t}) = \mathbf{t}^\T\eta^{-1}\mathbf{t}$ and $\hat{\xi} = \hat{V}_{\tau\tau}$.  Since $\hat{\xi}$ matches the top-left block of the Neyman covariance estimator Lemma~\ref{lem: Neyman covariance estimator satisfies conditions} shows that Condition~3 is met.  Condition~2 holds because the quadratic form $f_\eta(\mathbf{t}) = \mathbf{t}^\T\eta^{-1}\mathbf{t}$ is certainly mirror symmetric, jointly continuous and convex by standard results for quadratic forms, and is non-negative since $\eta$ is positive definite, and so its inverse must be as well.

	The analysis for $T_{pool}$ follows similar logic, but with the added observation that
    \begin{align*}
        \hat{V}_{Pool}(\yobs, \bfZ) &\convP[] \frac{\Sigma_{y(0),\infty}}{p} + \frac{\Sigma_{y(1),\infty}}{1 - p}\\
        \hat{V}_{Pool}(\yobs, \bfW) &\convP[] \frac{\Sigma_{y(0),\infty}}{p} + \frac{\Sigma_{y(1),\infty}}{1 - p}.
    \end{align*}
    This follows because, under our assumptions, $\hat{\Sigma}_{y(1)}(\yobs, \bfZ) \convP[] \Sigma_{y(1), \infty}$ and  $\hat{\Sigma}_{y(0)}(\yobs, \bfZ) \convP[] \Sigma_{y(0), \infty}$ while both $\hat{\Sigma}_{y(1)}(\yobs, \bfW)$ and $\hat{\Sigma}_{y(0)}(\yobs, \bfW)$ converge in probability to $p\Sigma_{y(1), \infty} + (1-p)\Sigma_{y(0), \infty}$.
\end{example}

\begin{example}[Max absolute $t$-statistic]
    Consider again multivariate $\sqrt{N}\hat{\bm{\tau}}$ and a completely randomized design.  The max-absolute $t$-statistic is
    \begin{align*}
        T_{|max|}(\mathbf{y}(\mathbf{Z}), \mathbf{Z}) = \max_{1\leq j \leq d}\frac{\sqrt{N}|\hat{{\tau}}_j|}{\sqrt{\hat{V}_{\tau\tau, jj}}},
    \end{align*}
    where $\hat{V}_{\tau\tau, jj}$ is the $jj$ element of $\hat{V}_{\tau\tau}$. For this statistic, $f_{\bm{\eta}}(\mathbf{t}) = \max_{1\leq j\leq d} |t_j|/\eta_j$, and $\hat{\bm{\xi}} = (\hat{V}^{1/2}_{\tau\tau, 11},...,\hat{V}^{1/2}_{\tau\tau, dd})^\T$.  Since $\hat{\bm{\xi}}$ is the square-root of the diagonal elements of $\hat{V}_{\tau\tau}$, Lemma~\ref{lem: Neyman covariance estimator satisfies conditions} again establishes Condition~3.  Certainly each coordinate projection $|t_j|/\eta_j$ is jointly continuous in $\eta_{j} > 0$ and $t_{j}$.  Taking the maximum over these functions preserves continuity.  The maximum of linear functions is convex so $f_{\bm{\eta}}(\mathbf{t})$ is quasi-convex.  Non-negativity and mirror symmetry are trivial algebraic properties inherited from the coordinate-wise absolute value function.  Thus, Condition~2 holds.
\end{example}

\begin{example}[Rerandomization]
    Consider a rerandomized design with balance criterion $\phi$ satisfying Condition~1 and let $\sqrt{N}\hat{{\tau}}$ be univariate. Consider the absolute  difference in means, $f_{\hat{\xi}}(\sqrt{N}\hat{{\tau}}) = \sqrt{N}|\hat{{\tau}}|$, such that $\hat{\xi} = 1$.  As before, since $\hat{\xi}$ is non-stochastic Condition~3 is immediate.  The continuity, quasi-convexity, non-negativity, and symmetry of $f_{\eta}$ are immediate consequences of the properties of the absolute value function.  Thus, Condition~2 holds.
\end{example}

\section{A case study with educational data}\label{supp: sec: ALO}
We demonstrate inference using Gaussian prepivoting in a completely randomized experiment.  \citet{ALO} implemented a moderate-scale completely randomized experiment to test the effectiveness of several strategies intended to boost academic performance.  Their experiment, the so-called \textit{Student Achievement and Retention} (STAR) project, enrolled incoming first-year undergraduate students -- except those with high-school grade point average (GPA) in the top 25\% -- in one of three treatment arms: a student support program, a financial incentive program, or both.  Allocation to the programs was performed completely at random.  Numerous demographic features of program participants were collected; we focus specifically on the participants' reported genders and high-school GPAs.  The primary outcomes of the study were first-year GPA and second-year GPA.  Further details on the nature of the interventions and the specific demographic features collected can be found in \citet{ALO}.  The data collected in the STAR project is publicly available in the online supplement to \citet{ALO}.

\citet{ALO} found no evidence to suggest that the program was effective at improving educational outcomes among participants who identified as men.  \citeauthor{lin13} used regression-adjusted estimators to examine inference for the marginal effect of offering financial incentives given that support services were offered; his analysis focuses on only the male participants of the study.  \citeauthor{lin13} performs several simulations under the assumption that Fisher's sharp null holds, but he remarks that
\begin{quote}
    Chung and Romano (2011a, 2011b) discuss and extend a literature on permutation tests that do remain valid asymptotically when the null hypothesis is weakened. One such test is based on the permutation distribution of a heteroskedasticity-robust $t$-statistic. Exploration of this approach under the Neyman model (with and without covariate adjustment) would be valuable.
\end{quote}
Gaussian prepivoting allows us to meet and exceed this objective: a permutation-testing framework can be applied with asymptotic validity under $H_{N}$ for a wide class of statistics -- including multivariate statistics for which studentization alone is insufficient to restore asymptotic conservativeness. Theorem~\ref{supp: thm: gaussian} guarantees finite sample exactness under $H_{F}$ and asymptotic conservativeness under $H_{N}$ of the prepivoted test statistic $G(\yobs, \bfZ)$ subject to mild conditions; moreover, Section~\ref{sec: reg adj} of the supplement extends this result to the context of regression-adjusted estimators.

We re-analyze the data studied by \citet{ALO} through the lens of Gaussian prepivoting.  Instead of restricting to the univariate outcome of first-year GPA, we examine the effect of treatment on both first-year and second-year GPA.  We implement prepivoting using the test statistics of Section~5 of the main text:
\begin{itemize}
    \item the Euclidean 2-norm of the difference in means, denoted $T_{||\cdot||_{2}}$,
    \item the multivariate studentized statistics $T_{\chi^{2}}$ and $T_{Pool}$,
    \item the maximum absolute $t$-statistic $T_{|max|}$.
\end{itemize}

Furthermore, we implement prepivoting in the cases above using the regression adjusted estimator of the difference in means -- regressing on high-school GPA -- instead of the na\"{i}ve difference in means.  In total $N = 141$ male-identifying participants have complete covariate and outcome data (high-school GPA, first and second year GPAs, respectively) and were offered \textit{at least} support services.  Of these individuals, $n_{1} = 55$ were offered both support services and financial incentives while $n_{0} = 86$ received only the offer for support services.

\begin{table}
        \caption{\label{tab: ALO data pvalues} $p$-values of the Fisher Randomization Test with and without using Gaussian prepivoting. The left two numerical columns use the Fisher Randomization Test directly on the base statistic without prepivoting.  The right two numerical columns apply Gaussian prepivoting to the base statistic before using the Fisher Randomization Test.  In ``With Adjustment" columns linear regression adjustment using high-school GPA was applied to estimate the difference in means; ``Without Adjustment" columns perform no regression adjustment.} 
        \centering
    \begin{tabular}{ccccc}
               & \multicolumn{2}{c}{No Prepivoting}       & \multicolumn{2}{c}{Prepivoting}          \\
        Base Statistic & Without Adjustment & With Adjustment & Without Adjustment & With Adjustment \\
       $T_{||\cdot||_{2}}$  & 0.140 & 0.095 & 0.154 & 0.104   \\
       $T_{\chi^{2}}$       & 0.159 & 0.126 & 0.159 & 0.126   \\
       $T_{Pool} $          & 0.141 & 0.107 & 0.153 & 0.121   \\
       $T_{|max|}$          & 0.181 & 0.129 & 0.174 & 0.122   \\
   \end{tabular}
\end{table}

Table~\ref{tab: ALO data pvalues} contains $p$-values of the Fisher Randomization Test before and after prepivoting.  The $p$-values obtained after prepivoting provide exact inference for $H_{F}$; moreover, the asymptotic results of Theorems~1 and 2 of the main text suggest that these $p$-values are likely to provide conservative inference for $H_{N}$. We stress that without Gaussian prepivoting only for $T_{\chi^{2}}$ would the Fisher Randomization Test be appropriate for Neymanian inference. The other three base statistics of Table~\ref{tab: ALO data pvalues} can exhibit asymptotically anti-conservative inference with the Fisher Randomization Test under $H_{N}$.  For both $T_{||\cdot||_{2}}$ and $T_{Pool} $ -- with and without regression adjustment -- the $p$-value obtained after prepivoting is no less than than the $p$-value derived without first prepivoting. These increased $p$-values suggest that the non-prepivoted procedures for $T_{||\cdot||_{2}}$ and $T_{Pool} $ may have been anti-conservative.  In fact, with $\alpha = 0.1$ an experimenter erroneously using the Fisher Randomization Test based upon regression adjustment with $T_{||\cdot||_{2}}$ to test $H_{N}$ would have rejected the null of no average effect.  Once prepivoting is applied the practitioner is asymptotically entitled to test $H_{N}$ with the Fisher Randomization Test and we observe that the procedure no longer rejects $H_{N}$, thereby rectifying the potentially anti-conservative nature of the preceding result.

The code to implement our analysis is provided online to facilitate reproducibility.

\color{black}

\section{Software}
Code written in \texttt{R} that builds the figures of the paper is available online at \url{https://github.com/PeterLCohen/PrepivotingCode}.  
Furthermore, at the same location, we provide concrete examples -- also written in \texttt{R} code -- illustrating how one might choose to implement Gaussian prepivoting from scratch.  For simplicity, we present prepivoting the absolute $\sqrt{N}$-scaled difference in means for a univariate completely randomized design.  In other words, we exactly demonstrate the implementation of Algorithm~1 for Gaussian prepivoting used in Example~1 of Section~5 in the main paper.  At the same location, we provide \texttt{R} code to reproduce the results of the data analysis in Section~\ref{supp: sec: ALO}.

\bibliographystyle{apalike}
\bibliography{bibliography}

\end{document}